\documentclass[11pt,twoside]{article}

\usepackage{fullpage}

\usepackage{amsthm, amsmath, amssymb, hyperref, dsfont, mathtools,bm}
\usepackage{framed}
\usepackage{caption,subcaption,graphicx,rotating,multirow}
\usepackage{url}
\usepackage{color,xcolor}
\usepackage{enumitem}
\usepackage{epstopdf}

\theoremstyle{plain}
\newtheorem{theorem}{Theorem}

\newtheorem{lemma}[theorem]{Lemma}

\newtheorem{corollary}[theorem]{Corollary}

\newtheorem{assumption}{Assumption}
\newtheorem{proposition}[theorem]{Proposition}

\newcommand{\Z}{\mathbb{Z}}
\newcommand{\R}{\mathbb{R}}

\renewcommand{\P}{\operatorname{\mathbb{P}}}
\newcommand{\E}{\operatorname{\mathbb{E}}}
\newcommand{\bigo}{\ensuremath{\mathcal{O}}}
\newcommand\smallo{
  \mathchoice
    {{\scriptstyle\mathcal{O}}}
    {{\scriptstyle\mathcal{O}}}
    {{\scriptscriptstyle\mathcal{O}}}
    {\scalebox{.7}{$\scriptscriptstyle\mathcal{O}$}}
  }

\newcommand{\indi}{\mathds{1}}
\newcommand{\half}{\frac{1}{2}}
\newcommand{\scL}{L^{\circlearrowleft}} 
\newcommand{\scp}{p^{\circlearrowleft}} 
\newcommand{\scdelta}{\delta^{\circlearrowleft}} 
\newcommand{\fsc}{f^{\textup{\texttt{sc}}}}
\newcommand{\hsc}{h^{\textup{\texttt{sc}}}}
\newcommand{\op}{\textup{op}}

\DeclareMathOperator{\var}{var}

\DeclareMathOperator{\trace}{trace}

\DeclareMathOperator{\sign}{sign}

\newcommand{\bbB}{\mathbb{B}}
\newcommand{\bbL}{\mathbb{L}}
\newcommand{\bbJ}{\mathbb{J}}

\newcommand{\rmd}{\mathrm{d}}

\def\se{\mbox{\rm\tiny e}}
\def\sv{\mbox{\rm\tiny v}}
\def\sperc{\mbox{\rm\tiny perc}}
\def\sBayes{\mbox{\rm\tiny Bayes}}
\def\sTV{\mbox{\rm\tiny TV}}
\def\sSK{\mbox{\rm\tiny SK}}

\newcommand{\Ent}{\mathds{H}}

\def\risk{\mathcal{R}}
\def\bfone{\mathbf{1}}

\allowdisplaybreaks
\numberwithin{equation}{section}
\numberwithin{theorem}{section}

\begin{document}

\title{Efficient $\Z_2$ synchronization on $\Z^d$ under symmetry-preserving side information}

\author{
Ahmed El Alaoui \thanks{Department of Statistics and Data Science, Cornell university.}
}

\date{}
\maketitle

\vspace{-.2cm}

\begin{abstract}
We consider $\Z_2$-synchronization on the Euclidean lattice. Every vertex of $\Z^d$ is assigned an independent symmetric random sign $\theta_u$, and for every edge $(u,v)$ of the lattice, one observes the product $\theta_u\theta_v$ flipped independently with probability $p$. The task is to reconstruct products $\theta_u\theta_v$ for pairs of vertices $u$ and $v$ which are arbitrarily far apart. 
Abb\'e, Massouli\'e, Montanari, Sly and Srivastava (2018) showed that synchronization is possible if and only if $p$ is below a critical threshold $\tilde{p}_c(d)$, and efficiently so for $p$ small enough. 
We augment this synchronization setting with a model of side information preserving the sign symmetry of $\theta$, and propose an \emph{efficient} algorithm which synchronizes a randomly chosen pair of far away vertices on average, up to a differently defined critical threshold $p_c(d)$. We conjecture that $ p_c(d)=\tilde{p}_c(d)$ for all $d \ge 2$. Our strategy is to \emph{renormalize} the synchronization model in order to reduce the effective noise parameter, and then apply a variant of the multiscale algorithm of AMMSS. The success of the renormalization procedure is conditional on a plausible but unproved assumption about the regularity of the free energy of an Ising spin glass model on $\Z^d$.              
\end{abstract}
\section{Introduction and main results}
In this paper we study a \emph{group synchronization} problem. Formulated generally, this is the task of estimating a sequence of group--valued random variables up to a global shift from noisy measurements of their differences. This problem arises in several applications such as community detection, where individuals belong to well-defined communities and observations come in the form of pairwise affinities between individuals~\cite{karrer2011stochastic}, time synchronization in sensor networks where the goal is to adjust the local clock of each sensor given noisy clock differences~\cite{howard2010estimation}, computer vision and microscopy, where the task is to estimate the shape of a moving object from multiple noisy snapshots of it~\cite{singer2011three}.

 This problem has mostly been studied in settings where the set of compared pairs forms a ``unstructured" graph such as the complete graph or a random graph from the Erd\"os-R\'enyi model. The analysis and the algorithms proposed in such cases rely on the high degree of exchangeability of the variables involved in the problem. A prime example being $\Z_2$ synchronization on the complete graph, where the observations can be put together in a matrix having the form of a rank-one-perturbed Wigner matrix; a well studied object for which powerful tools from random matrix theory and mean-field methods originating in statistical physics are readily available. These methods will typically fail on more structured graphs having geometric features such as finite-dimensional lattices. 
Furthermore, studying signal recovery problems on structured graphs can potentially be relevant to applications such as the ones mentioned above where, due to resource constraints or simply the problem's nature, only a limited number of comparisons can be measured.

To our knowledge, Abb\'e, Massouli\'e, Montanari, Sly and Srivastava~\cite{abbe2017group} were the first to consider the synchronization problem on the Euclidean lattice $\Z^d$ for compact groups. They proved for $d\ge 3$ that synchronization is possible with an efficient algorithm if the noise is weak enough and impossible in the opposite regime where the noise is strong. (They also treat the case $d=2$ in their paper, however this case is special, as the structure of the group plays an important role in the results.) By monotonicity with respect to the value of the noise, this establishes the existence of a critical value where recovery is possible below it and impossible above. We call this threshold the \emph{information-theoretic} threshold, or limit, of synchronization. In this paper we ask the question
\begin{quote}                         
Is group synchronization on $\Z^d$ possible with an \emph{efficient} algorithm up the information-theoretic limit?  
 \end{quote}
 We answer this question in the affirmative in the case of the simplest non trivial group $\Z_2 = \{-1,+1\}$, for all dimensions $d \ge 2$, assuming one has access to a vanishing amount of `side information' providing extra help in the synchronization task. We should mention that contrary to `mean-field models' as discussed above, no `explicit' characterization of this information-theoretic threshold is known on $\Z^d$. Therefore whatever strategy aiming to answer the above question has to be content with the implicit definition of the threshold, as we do in this paper, and build an algorithm based on it.   
 
 We now set up the mathematical problem. 
Let $\theta = (\theta_u)_{u \in \Z^d}$ be i.i.d.\ symmetric $\pm 1$ random variables assigned to the vertices of $\Z^d$. The edges of the lattice $\mathbb{E}^d = \{(u,v) \in \Z^d \times \Z^d : |u-v| = 1\}$ ($|\cdot|$ is the $\ell_2$ distance in $\R^d$) 
are assigned random variables 
\begin{align}\label{eq:lattice_model}
Y^{\delta}_{uv} = 
\begin{cases}
+\theta_u\theta_v &\mbox{with probability } 1-p,\\
-\theta_u\theta_v &\mbox{with probability } p,
\end{cases}
\end{align}
independently for every $(u,v) \in \mathbb{E}^d$ conditional on $\theta$, where $p \in (0,\half)$ is a fixed parameter. We let $\delta = 1-2p$ so that $\E[Y_{uv}^{\delta}|\theta_u,\theta_v]=\E[Y_{uv}^{\delta}|\theta_u\theta_v]=\delta \theta_u\theta_v$.  

Let $\Lambda_n = [-n,n]^d \cap \Z^d$ be a finite box of side length $2n+1$  and let $E_n = \mathbb{E}^d\cap \Lambda_n^2$ be its set of edges. We will informally denote the graph $(\Lambda_n,E_n)$ by its vertex set $\Lambda_n$. 
Let us also denote the set of observations in the box $\Lambda_n$  by $Y_{\Lambda_n}^{\delta}=\{Y_{uv}^\delta \, :\, u,v\in \Lambda_n\}$. 

Due to sign symmetry, it is not possible to distinguish $\theta$ from $-\theta$ from the knowledge of the edge observations $Y_{\Lambda_n}^\delta$. The \emph{synchronization problem} is as follows: given $Y_{\Lambda_n}^\delta$ and possibly a small amount of side information as clarified below, we want to produce estimates $T_{uv}$ that approximate the product $\theta_u\theta_v$ with non-trivial error for `most' pairs of vertices $u,v\in \Lambda_n$. 

We assume that side information comes in the form of pairwise measurements according to the \emph{`spiked GOE'} model\footnote{The terminology is meant to evoke the random matrix model $\theta\theta^\top + W$ where $W$ is from the Gaussian Orthogonal Ensemble, although due to the finite interaction range $L$, this analogy is only superficial.} with signal-to-noise ratio (SNR) $\eta$ and interaction range $L$:
\begin{equation}\label{eq:GOE_side_info}
Y^{\eta,L}_{uv} = \sqrt{\frac{\eta}{L^d}}\theta_{u}\theta_{v} + Z_{uv},\qquad \forall u,v \in \Lambda_n ~\mbox{s.t.}~ |u-v|_{\infty} \le L,
\end{equation}
where $Z_{uv} \sim N(0,1)$ are independent for all unordered pairs $(u,v)$, $|\cdot|_{\infty}$ denotes the $\ell_\infty$ distance in $\R^d$. The range of the interaction $L$ is a large but fixed constant independent of $n$, and $\eta$ is an arbitrarily small constant. 

We denote by $Y_{\Lambda_n}^{\delta,\eta,L} = \{Y_{\Lambda_n}^{\delta},Y_{\Lambda_n}^{\eta,L}\}$  the union of the measurements~\eqref{eq:lattice_model} and~\eqref{eq:GOE_side_info}. We are interested in maximizing the criterion 
\begin{equation}\label{eq:risk}
\risk_{\Lambda_n}(T) := \frac{1}{|\Lambda_n|^2} \sum_{u,v\in \Lambda_n} \E\Big[T_{uv}(Y_{\Lambda_n}^{\delta,\eta,L})\theta_u\theta_v\Big],
\end{equation}
which measures the average correlation between the estimate $T_{uv}$ and $\theta_u\theta_v$, among all estimators $T = (T_{uv})_{u,v\in \Lambda_n}$ where $T_{uv}:\{-1,+1\}^{E_n} \times \R^{\Lambda_n^2} \to \{-1,+1\}$ are measurable functions.

The presence of side information~\eqref{eq:GOE_side_info} is a technical device enabling the analysis of the posterior distribution of $\theta_{\Lambda_n} = (\theta_u)_{u \in \Lambda_n}$. Our results will be stated in the regime $L \to \infty$ and $\eta \to 0$  (after $n \to \infty$). The use of this device is well established in information theory, spin glass theory, and models of high-dimensional statistical inference~\cite{MontanariSparse,panchenko2013sherrington,lelarge2017fundamental}, although a different model of side information, or perturbation, is typically used in the literature. The specific form of the side information~\eqref{eq:GOE_side_info} does not break the sign symmetry of $\theta_{\Lambda_n}$ and as we explain in Section~\ref{sec:related} the problem of synchronization is still non trivial.

Before considering the algorithmic question, we first define the information-theoretic limits of synchronization. For the purpose of establishing some key definitions, we consider the situation where side information of the form~\eqref{eq:GOE_side_info}  is available for every pair $u,v \in \Lambda_n$; this is equivalent to $L=L_n=2n+1$. 
We now define the average pair correlations      
\begin{equation}\label{eq:pair_correlation}
\varphi_{\Lambda_n}^{\se} := \frac{1}{|\Lambda_n|^2} \sum_{u,v\in \Lambda_n}  \E \Big[\E\big[\theta_u\theta_v | Y_{\Lambda_n}^{\delta,\eta,L_n}\big]^2\Big].
\end{equation}
Note that $\varphi_{\Lambda_n}^{\se} = \risk_{\Lambda_n}(T^{\sBayes})$ where $T^{\sBayes}_{uv} = \E\big[\theta_u\theta_v | Y_{\Lambda_n}^{\delta,\eta,L_n}\big]$ is the Bayes-optimal estimator.  

\begin{proposition}\label{prop:conv_pair_corr}
The sequence $(\varphi_{\Lambda_n}^{\se})_{n\ge 1}$ has a limit $q_{\star}^2(\delta,\eta)$ for all $\delta\in [0,1]$ and all except countably many $\eta\ge 0$. Moreover, the maps $\delta \mapsto q_{\star}(\delta,\eta)$ and $\eta \mapsto q_{\star}(\delta,\eta)$ are non-decreasing.   
\end{proposition}

The proof of Proposition~\ref{prop:conv_pair_corr} can be found in Section~\ref{sec:GOEside}. Let us denote by  $\mathcal{D}_{\delta}$ the countable set of $\eta$ (for a given $\delta$) where $(\varphi_{\Lambda_n}^{\se})_n$ fails to converge, as per Proposition~\ref{prop:conv_pair_corr}. Then the following limit exists for all $\delta \in [0,1]$ and is non-decreasing in $\delta$:
 \begin{equation}\label{q_star_star}
 q_{\star\star}(\delta) := \lim_{\underset{\eta \notin \mathcal{D}_{\delta}}{\eta \to 0^+}} q_{\star}(\delta,\eta).
 \end{equation}
 We now define \emph{the synchronization threshold}
\begin{equation}\label{rec_threshold}
\delta_c = \inf\{\delta>0 ~:~ q_{\star\star}(\delta) >0\}.
\end{equation}

The next two results show that synchronization is efficiently possible above $\delta_c$ and impossible below (efficiently or otherwise). The synchronization result is conditional on an unproved but highly plausible assumption which we state in Section~\ref{sec:posteriors}, Assumption~\ref{assump:lipschitz}.  
\begin{theorem}\label{thm:reconstruction}
Conditional on Assumption~\ref{assump:lipschitz}, for all $d \ge 2$ and all $\delta > \delta_c$ there exists $\eta_0 = \eta_0(d,\delta)>0$ such that for almost all $\eta\le \eta_0$, there exists $L_0 = L_0(\delta,d,\eta)\ge 1$ and $n_0 =n_0(\delta,d,\eta)$ such that the following holds. For all $n$ there exists a randomized estimator $T^{(n)} = (T^{(n)}_{uv})_{u,v \in \Lambda_n}$ where $T^{(n)}_{uv} = T^{(n)}_{uv}(Y_{\Lambda_n}^{\delta,\eta,L})\in \{-1,+1\}$ with runtime $\bigo(n^{2d})$ such that for all $n \ge n_0$ and all $L \ge L_0$,
\[\risk_{\Lambda_n}\big(T^{(n)}\big) \ge \frac{9}{10}q_{\star\star}^2(\delta).\]  
\end{theorem}
The above result establishes the existence of an algorithm whose running time is quadratic in the size of the problem and which succeeds at synchronizing a randomly chosen pair of variables on average, immediately above the synchronization threshold $\delta_c$, when given an arbitrarily small amount of side information as per Eq.~\eqref{eq:GOE_side_info}. (Note that the interaction range $L$ in the side information does not depend on $n$.) 
 Moreover the synchronization accuracy is near optimal (the constant $\frac{9}{10}$ is arbitrary and can be replaced by any other constant $c_0<1$.)  
We next answer the natural question of whether synchronization is possible below $\delta_c$. We show that this is not the case in the following sense.
\begin{theorem}\label{thm:lower_bound}
Let $\delta<\delta_c$ and $L = L_n\ge 1$ be any sequence of integers depending on $n$. There exists a set $\mathcal{A} \subset \R_+$ of full Lebesgue measure such that the following holds. For any sequence of estimators $T^{(n)} = (T^{(n)}_{uv})_{u,v \in \Lambda_n}$ where $T^{(n)}_{uv} = T^{(n)}_{uv}(Y_{\Lambda_n}^{\delta,\eta,L_n})\in \{-1,+1\}$, we have
\[\lim_{\underset{\eta \in \mathcal{A}}{\eta \to 0^+}} \liminf_{n\to \infty} \big|\risk_{\Lambda_n}\big(T^{(n)}\big)\big|=0.\]
\end{theorem}

Theorem~\ref{thm:reconstruction} is proved in Section~\ref{sec:proof_main}, and Theorem~\ref{thm:lower_bound} is proved in Section~\ref{sec:proof_13}.

\section{Discussion and related work}
\label{sec:related}
As mentioned earlier, there is a wealth of results for reconstruction problems on the complete graph, sparse and dense regular and Erd\"os-R\'enyi random graphs and trees. See for instance~\cite{singer2011angular,singer2012vector,boumal2016nonconvex,bandeira2013cheeger,bandeira2017tightness,chen2018projected,abbe2017community} and references therein for work using various methods for group synchronization and its applications. However, the study of reconstruction on structured graph models with a lesser degree of exchangeability is fairly limited. A few pointers to the latter category are~\cite{globerson2015hard,abbe2017group,sankararaman2018community,abbe2018graph,abbe2020information,polyanskiy2020application}.
The former class of models, in addition to being analytically tractable, is also known to exhibit an information-computation gap where in a certain regime of parameters, reconstruction (in our case, synchronization) is information-theoretically possible but all known efficient (polynomial-time) algorithms fail at it. 

\smallskip\noindent\textbf{Local algorithm with BEC.} In contrast, Montanari and the author showed in~\cite{alaoui2019computational} that on amenable graphs, and under a stronger model of side information where one receives the value of $\theta_u$ with small probability independently for each vertex $u$---this is called the binary erasure channel (BEC)---optimal reconstruction is possible in a certain asymptotic sense by means of an efficient local algorithm. Therefore, reconstruction on amenable graphs with BEC side information does not exhibit a statistical-computational gap. The local algorithm uses the side information from the BEC in a crucial way to induce, asymptotically, a decoupling among the variables $\theta_u$. Amenability then implies that observations beyond a receding boundary around a given vertex $u$ carries vanishing information about its spin $\theta_u$. Therefore, under BEC, it is possible to estimate each $\theta_u$ by only using information in a ball of bounded radius around $u$. (Synchronization is then trivial as one can return an estimate for $\theta_u\theta_v$ by multiplying the estimates corresponding to $\theta_u$ and $\theta_v$ respectively.) 
The present model of side information~\eqref{eq:GOE_side_info} is weaker than BEC. In particular the sign symmetry of the variables $\theta_u$ is not broken given the observations $Y^{\delta,\eta,L}_{\Lambda_n}$, so one cannot estimate individual spins better than random and the strategy employed in~\cite{alaoui2019computational} fails. One can instead focus on the task of synchronizing two faraway vertices. 

\smallskip\noindent\textbf{Synchronization via a multiscale algorithm.} Our algorithm is based on a multiscale procedure similar to the one used in~\cite{abbe2017group}. Our main innovation concerns the base layer of the multiscale hierarchy: we partition the lattice into overlapping patches of constant size (see Figure~\ref{fig:block_core}) and construct a \emph{renormalized} instance of the synchronization problem where patches play the role of vertices, i.e., we assign a global spin to each patch and construct a new synchronization variables between every two neighboring patches. We exploit properties of the posterior measure of $\theta$ when augmented with the side information to show that the effective noise parameter for the renormalized instance can be made as small as one desires by making the patches large enough.      
We then adapt the multiscale algorithm of~\cite{abbe2017group} and apply it to this renormalized instance, therefore synchronizing any two vertices by traveling the path connecting them in tree encoding the multiscale hierarchy. The analysis of this second stage also requires significant effort as the renormalized instance is a more complicated graphical model than the original synchronization instance.       

\smallskip\noindent\textbf{The synchronization threshold.} The authors of~\cite{abbe2017group} consider a definition of synchronization slightly different from ours. The latter is considered possible if one can synchronize an \emph{arbitrary} pair of far away vertices, and the synchronization threshold is defined as
\[\tilde{\delta}_c := \inf \Big\{\delta >0 ~:~ \liminf_{|u-v| \to \infty} \, \sup_{T_{uv}} \, d_{\sTV} \Big(\text{Law}\big(T_{uv}(Y^{\delta}_{\Z^d}) \theta_u\theta_{v}\big) , \bar{\nu}\Big) >0 \Big\},\]  
where $Y^{\delta}_{\Z^d} = \{Y^{\delta}_{uv}: (u,v) \in \mathbb{E}^d\}$, $T_{uv} : \{-1,+1\}^{\mathbb{E}^d} \to \{-1,+1\}$ measurable, $\bar{\nu}$ is the uniform distribution over $\{-1,+1\}$, and $d_{\sTV}$ is the total variation distance. 
This definition relies only on the lattice observations $Y^{\delta}_{uv}, (u,v) \in \E^d$ and not on any side information. We find it more convenient to work with our definition of $\delta_c$, Eq.~\eqref{rec_threshold}, which is more suitable to algorithmic reasoning. Nevertheless, in light of Theorem~\ref{thm:reconstruction} and Theorem~\ref{thm:lower_bound} we find it plausible to conjecture that the equality $\tilde{\delta}_c = \delta_c$ holds for all $d \ge 2$.

\smallskip\noindent\textbf{The Ising spin glass.}  
In the absence of measurements on the edges of $\Z^d$, i.e., $\delta=0$, the quantity $q_{\star}(0,\eta)$ in Proposition~\ref{prop:conv_pair_corr} has a well understood, rather explicit expression~\cite{lesieur2015phase,lelarge2017fundamental}. In particular it is known that $q_{\star}(0,\eta)=0$ for all $\eta \le 1$ (therefore $q_{\star\star}(0)=0$) and $q_{\star}(0,\eta)>0$ for $\eta >1$. 
When $\delta>0$, little is known about $q_{\star}(\delta,\eta)$. For instance, the pure lattice case $\eta=0$ is of interest to statistical physics, since the posterior measure of $\theta$ is equivalent to the Ising spin glass measure on $\Z^d$, on the so-called Nishimori line: A simple application of the Bayes rule reveals 
\[\P(\theta | Y^{\delta}_{\Lambda_n}) \propto e^{\beta \sum_{(u,v)\in E_n} Y^{\delta}_{uv}\theta_u\theta_v},~~~\beta=  \frac{1}{2}\log\big(\frac{1-p}{p}\big).\]
Let us write $Y^{\delta}_{uv} = Z_{uv}\theta_{0u}\theta_{0v}$ where $\P(Z_{uv}=+1)=1-p=1-\P(Z_{uv}=-1)$ as per Eq.~\eqref{eq:lattice_model} (the added subscript in $\theta_0$ is to distinguish it from a generic vector $\theta$ in the above formula). Let us also define $\sigma_u = \theta_u\theta_{0u}$ for all $u \in \Lambda_n$. Then the push-forward of $\P(\cdot | Y^{\delta}_{\Lambda_n})$ by $\theta \to \sigma$ is of the form
\[\mu_{\Lambda_n}(\sigma) \propto e^{\beta \sum_{(u,v)\in E_n} Z_{uv}\sigma_u\sigma_v}.\]
The relation $\beta=  \frac{1}{2}\log\big(\frac{1-p}{p}\big)$ between the inverse temperature and the bias of the disorder variables $Z_{uv}$ defines a curve in the plane $(p,\beta)$ called the Nishimori line; see also~\cite[Chapter 4]{NishimoriBook}. Furthermore, $\delta_c = 1-2p_c$ as defined in~\eqref{rec_threshold} corresponds to \emph{the tri-critical point} $(p_c,\beta_c)$ conjecturally separating three distinct phases of the model. 
A few predictions are available for its location in $d=2$: Nonrigorous analytical considerations based on the replica method and special symmetries of the model have lead the authors of~\cite{maillard2003symmetry} to put forward the conjecture that $p_c$ solves the equation $-p\log p - (1-p)\log (1-p) = \frac{\log 2}{2}$ with  $p<\half$, which leads to $p_c \simeq 0.110028$, i.e., $\delta_c^2 \simeq 0.608312$. However Ohzeki~\cite{ohzeki2009locations} suggests that this is only an approximation which can be further improved via a renormalization group analysis, and they advance the value $p_c \simeq 0.109168$ ($\delta_c^2 \simeq 0.610939$). On the other hand, Monte-Carlo simulations for $d=2$ made by Toldin, Pelissetto and Vicari~\cite{toldin2009strong} seem to agree with the theoretical predictions to a good extent and indicate that $p_c \simeq 0.10917$ ($\delta_c^2 \simeq 0.61099$.) 

On the rigorous mathematical side, explicit bounds on $\tilde{\delta}_c$ are known in the recent literature: It was shown in~\cite{abbe2017group} that $p_{\sperc}(d) \le \tilde{\delta}_c  < 1$, where $p_{\sperc}(d)$ be the critical threshold of Bernoulli bond percolation on $\Z^d$. Abb\'e and Boix~\cite{abbe2020information}, and Polyanskiy and Wu~\cite{polyanskiy2020application} independently introduced a more refined \emph{information--percolation} argument implying that synchronization is impossible if $\delta^2 \le p_{\sperc}(d)$, i.e., $\tilde{\delta}_c^2 \ge p_{\sperc}(d)$. The iquadratic mprovement comes from a comparison with a more accurate, potentially inhomogeneous, percolation model where an edge $(u,v)\in\E^d$ is open with probability $I_2(Y^{\delta}_{uv} ; \theta_u,\theta_v )$, $I_2$ being the mutual information in $\chi^2$ divergence. In the two-dimensional case $d=2$, this latter bound reads $\tilde{\delta}_c^2 \ge \frac{1}{2}$, to be compared with the above-mentioned predictions. 

At the opposite end of the spectrum, one would expect mean-field behavior when $d$ is large, and the critical point of the model should approach its analogue on the infinite regular tree of the same degree, which this given by the Kesten-Stigum threshold for reconstruction on the $2d$-regular tree with binary alphabet~\cite{evans2000broadcasting}. This leads to a natural conjecture: $\tilde{\delta}_c^2 \sim 1/(2d)$ as $d \to \infty$.

\paragraph{Organization.} We present the renormalization procedure, which is the first step of the multiscale algorithm in Section~\ref{sec:algorithm}. 
Section~\ref{sec:renormalization} is devoted to statements of the key theorems enabling the analysis. The rest of multiscale construction is presented and analyzed in Section~\ref{sec:multiscale}. Sections~\ref{sec:posteriors},~\ref{sec:proof_everything} and~\ref{sec:free_energy} are devoted to the technical results and proofs underlying the key theorems of  Section~\ref{sec:renormalization}. 

\paragraph{Notation.} We will frequently take limits of numerical quantities which are only defined almost everywhere or everywhere except on a countable number of points with respect to a real-valued parameter, see e.g., the definition of $q_{\star\star}$, Eq.~\eqref{q_star_star}. It will be implicitly understood and with no further precision that the limits are taken along sequences avoiding the exceptional problematic set. We use the asymptotic notation $\smallo_{L,\eta}(1)$ for a quantity tending to zero when $L \to \infty$ followed by $\eta \to 0$ (in this order).

\begin{figure}[t]
\centering
\includegraphics[width=.4\textwidth]{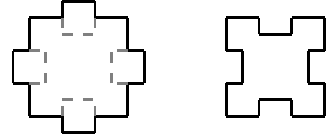}
\caption{Left: A block bound by black solid lines. The dashed inner boundaries define the surrounding square joints. Right: The core of a block, i.e., a block deprived of its joints.}
\label{fig:block_core}
\end{figure}

\section{The synchronization algorithm}
\label{sec:algorithm}

In this section we describe the synchronization algorithm which underlies Theorem~\ref{thm:reconstruction}. 
\paragraph{1.\ Partitioning.} We first cover the lattice $\Lambda_n$ with pairwise interlocking blocks:  
We consider a tiling of $\R^d$ with translates of the cube $C = [-\half,\half]^d$ along the coordinate directions. 
We put translates of the cube $\frac{1}{3}C$ at the midpoint of every edge of the integer lattice: these are all sets $\frac{1}{3}C + (k+\half) e_i$ where $k \in \Z$ and $(e_1,\cdots,e_d)$ is the standard unit basis of $\R^d$.
We call these translates \emph{joints}. Note that each joint overlaps with two adjacent elements of the tiling. 
We define a \emph{block} to be the union of an element of the tiling with its $2d$ surrounding joints; see Figure~\ref{fig:block_core} (left). Hence two adjacent blocks have a non empty intersection which is precisely the joint between them; see Figure~\ref{fig:partition}. We also define the \emph{core} of a block $A$ as $A^{\bullet} = A \setminus \cup J$ where the union is over all joints surrounding $A$; see Figure~\ref{fig:block_core} (right).
Now we dilate all distances by a factor $\scL$ and then intersect the dilated sets with $\Z^d$: 
\begin{equation*}
\bbB = \Big\{ (\scL\cdot A) \cap \Z^d  ~:~ A \mbox{ is a block} \Big\},\quad 
\bbJ = \Big\{ (\scL\cdot J) \cap \Z^d  ~:~ J \mbox{ is a joint} \Big\}.
\end{equation*}
The parameter $\scL\ge 1$ is an integer representing the scale of the renormalization. We retain the terminology of blocks, joints and cores for the elements of $\bbB$, the elements of $\bbJ$ and blocks deprived of their surrounding joints, respectively. 
We define the \emph{intersection graph} $\bbL = (\bbB,\bbJ)$ to have blocks as its vertices and joints as its edges; $\bbL$ is clearly isomorphic to the integer lattice. We write $B \sim B'$ for adjacency in $\bbL$.
Finally, we let  $\bbB_n = \{B \in \bbB : B \cap \Lambda_n \neq \emptyset\}$. For convenience we ignore the blocks $B$ that fall on the boundary of $\Lambda_n$, as their contribution to $\risk_{\Lambda_n}$, Eq.~\eqref{eq:risk}, is negligible.  

We  set the range of the GOE interaction~\eqref{eq:GOE_side_info} to be $L = 2\scL$. This insures that an observation $Y^{\eta,L}_{xy}$ is available for every two vertices $x,y \in B$ for every block $B \in \bbB_n$.

\begin{figure}[t]
\centering
\includegraphics[width=.3\textwidth]{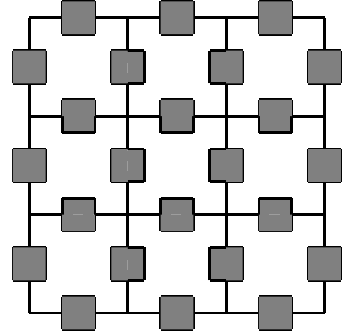}
\caption{Interlocking blocks in $d=2$. Cores are in white and bound by solid lines. Joints are in gray.}
\label{fig:partition}
\end{figure}

\paragraph{2.\ Processing side information.} Without loss of generality  we can assume we have access to an arbitrary finite number of independent copies of the GOE side information $Y_{\Lambda_n}^{\eta,L}$, Eq.\eqref{eq:GOE_side_info} with arbitrary SNR $\eta$. Indeed, given observations of the form~\eqref{eq:GOE_side_info} we can artificially create two independent copies with any SNR smaller than $\eta$ by adding and subtracting an independent normal r.v.\ with appropriate variance to/from $Y^{\eta,L}_{xy}$ for each pair $x,y$. Recursing this operation $\ell$ times, we get $2^\ell$ copies of~\eqref{eq:GOE_side_info} with any desired SNR (smaller than $\eta$).

Now we associate to the blocks $B \in \bbB_n$ independent (conditionally on $\theta$) sets of side information produced as follows. 
 Let $t \in [0,1]$ be chosen, for instance, uniformly at random. This parameter is shared across all blocks $B\in \bbB_n$. Now we fix a block $B$ and do the following: 
 For every $B' = B+\scL a$ where $a\in\{ \pm e_1,\cdots, \pm e_{d}\}$ (so that $B' \sim B$), let 
\begin{align*}
Y^{\eta, (a,\bullet)}_{uv} &=  \sqrt{\frac{t\eta}{|B|}}\theta_{u}\theta_{v} + Z^{(a,\bullet)}_{uv} ~\mbox{for all}~ x,y \in B,\\
Y^{\eta, (a,\cap)}_{uv} &= \sqrt{\frac{(1-t)\eta}{|B\cap B'|}}\theta_{u}\theta_{v} + Z^{(a,\cap)}_{uv} ~\mbox{for all}~ x,y\in B\cap B',\\
Y^{\eta, (a,\setminus)}_{uv} &= \sqrt{\frac{(1-t)\eta}{|B\setminus B'|}}\theta_{u}\theta_{v} + Z^{(a,\setminus)}_{uv} ~\mbox{for all}~ x,y \in B \setminus B',
\end{align*}
where $\big\{(Z^{(a,\bullet)}_{uv})_{u,v\in B},(Z^{(a,\cap)}_{uv})_{u,v\in B\cap B'},(Z^{(a,\setminus)}_{uv})_{u,v\in B\setminus B'}\big\}$ are all mutually independent standard normal random variables which are also independent from everything else.
Importantly, these noise r.v.'s are independent from those used in the construction of side information on a different block. 
The fact that these observations are available is due the relation $L =2\scL$.  
We concatenate this side information in a convenient notation:
\begin{align}\label{eq:processed_side_info}
Y^{\eta}_B &:= \bigcup_{\underset{B' = B+\scL a}{a \in \{\pm e_1,\cdots,\pm e_d\}}} \Big(\{Y^{\eta, (a,\bullet)}_{xy}: (x,y)\in B^2\} 
\cup \{Y^{\eta, (a,\cap)}_{xy}: (x,y)\in (B \cap B')^2\}\\ 
&\hspace{4cm}\cup \{Y^{\eta, (a,\setminus)}_{xy}: (x,y)\in (B\setminus B')^2 \}\Big).\nonumber
\end{align}
With this notation let us reiterate that $Y^{\eta}_{B}$ and $Y^{\eta}_{B'}$ are independent conditionally on $\theta$ for $B \neq B'$.
This construction is designed to `lock' certain overlaps pertaining to the core and the joints of the block $B$ together. 
This locking property will be explained in greater detail in Section~\ref{sec:overview_proof}.

\paragraph{3.\ Sampling.} For each block $B \in \bbB_n$, we independently generate a sample $\big(\theta^{B}_x\big)_{x\in B^{}}$ from the posterior distribution of $(\theta_{x})_{x \in B}$ given the lattice observations $Y^{\delta}_{B} = \{Y^{\delta}_{uv} : u,v\in B, |u-v|=1\}$ and the side information $Y^{\eta}_{B}$ as defined above:   
\begin{equation}\label{eq:sample_posterior}
\big(\theta^{B}_x\big)_{x\in B^{}} \sim \P\big(\cdot | Y^{\delta}_{B} ,Y^{\eta}_{B}\big).
\end{equation}

\paragraph{4.\ New synchronization variables.} For every two adjacent blocks $B \sim B' \in \bbB_n$, we construct a new  synchronization variable $\widetilde{Y}_{B,B'}$ by taking the sign of the inner product of $(\theta^{B}_x)_{x\in B\cap B'}$ and $(\theta^{B'}_x)_{x\in B\cap B'}$:   
\begin{equation}\label{eq:new_observations}
\widetilde{Y}_{B,B'} = \sign\Big(\sum_{x\in B \cap B'} \theta^{B}_x\theta^{B'}_x\Big).
\end{equation}

\paragraph{5.\ Multiscale scheme.}  We apply the multiscale scheme presented in Section~\ref{sec:multiscale} on the new observations $(\widetilde{Y}_{B,B'})_{B\sim B'}$ to synchronize the blocks and produce estimates $\hat{\sigma}_B$ for every block $B \in \bbB_n$.  

\paragraph{6.\ Final estimates.} We output final estimates for every vertex $x$ in the original lattice $\Lambda_n$: $\hat{\theta}_x = \hat{\sigma}_B \theta^B_x$ where $B \in \bbB_n$ is such that $x \in B$. (If $x$ belongs to the intersection of two adjacent blocks then we choose one of the blocks containing it arbitrarily.) We then let $T_{xy} = \hat{\theta}_x\hat{\theta}_y$ for all pairs $x,y \in \Lambda_n$.

 \section{Renormalization}
 \label{sec:renormalization}
 The above algorithm builds a bigger lattice consisting of blocks and constructs new pairwise measurements between blocks (Step 4). We explain here the purpose of this operation.  Let us   
 assign a spin $\tilde{\theta}_{B} \in \{\pm 1\}$ to every block $B\in \bbB_n$ based on the inner product of $(\theta^{B}_x)_{x\in B}$ as constructed in Step 3, and the hidden assignment $\theta$:  
 \begin{equation}\label{eq:block_spins}
 \tilde{\theta}_{B} = \sign\Big(\sum_{x\in B}\theta^{B}_x\theta_x\Big).
 \end{equation}
These variables are of course hidden from the observer as they depend on $\theta$.   
 This provides us with a new synchronization instance $\{(\tilde{\theta}_{B})_{B\in \bbB_n}, (\widetilde{Y}_{B,B'})_{B\sim B'}\}$ on the rescaled lattice $\bbL$. So
 we mapped the original synchronization instance $\{(\theta_u)_{u\in \Lambda_n},(Y_{uv}^{\delta})_{(u,v)\in E_n}\}$ on $(\Lambda_n,E_n)$ to a new, \emph{renormalized} synchronization instance $\{(\tilde{\theta}_{B})_{B\in \bbB_n}, (\widetilde{Y}_{B,B'})_{B\sim B'}\}$ on $\bbL$ where the role of vertices is now played by blocks\footnote{Observe however that this new synchronization instance  is a more complicated graphical model since it does not have the same conditional independence structure as the model as defined by~\eqref{eq:lattice_model}}. The observations $\widetilde{Y}_{B,B'}$ are meant to be noisy versions of the products $\tilde{\theta}_{B}\tilde{\theta}_{B'}$, the same way $Y_{uv}$ is a noisy version of the product $\theta_u\theta_v$. Moreover, the outputs $(\hat{\sigma}_B)_{B \in \bbB_n}$ produced in Step 5 of the algorithm are meant to be such that $\hat{\sigma}_B\hat{\sigma}_{B'}$ is a good estimate of $\tilde{\theta}_B\tilde{\theta}_{B'}$ for all pairs $B,B'$.

The main point of this mapping is to increase the signal strength: whereas $\P(Y^{\delta}_{uv} = -\theta_u \theta_v) = p$ is a constant of the problem, we will show that the \emph{effective noise parameter} of the renormalized model
\begin{equation}\label{eq:renormalized_error}
 \scp:= \P\left(\widetilde{Y}_{B,B'} = - \tilde{\theta}_{B}\tilde{\theta}_{B'}\right),
 \end{equation}
tends to zero as the renormalization rescale $\scL$ becomes large.   
\begin{theorem}\label{thm:renormalization_flow}
Under Assumption~\ref{assump:lipschitz}, for $\delta> \delta_c$ and almost every $t\in [0,1]$,
\[\lim_{\eta\to 0^+} \lim_{\scL\to \infty} \scp= 0.\]
\end{theorem}

Therefore, for $\scL$ large (and $\eta$ small), it is easier to solve the synchronization problem on the rescaled lattice $\bbL$ since this problem is less noisy:   
we cary out in Section~\ref{sec:multiscale} a multiscale analysis to synchronize the blocks $B \in \bbB_n$ given the new synchronization variables $(\widetilde{Y}_{B,B'})_{B\sim B'}$, Eq.~\eqref{eq:new_observations} and show the following result.
\begin{theorem}\label{thm:multiscale}
Under Assumption~\ref{assump:lipschitz}, for $\delta>\delta_c$, there exists $\eta_0>0$ such that for almost all $\eta\le \eta_0$, there exists $L_0 \ge 1$ such that for all $\scL \ge L_0$ and almost every $t \in [0,1]$, the following holds. For every pair of blocks $B_1,B_2 \in \bbB_n$, there exists an estimator $\widetilde{T}_{B_1,B_2} =  \hat{\sigma}_{B_1}\hat{\sigma}_{B_2} \in \{\pm 1\}$ which takes as input the new observations $(\widetilde{Y}_{B,B'})_{B\sim B'}$ and satisfies 
\[\P\Big(\widetilde{T}_{B_1,B_2} \neq \tilde{\theta}_{B_1}\tilde{\theta}_{B_2}\Big) \le \frac{1}{21}.\]
Moreover, the estimators $(\widetilde{T}_{B_1,B_2})_{B_1,B_2 \in \bbB_n}$ can be computed in time $\bigo(n^{2d})$. 
\end{theorem}

The proof of Theorem~\ref{thm:multiscale} relies on an adaptation of an approach already used in~\cite{abbe2017group}, augmented with the result of Theorem~\ref{thm:renormalization_flow}. The latter theorem is a consequence of a set of several asymptotic statements concerning intra- and inter-block overlaps which we collect in the next theorem. 

For $\big(\theta^{B}_x\big)_{x\in B} \sim \P\big(\cdot | Y^{\delta}_B ,Y^{\eta}_B\big)$ and $\big(\theta^{B'}_x\big)_{x\in B'} \sim \P\big(\cdot | Y^{\delta}_{B'} ,Y^{\eta}_{B'}\big)$, we define the correlation between samples on adjacent blocks $B\sim B'$:  
\begin{equation}\label{eq:overlap_intersection}
W_{B,B'} := \frac{1}{|B \cap B'|}  \sum_{x\in B \cap B'} \theta^{B}_x\theta^{B'}_x,
\end{equation} 
and between samples and the ground truth within a block $B \in \bbB_n$:
\begin{equation}\label{eq:overlap_block}
M_B :=  \frac{1}{|B|} \sum_{x\in B}\theta^{B}_x \theta_x,\quad \mbox{and} \quad M_{B^{\bullet}} :=  \frac{1}{|B^{\bullet}|}\sum_{x\in B^{\bullet}}\theta^{B}_x \theta_x,
\end{equation}
where $B^{\bullet}$ is the core of the block $B$. 

\begin{theorem}\label{thm:everything}
Under Assumption~\ref{assump:lipschitz},  for all $\delta>\delta_c$ and almost every $t\in [0,1]$, the following holds in the limit $\scL \to \infty$ followed by $\eta\to 0^+$: 
\begin{enumerate}
\item For every block $B\in \bbB_n$, 
\[ \E[M_B^2]  \longrightarrow q_{\star\star}^2(\delta), ~~ \E\big[(M_B - M_{B^{\bullet}})^2\big] \longrightarrow 0, ~~\mbox{and}~~ \var(M_{B}^2) \longrightarrow 0.\]
\item For every pair of adjacent blocks $B\sim B'$, 
\[ \E[W_{B,B'}^2] \longrightarrow q_{\star\star}^2(\delta), ~~\mbox{and}~~  \var(W_{B,B'}^2) \longrightarrow 0.\]
\item For every pair of adjacent blocks $B\sim B'$,
\[ \E[W_{B,B'} M_B M_{B'}] \longrightarrow q_{\star\star}^3(\delta).\]
\end{enumerate}
\end{theorem}
%
\subsection{An overview of the main ideas of the proof} 
\label{sec:overview_proof}
The proof of Theorem~\ref{thm:everything} is where the side information~\eqref{eq:processed_side_info} is crucially used.  We will give here a view of the main ideas. The full proof is presented in Section~\ref{sec:proof_everything}. The statements of the theorem follow from the study of the analytic properties (with respect to the parameters $t$ and $\eta$) of the \emph{free energies} associated to a block $B$ and to the union of two adjacent blocks $B \cup B'$, in the limit where the volume of $B$ and $B'$ tends to infinity. These free energies are the (properly normalized) expected logarithms of the partition functions of the posterior measures of $(\theta_u)_{u\in B}$ given $Y^{\delta,\eta}_{B}$, and $(\theta_u)_{u\in B \cup B'}$ given $Y^{\delta,\eta}_{B \cup B'}$ respectively, see e.g., Eq.~\eqref{eq:free_1}. We will show that these free energies converge to certain variational formulas involving the limiting free energy of the posterior measure of $\theta_{\Lambda_n}$ given the lattice measurements $Y^{\delta}_{\Lambda_n}$ together with a side information where one observes the spin $\theta_u$ of every vertex $u$ corrupted by an \emph{independent} Gaussian noise $z_u$. These results are stated and discussed in Section~\ref{sec:posteriors}, and proved in Section~\ref{sec:free_energy}. The convexity of the free energies with respect to $\eta$ allows to deduce convergence and concentration of the overlap $M_B$ in~\eqref{eq:overlap_block}. 

As for $M_{B^{\bullet}}$ and $W_{B,B'}$, these overlaps involve only a subset of the vertices of the block, so it is not immediately clear that they should also converge. This is where we show a \emph{overlap locking property} by varying the parameter $t$. For the sake of this discussion, suppose we have a sequence of graphs $G_n = (V_n,E_n)$ of growing size, and let $(A_n,B_n)$ be a partition of $V_n$ into two subsets such that $ |A_n|/|V_n| = \alpha_n \to \alpha$ and $|B_n|/|V_n| \to 1-\alpha$, $\alpha \in (0,1)$. Each vertex $u$ is endowed with random variable $\theta_u$ drawn independently at random form a distribution $p_0$, and we receive an observation $Y_{uv} \sim Q( \cdot | \theta_u, \theta_v)$ independently on every edge $(u,v) \in E$, $Q$ being a probability kernel. Additionally, in the spirit of the Guerra-Toninelli interpolation~\cite{guerra2002thermodynamic}, we receive side information of the form 
\begin{align*}
\tilde{Y}_{uv} &= \sqrt{\frac{t \lambda}{|V_n|}} \theta_u \theta_v + Z_{uv}~~~ \mbox{for all}~ u,v \in V_n,\\
\tilde{Y}_{uv}' &= \sqrt{\frac{(1-t) \lambda}{|A_n|}} \theta_u \theta_v + Z_{uv}'~~~ \mbox{for all}~ u,v \in A_n,\\
\tilde{Y}_{uv}'' &= \sqrt{\frac{(1-t) \lambda}{|B_n|}} \theta_u \theta_v + Z_{uv}''~~~ \mbox{for all}~ u,v \in B_n, 
\end{align*}
where $ Z_{uv},  Z_{uv}', Z_{uv}''$ are mutually independent standard normal random variables. We define the free energy $\phi_n(\lambda,t)$ as the expected logarithm of the normalizing constant of the posterior measure $\P(\cdot |  Y,\tilde{Y})$:
\begin{equation}\label{eq:free_1}
\phi_n(\lambda,t) = \frac{1}{|V_n|} \E \log \int \prod_{(u,v) \in E_n} Q(Y_{uv} | \theta_u, \theta_v) \cdot e^{H_n(\theta)} \prod_{u \in V_n} \rmd p_0(\theta_u),
\end{equation}
where $H_n = H_{V_n} + H_{A_n} + H_{B_n}$ and 
\begin{align*}
H_{V_n}(\theta) &= \sum_{u,v\in V_n} \sqrt{\frac{t \lambda}{|V_n|}} \tilde{Y}_{uv}\theta_u \theta_v - \frac{t \lambda}{2|V_n|} \theta_u^2 \theta_v^2, \\
H_{A_n}(\theta) &= \sum_{u,v\in A_n} \sqrt{\frac{(1-t) \lambda}{|A_n|}} \tilde{Y}_{uv}'\theta_u \theta_v - \frac{(1-t) \lambda}{2|A_n|} \theta_u^2 \theta_v^2, \\
H_{B_n}(\theta) &= \sum_{u,v\in B_n} \sqrt{\frac{(1-t) \lambda}{|B_n|}} \tilde{Y}_{uv}''\theta_u \theta_v - \frac{(1-t) \lambda}{2|B_n|} \theta_u^2 \theta_v^2 .
\end{align*} 
Let us assume that $\phi_n(\lambda,t)$ converges pointwise to a limit $\phi(\lambda,t)$ as $|V_n| \to \infty$. Since $\phi_n$ is convex with respect to $t$, so is $\phi$, and we have $\frac{\rmd}{\rmd t} \phi_n \to \frac{\rmd}{\rmd t} \phi$ for all $\lambda \ge 0$ and almost all $t \in [0,1]$. For a subset $S \subseteq V_n$, we let $R_{1,2}(S) = \frac{1}{|S|} \sum_{u \in S} \theta_u^1\theta_u^2$ where $\theta^1,\theta^2$ are drawn independently from $\P(\cdot | Y,\tilde{Y})$. A computation using Gaussian integration by parts yields 
\begin{align*}
\frac{\rmd}{\rmd t} \phi_n &= \frac{\lambda}{2 |V_n|} \E \Big \langle |V_n| R_{1,2}(V_n)^2 - |A_n| R_{1,2}(A_n)^2 -|B_n| R_{1,2}(B_n)^2\Big\rangle\\
&= -\frac{\lambda}{2} \alpha_n (1-\alpha_n) \E \Big \langle \Big( R_{1,2}(A_n) - R_{1,2}(B_n) \Big)^2\Big\rangle.
\end{align*}                   
The bracket angles $\langle \cdot \rangle$ denote the average with respect to the posterior measure $\P(\cdot |  Y,\tilde{Y})$. 
Now if we are able to show that the limit $\phi$ does not depend on $t$, then we obtain $R_{1,2}(V_n) \simeq R_{1,2}(A_n) \simeq R_{1,2}(B_n)$ ``asymptotically almost surely" for almost all values of $t$. We say in this case that $R_{1,2}(A_n)$ and $R_{1,2}(B_n)$ are \emph{locked} together.  
This is essentially the strategy for treating $M_{B^{\bullet}}$ and $W_{B,B'}$: we show that they are locked to $M_B$ and $M_{B \cup B'}$ respectively. The unproven Assumption~\ref{assump:lipschitz} is required in the process of proving the convergence of $\phi_n \to \phi$, (and in particular for showing that $\phi$ it does not depend on $t$; see Proposition~\ref{prop:limits_phi1_phi2} and Section~\ref{sec:proofs_limits}.   

We observe that a similar property was discovered by Panchenko in the context of spin glasses for the multi-species Sherrington-Kirkpatrick model~\cite{panchenko2015free}, but the underlying mechanism is quite different, although it also relies on a perturbation of the Hamiltonian. Panchenko names this property \emph{synchronization}, but this word has a different meaning in our context.     

Next, showing $\E[W_{B,B'} M_B M_{B'}] \longrightarrow q_{\star\star}^3(\delta)$ requires one more ingredient, of an operator-theoretic flavor. Let us first expand the expression at hand:
\begin{align*}    
\E[W_{B,B'} M_B M_{B'}] &= \frac{1}{|B \cap B'| |B| |B'|} \sum_{x \in B \cap B' , y \in B, z \in B'} \E\big[\theta_x^{B}\theta_x^{B'} \theta_y^{B}\theta_y \theta_z^{B'}\theta_z\big]\\
= \frac{1}{|B \cap B'| |B| |B'|}& \sum_{x \in B \cap B' , y \in B, z \in B'} \E\Big[\E\big[\theta_x \theta_y | Y_B\big] \E\big[\theta_x \theta_z| Y_{B'}\big] \E\big[\theta_y \theta_z | Y_{B \cup B'}\big]\Big].
\end{align*}
We use the locking property to show that the above expression can be approximated by 
\[\frac{1}{|B \cup B'|^3} \sum_{x,y,z \in B \cup B'} \E\Big[\E\big[\theta_x \theta_y | Y_{B \cup B'}\big] \E\big[\theta_x \theta_z| Y_{B \cup B'}\big] \E\big[\theta_y \theta_z | Y_{B \cup B'}\big]\Big].\]  
We observe that this is equal to $\E \trace (\chi^3)$, where $\chi$ is the $|B \cup B'| \times |B \cup B'|$  matrix with entries
\[\chi_{xy} = \frac{1}{|B \cup B'|} \E\big[\theta_x \theta_y| Y_{B \cup B'}\big].\]
On the other hand we know that, say, from Proposition~\ref{prop:conv_pair_corr}, that the expected trace of the square of $\chi$ converges to $q_{\star\star}^2(\delta)$:
\[\E \trace (\chi^2) = \frac{1}{|B \cup B'|^2} \sum_{x,y \in B \cup B'} \E\Big[\E\big[\theta_x \theta_y | Y_{B \cup B'}\big]^2\Big] \longrightarrow q_{\star\star}^2(\delta).\]
Therefore showing that the expected trace of the cube of $\chi$ converges to $ q_{\star\star}^3(\delta)$ amounts to showing that $\chi$ is approximately rank one. Concretely we will how that $\E\|\chi\|_{\textup{op}} \to q_{\star\star}(\delta)$, then use the sandwiching $\|\chi\|_{\textup{op}} \le \trace (\chi^3)^{1/3} \le  \trace (\chi^2)^{1/2}$ (valid since $\chi$ is a positive semidefinite matrix) together with a concentration argument. Now it suffices to exhibit a vector $v \in \R^{|B \cup B'|}$ maximizing the quadratic form $v^\top \chi \, v$. We let
\[v_x = \frac{1}{\sqrt{|B \cup B'|}} \E\big[\theta_x | Y_{B \cup B'} , \tilde{y}\big],~~~~ x \in B \cup B',\] 
where $\tilde{y} \in \R^{|B \cup B'|}$ is a random vector with i.i.d.\ coordinates $\tilde{y}_x = \sqrt{\varepsilon} \theta_x + z_x$, where $z_x \sim N(0,1)$. (Note that by symmetry, $v=0$ if $\varepsilon = 0$, therefore the necessity of introducing a small bias.)
We will show that as $\varepsilon \to 0$, we have $\E |v|^2 \to q_{\star\star}(\delta)$, and $\E\big[v^\top \chi \, v \big] \to q_{\star\star}^2(\delta)$, and then finish the proof by a concentration argument. These concentrations statements are collectively called \emph{decoupling bounds} and are presented in Section~\ref{sec:decoup}. 

\subsection{Proof of Theorem~\ref{thm:renormalization_flow} and Theorem~\ref{thm:reconstruction}}
\label{sec:proof_main}
We prove Theorem~\ref{thm:renormalization_flow} as a consequence of Theorem~\ref{thm:everything}:
\begin{proof}[Proof of Theorem~\ref{thm:renormalization_flow}]
For two adjacent block $B \sim B'$, let $\varepsilon_{B,B'} = \sign(W_{B,B'} M_B M_{B'})$. Then
\[\E[W_{B,B'} M_B M_{B'}] = \E[\varepsilon_{B,B'} |W_{B,B'}| |M_B| |M_{B'}|].\]
Since for any non-negative random variable we have $\var(X)^2 \le \frac{1}{2}\var(X^2)$, the variance statements in items 1 and 2 of Theorem~\ref{thm:everything} imply that $|W_{B,B'}|$, $|M_B|$ and $|M_{B'}|$ have vanishing variances. Combined with convergence of the expectations, this implies that $\E[|M_B|]$, $\E[|M_{B'}|]$ and $\E[|W_{B,B'}|]$ converge to $q_{\star\star}(\delta)$.  
Therefore
\[ \E[\varepsilon_{B,B'} |W_{B,B'}| |M_B| |M_{B'}|] = \E[\varepsilon_{B,B'}] q_{\star\star}^3(\delta) + o_{\scL,\eta}(1),\]
where $o_{\scL,\eta}(1)\to 0$ for a.e.\ $t\in[0,1]$ as $\scL \to \infty$ then $\eta\to 0$.
Given item 3, this implies that when $\delta>\delta_c$, $q_{\star\star}(\delta)>0$ and
\[ 1-2\scp= \E[\varepsilon_{B,B'}] \longrightarrow 1.\]
\end{proof}

Next, we use Theorem~\ref{thm:multiscale} to prove our main reconstruction result: Theorem~\ref{thm:reconstruction}. 

\begin{proof}[Proof of Theorem~\ref{thm:reconstruction}]
Recall that our final estimate of $\theta_u\theta_v$ is $T_{uv} = \widetilde{T}_{B_u,B_v}\theta^{B_u}_u\theta^{B_v}_v$, where $B_u$ and $B_v$ are blocks to which vertices $u$ and $v$ respectively belong (if a vertex happens to belong to a joint then its parent block is chosen arbitrarily) and $\theta^{B_u}_u$ is given in~\eqref{eq:sample_posterior}.  
We compute the performance of this estimator.  
Since the graph $\bbL$ is bipartite, $\Lambda_n$ can be partitioned into an alternating succession of blocks and cores of blocks: start with the block at the origin then take the cores of all the $3^d -1$ blocks surrounding it, and recurse. This forms a disjoint union. We color white ($\texttt{w}$)  the blocks $B \in \bbB_n$ appearing in full in this decomposition and  gray ($\texttt{g}$) those whose cores appear instead. Below, we use the notation $B\texttt{w}$ or $B\texttt{g}$ to indicate the color of the block $B$. We write         
\[ \risk_{\Lambda_n}(T) = \frac{1}{|\Lambda_n|^2} \sum_{u,v \in \Lambda_n} \E\Big[T_{uv}\theta_u\theta_v\Big]
= \texttt{ww} + 2\texttt{wg} + \texttt{gg}, \]   
where
\begin{align*}
 \texttt{ww}  =  \sum_{B \texttt{w} ,B'  \texttt{w} \in \bbB_n}& \sum_{u \in B, v \in B'} A_{uv}^{BB'},~~~
 \texttt{wg} = \sum_{B  \texttt{w},B'  \texttt{g} \in \bbB_n} \sum_{u \in B, v \in B'^{\bullet}} A_{uv}^{BB'},\\
 \texttt{gg} &= \sum_{B  \texttt{g},B'  \texttt{g} \in \bbB_n} \sum_{u \in B^{\bullet}, v \in B'^{\bullet}} A_{uv}^{BB'},
 \end{align*}
 and 
\[ A_{uv}^{BB'} =  \frac{1}{|\Lambda_n|^2} \E\Big[\widetilde{T}_{B, B'} \theta^{B}_u\theta^{B'}_v\theta_u\theta_v\Big].\]
We analyze each one the three expressions separately. 
For the last two expressions, we have
 \begin{align*}
  \sum_{u \in B, v \in B'^{\bullet}} A_{uv}^{BB'} &=  \frac{|B| \cdot |B'^{\bullet}|}{|\Lambda_n|^2}  \E\Big[\widetilde{T}_{B, B'} M_{B} M_{B'^{\bullet}}\Big]
 = \frac{|B|\cdot |B'^{\bullet}|}{|\Lambda_n|^2} \Big(\E\Big[\big(\widetilde{T}_{B, B'} M_{B} M_{B'}\Big] + \mbox{error}\Big),\\
  \sum_{u \in B^{\bullet}, v \in B'^{\bullet}} A_{uv}^{BB'} &=  \frac{|B^{\bullet}| \cdot |B'^{\bullet}|}{|\Lambda_n|^2}  \E\Big[\widetilde{T}_{B, B'} M_{B^{\bullet}} M_{B'^{\bullet}}\Big]
 = \frac{|B^{\bullet}| \cdot |B'^{\bullet}|}{|\Lambda_n|^2} \Big(\E\Big[\widetilde{T}_{B, B'}  M_{B} M_{B'}\Big] + \mbox{error}\Big),
  \end{align*}
where $|\mbox{error}| \le 2\E\big[|M_{B'} - M_{B'^{\bullet}}|\big]$.
Moreover,
 \begin{align*}
  \sum_{u \in B, v \in B'} A_{uv}^{BB'} &=  \frac{|B| \cdot |B'|}{|\Lambda_n|^2}  \E\Big[\widetilde{T}_{B, B'} M_{B} M_{B'}\Big],\\
 &= \frac{|B|^2}{|\Lambda_n|^2} \E\Big[\big(\widetilde{T}_{B, B'} \tilde{\theta}_{B} \tilde{\theta}_{B'}\big) |M_{B}| |M_{B'}|\Big].
  \end{align*}
Now we apply item 1 of Theorem~\ref{thm:everything}: $\var(|M_B|) = o_{\scL,\eta}(1)$, $\E[|M_B|] = q_{\star\star} + o_{\scL,\eta}(1)$ and $\mbox{error} = o_{\scL,\eta}(1)$, and Theorem~\ref{thm:multiscale}:  $\E\big[\widetilde{T}_{B, B'} \tilde{\theta}_{B} \tilde{\theta}_{B'}\big] \ge \frac{19}{21}$. Here, $o_{\scL ,\eta}(1) \to 0$ as $\scL \to \infty$ then $\eta \to 0$.
 Whence 
\begin{align*}
 \texttt{ww}  &\ge  \frac{1}{|\Lambda_n|^2}  \sum_{B \texttt{w} ,B'  \texttt{w} \in \bbB_n} |B| \cdot |B'| \Big(\frac{19}{21}q_{\star\star}^2(\delta)+ o_{\scL,\eta}(1)\Big),\\
 \texttt{wg} &\ge  \frac{1}{|\Lambda_n|^2}  \sum_{B \texttt{w} ,B'  \texttt{g} \in \bbB_n} |B| \cdot |B'^{\bullet}|  \Big(\frac{19}{21}q_{\star\star}^2(\delta)+ o_{\scL,\eta}(1)\Big),\\
 \texttt{gg} &\ge  \frac{1}{|\Lambda_n|^2}  \sum_{B \texttt{g} ,B'  \texttt{g} \in \bbB_n} |B^{\bullet}| \cdot |B'^{\bullet}|  \Big(\frac{19}{21}q_{\star\star}^2(\delta)+ o_{\scL,\eta}(1)\Big).
\end{align*}
Consequently,
$\risk_{\Lambda_n}(T)  \ge \frac{19}{21}q_{\star\star}^2(\delta) + o_{\scL,\eta}(1)$,
Now we let $\eta$ be sufficiently small and $\scL$ be sufficiently large (possibly depending on $\eta$) such that the above is, say, larger than $\frac{9}{10}q_{\star\star}^2(\delta)$. 
Finally, the full running time of this algorithm is $\bigo(n^{2d})$ since the construction of the estimators $\widetilde{T}_{B,B'}$ and the sampling operation (step 3) both take constant time in $n$.    
\end{proof}

\section{The multiscale scheme}
\label{sec:multiscale}
We now present the multiscale algorithm and prove Theorem~\ref{thm:multiscale}. 

\subsection{The algorithm}

\paragraph{1.\ Hierarchical partioning.} We construct a hierarchy of partitions $(\mathcal{B}^{(k)})_{k \ge 0}$ of the rescaled lattice $\bbL$ which we define recursively. 
Let $\mathcal{B}^{(0)} = \bbB$ and $\bbL^{(0)} = \bbL$. Let $\mathbf{0}^{(0)} \in \bbB$ is the block centered at the origin.   
Given $\bbL^{(k)}$ which we assume is isomorphic to $(\Z^d,\E^d)$ (this is already the case at level $k=0$), we construct $\bbL^{(k+1)}$ satisfying the same property. Let $d_k$ be the $\ell_{\infty}$ distance in $\bbL^{(k)}$, which is well defined by isomorphism to $(\Z^d,\E^d)$. 
Now we fix a integer constant $\kappa \ge 1$ and define the origin of  $\mathcal{B}^{(k+1)}$ as
\[\mathbf{0}^{(k+1)} =  \Big\{B \in \mathcal{B}^{(k)}: d_k\big(B^{(k)},\mathbf{0}^{(k)}\big) \le \kappa (k+1)^2\Big\},\]  
and let  
\[\mathcal{B}^{(k+1)} =  \Big\{ \mathbf{0}^{(k+1)} + m \ell_k e_i ~:~  m\in \Z, ~i=1,\cdots,d \Big\},\]
where $e_i$ are the coordinate directions of unit length in $\bbL^{(k)}$. Two elements $B_1$ and $B_2$ of $\mathcal{B}^{(k+1)}$ are adjacent, and we write $B_1 \sim B_2$, if one is a translate of the other by $\ell_k e_i$ for some $i$. This completely defines $\bbL^{(k+1)}$. 
We refer to the elements of $\mathcal{B}^{(k)}$ as the level-$k$ blocks. The original lattice $\Z^d$ can be thought of as being at level $k=-1$. We also define $\ell_k = 2\kappa (k+1)^2 +1$ be the side length of a block in $\mathcal{B}^{(k+1)}$. We require the following three conditions on the sequence of scales:
\begin{equation}\label{eq:scales_condition}
\begin{split}
\mathbf{A1}:~~~\sup_{k\ge 0} \Big\{k^{2d} \ell_{k}^{2d} &(2 \kappa)^{-(d-1)(k+5)}\Big\} \le \frac{1}{2}, \qquad
\mathbf{A2}:~~~ \sum_{k=0}^{\infty}\frac{1+3^{d-1}}{\ell_{k}^{d-1}} \le \frac{1}{20}, \\
\mathbf{A3}:&~~~(1+3^{d})\sum_{k=0}^{\infty} k^{2d} (2\kappa)^{-(d-1)(k+6)} \le \frac{1}{42}.
\end{split}
\end{equation}
All three conditions can be satisfied simultaneously if $\kappa = \kappa_0(d)$ is a sufficiently large constant depending on $d$.
Indeed, since $\ell_{k} \sim (2\kappa)(k+1)^2$ for $k$ or $\kappa$ large, the supremum in $\mathbf{A1}$ is finite and can be made arbitrarily small if $\kappa$ is large enough. Next, the two series in $\mathbf{A2}$ and $\mathbf{A3}$ are convergent for all $\kappa >\frac{1}{2}$ and all $d \ge 2$, and their values are decreasing to zero as a function of $\kappa$. We fix $\kappa$ such that $\mathbf{A1, A2, A3}$ are satisfied.

\paragraph{2.\ Block variables 1.} We assign recursively to every level-$k$ block $B\in \mathcal{B}^{(k)}$ a random variable $W^{(k)}_{B} \in \{-1,+1\}$ which is measurable with respect to the new synchronization variables $\{\widetilde{Y}_{B_1,B_2} : B_1\sim B_2 \in \bbB\}$ defined in~\eqref{eq:new_observations}. Our estimates $\widetilde{T}_{B_1,B_2}$ for two level-0 blocks $B_1,B_2$ are defined as
\begin{equation}\label{eq:block_sync}
\widetilde{T}_{B_1,B_2} = \prod_{k \ge 0} W^{(k)}_{B_{1,k}}W^{(k)}_{B_{2,k}},
\end{equation}
where $B_{1,k}$ is the unique level-$k$ block containing $B_2$, and similarly for $B_{2,k}$. Note that the above product is finite and extends up to the most recent ancestor of $B_1$ and $B_2$, that is the lowest level block which contains both $B_1$ and $B_2$. It will be convenient to define the product of the variables $W^{(k)}_{B}$ up to a given high in the hierarchy: For a level-$0$ block $B$ we let 
\[\widetilde{W}^{(k)}_{B} = \prod_{j = 0}^{k} W^{(j)}_{B_{j}},\]
were $B_{j}$ is the unique level-$j$ block containing $B$. (In this case $\widetilde{T}_{B_1,B_2} = \widetilde{W}^{(k)}_{B_1}\widetilde{W}^{(k)}_{B_2}$ where $k$ is the level of the lowest common ancestor of $B_1$ and $B_2$.)

\paragraph{3.\ Synchronization variables.} We recursively compute synchronization variables $Y^{(k)}_{B,B'}$ for every pair of adjacent level-$k$ blocks. For $k=0$ we let $Y^{(0)}_{B,B'}=\widetilde{Y}_{B,B'}$. For general $k$ we proceed as follows. For two adjacent level-$k$ blocks $B_1,B_2$, we define $\partial^{(0)}(B_1,B_2)$ to be the set of pairs of level-$0$ adjacent blocks $B\sim B'$ such that $B \in B_1$ and $B' \in B_2$. 
The set $ \partial^{(0)}(B_1)$ of level-$0$ blocks $B \in B_1$ which have an adjacent block $B' \in B_2$, seen as graph, is isomorphic to a box in $\Z^{d-1}$, and is thus bipartite, therefore 2-colorable. 
Let $\mathsf{col}: \partial^{(0)}(B_1) \mapsto \{0,1\}$ be such a coloring (which can be found in linear time in a recursive way). We now define the subset of $\partial^{(0)}(B_1,B_2)$ consisting of those edges whose end in $B_1$ are colored $1$:
\begin{equation}\label{eq:boundary_skip}
\Xi(B_1,B_2) = \big\{ (B,B') \in \partial^{(0)}(B_1,B_2) ~:~\mathsf{col}(B)=1\big\}.
\end{equation}
Importantly, two edges in $\Xi$ have their endpoints non-adjacent. This property will be important in the analysis due to the fact that two non-adjacent level-$0$ blocks do \emph{not} overlap. Finally, we define 
\begin{equation}\label{eq:level_k_sync_var}
Y^{(k)}_{B_1,B_2} = \sign\Big(\sum_{(B,B')\in \Xi(B_1,B_2)} \widetilde{Y}_{B,B'} \widetilde{W}^{(k-1)}_{B}\widetilde{W}^{(k-1)}_{B'}\Big).
\end{equation}

\paragraph{4.\ Block variables 2.} It remains to describe how $W^{(k)}_B$ is constructed from the knowledge of the level-$k$ synchronization variables $\big\{Y^{(k)}_{B_1,B_2} : B_1 \sim B_2 \in \mathcal{B}^{(k)}\big\}$. Let $B^*$ a level-$(k+1)$ block. We call a \emph{quartet} a collection of 4 level-$k$ blocks $B_1\sim B_2 \sim B_3 \sim B_4 \sim B_5=B_1$ which form a square inside $B^*$. We say that a quartet is \emph{incoherent} if $\prod_{i=1}^4 Y^{(k)}_{B_{i},B_{i+1}} = -1$, and coherent otherwise. Let $\mathcal{I}_{B^*}$ be the largest connected component of level-$k$ sub-blocks $B \in B^*$ which belong to no incoherent quartet. Now we assign, whenever possible, a random variable $W^{(k)}_B$ to every $B \in B^*$ such that 
\begin{equation}\label{eq:new_block_variables}
Y^{(k)}_{B_{1},B_{2}} = W^{(k)}_{B_1} W^{(k)}_{B_2} ~~\mbox{for all}~~ B_1,B_2 \in \mathcal{I}_{B^*} ~\mbox{with}~ B_1\sim B_2.
\end{equation}
If this is not possible then all blocks are assigned the value $+1$. Next, the blocks which do not belong to $\mathcal{I}_{B^*}$ are all assigned the value $+1$.   
This completes the description of the multiscale algorithm.

\subsection{Proof of Theorem~\ref{thm:multiscale}}
We claim that for an appropriate choice of the renormalization length $\scL$, the construction satisfies the conclusion of Theorem~\ref{thm:multiscale}. 
We follow the argument in~\cite{abbe2017group} which used a similar construction directly on the original lattice $\Z^d$ for $d=2$. There are two novel aspects in our setting: the construction is generalized to all dimensions, and the initial renormalization step, in particular the fact that $\scp\to 0$ as $\scL \to \infty$ is used to `jump-start' the argument.     

 Let $\scdelta = 1-2\scp$. Let us call an edge $B_1 \sim B_2$ for two level-$k$ blocks $B_1,B_2 \in \mathcal{B}^{(k)}$ \emph{honest} if
\begin{equation}\label{eq:honest}
\sum_{(B,B')\in \Xi(B_1,B_2)} \widetilde{Y}_{B,B'} \tilde{\theta}_{B} \tilde{\theta}_{B'} \ge \frac{9}{10}\scdelta \big|\Xi(B_1,B_2)\big|.
\end{equation}
Recall $\Xi(B_1,B_2)$ was defined in Eq.~\eqref{eq:boundary_skip}. 
Now we recursively define the notion of a \emph{good} block. Level-$0$ blocks are good. A level-$k \ge 1$ block $B$ is good if two conditions are satisfied:
\begin{itemize}   
\item All adjacent level-$(k-1)$ sub-blocks of $B$ form honest edges.
\item $B$ contains at most one \emph{bad} (i.e., not good) level-$(k-1)$ sub-block.
\end{itemize}
\begin{lemma}\label{lem:bad_block}
There exists $p_0 =p_0(d,\kappa) \in (0,\frac{1}{2})$ such that if $\scp \le p_0$ then for all $k\ge 0$ and all $B \in \mathcal{B}^{(k)}$, 
\[\P(B ~\mbox{is bad}) \le  k^{2d}(2\kappa)^{- (d-1)(k+6)}.\]  
\end{lemma}
\begin{proof}
We proceed by induction. 
The base case $k=0$ being true, let us assume that for some $k\ge 0$ each level-$k$ block is good with probability at least $1-\epsilon_{k}$. Fix a level-$(k+1)$ block $B^*$, and two of its level-$k$ sub-blocks $B_1, B_2$ which are adjacent. We evaluate the probability that this edge is honest. We let $Z_{B,B'} = \widetilde{Y}_{B,B'} \tilde{\theta}_{B} \tilde{\theta}_{B'}$ for two level-$0$ blocks $B$ and $B'$. By construction, $\Xi(B_1,B_2)$ contains edges which are far away by a distance of at least two. If $(B,B')$ and $(B'',B''')$ are two such edges then $B$ and $B''$ do not overlap, and ditto for $B'$ and $B'''$. This crucially implies that the random variables $(Z_{B,B'})_{(B,B') \in \Xi(B_1,B_2)}$ are \emph{mutually independent}. 
Thus the sum on the left-hand side of~\eqref{eq:honest} is a sum of i.i.d.\ signs with common mean $\scdelta$. An application of Bennett's inequality reveals that the probability that this edge is not honest is bounded by 
\[\exp\big(-  c_0(\scp) |\Xi(B_1,B_2)| \big),\] 
where $c_0(\scp) =  4\scp (1-\scp) h\big(\frac{\scdelta}{40 \scp (1-\scp)}\big)$ and $h(u) = (1+u)\log(1+u) - u$, $u \ge 0$.
Further, $B^*$ has $\ell_{k}^d$ level-$k$ sub-blocks and $d\ell_{k}^{d}$ edges, therefore a union bound implies
\[\P(B^* ~\mbox{is bad}) \le d\ell_{k}^{d} e^{-c_0(\scp)|\Xi(B_1,B_2)|} + \P(B^* \mbox{ has at least two bad sub-blocks}).\]
Since the events that two different (sub-)blocks are bad are independent, the induction hypothesis implies that the probability on the right-hand side in the above display is bounded by ${\ell_{k}^d\choose 2} \epsilon_{k}^2$. 
Moreover, have $|\Xi(B_1,B_2)| \ge \lfloor |\partial^{(0)}(B_1,B_2)|/2 \rfloor$ and 
\[|\partial^{(0)}(B_1,B_2)| = \prod_{j=0}^{k-1} \ell_{j}^{d-1} \ge  \big(3(2 \kappa)^{k} k!^2\big)^{d-1} ~~~\mbox{if}~~ k \ge 1,\] 
and $|\partial^{(0)}(B_1,B_2)| = 1$ if $k=0$.
We use the crude bound $|\Xi(B_1,B_2)| \ge (2 \kappa)^{(d-1)k}$ and obtain 
\begin{align*}
\P(B^* ~\mbox{is bad}) &\le d\ell_{k}^{d}\, \exp\big(-c_0(\scp)(2 \kappa)^{(d-1)k}\big) + \ell_{k}^{2d}\epsilon_{k}^2\\
&\le d\ell_{k}^{d} \, c_0(\scp)^{-1} (2 \kappa)^{-(d-1)k} + \ell_{k}^{2d}\epsilon_{k}^2 ~=: \epsilon_{k+1}.
\end{align*}
The problem now reduces to studying the above iteration, starting from $\epsilon_0 = 0$. 
We consider the auxiliary sequence 
\[u_k = k^{-2d} (2\kappa)^{(d-1)(k+6)}\epsilon_k,\]
and show that $\sup_{k} u_k \le 1$ when $\scp$ is small enough. This sequence satisfies
\begin{align*}
u_{k+1} &= \frac{k^{4d} \ell_{k}^{2d}}{(k+1)^{2d}}  (2 \kappa)^{-(d-1)(k+5)} u_{k}^2 + d\frac{\ell_{k}^{d}}{(k+1)^{2d}} \, c_0(\scp)^{-1} (2 \kappa)^{7(d-1)}. \\
&\le k^{2d} \ell_{k}^{2d} (2 \kappa)^{-(d-1)(k+5)} u_{k}^2 + \frac{2d (2\kappa+1)^d(2 \kappa)^{7(d-1)}}{c_0(\scp)}\\
&=: A(k) u_{k}^2 + B. 
\end{align*}
By condition $\mathbf{A1}$ in Eq.~\eqref{eq:scales_condition} we have $A(k) \le \frac{1}{2}$ for all $k \ge 0$. Moreover, it is easy to see that $c_0(\scp) \to +\infty$ as $\scp \to 0$, so there exists $p_0 = p_0(d,\kappa)<1$ such that if $\scp \le p_0$ then $B \le \frac{1}{2}$, and we are lead to consider the iteration
\[u_{k+1} \le \frac{1}{2}u_k^2 + \frac{1}{2},~~ \mbox{with}~~u_0=0.\]  
It is clear that if $u_{k} \le 1$ then $u_{k+1} \le 1$, and therefore $\sup_{k\ge 0} u_k \le 1$, whence $\epsilon_{k} \le k^{2d} (2\kappa)^{-(d-1)(k+6)}$ for all $k \ge 0$. 
\end{proof}

For $B^* \in \mathcal{B}^{(k+1)}$, we denote by $\mathcal{H}^{(k+1)}_{B^*}$ the event that the equation~\eqref{eq:new_block_variables} is satisfied for some choice of variables $(W_{B}, B \in \mathcal{I}_{B^*})$, and call a block $B$ \emph{agreeable} if it belongs to no incoherent quartet, i.e., $B \in \mathcal{I}_{B^*}$ (where $B^*$ is the upper-level block containing it.)   
\begin{lemma}\label{lem:induction_agreeable}
For all $k\ge0$, if the $k$-level block $B_1$ is good then the following holds:
\begin{enumerate}
\item If $k\ge 1$ then all quartets of level-$(k-1)$ good sub-blocks of $B_1$ are coherent.  
\item The event $\mathcal{H}_{B_1}^{(k)}$ holds.
\item There exists random a variable $S_{B_1}^{(k)} \in \{\pm 1\}$ such that if $B$ is a level-$0$ block in $B_1$ whose ancestors up to $B_1$ are good and agreeable, then
\[\tilde{\theta}_{B} = S_{B_1}^{(k)} \widetilde{W}^{(k-1)}_{B}.\]  
(By convention, $\widetilde{W}^{(-1)}_{B}=1$.) 
\item For any good level-$k$ block $B_2 \sim B_1$, we have
\[\sum_{\underset{\mathsf{col}(B)=1}{B \in B_1 \cap \partial B_2}} \tilde{\theta}_{B} S_{B_1}^{(k)} \widetilde{W}^{(k-1)}_{B} \ge (1-\epsilon_k) |\Xi(B_1,B_2)|,\] 
where $\epsilon_0 = 0$ and $\epsilon_{k+1}=\epsilon_{k} +  \frac{2(3^{d-1}+1)}{\ell_k^{d-1}}$ for all $k\ge 0$. In particular, due to condition $\mathbf{A2}$ in Eq.~\eqref{eq:scales_condition}, we have
\[\epsilon_k = \sum_{j\le k} \frac{2(3^{d-1}+1)}{\ell_{j}^{d-1}} \le \frac{1}{10} ~~~\mbox{for all}~k \ge 0.\]
\end{enumerate} 
\end{lemma}
\begin{proof}
We proceed by induction. The base case $k=0$ is clear: we can take $S_{B}^{(0)} = \tilde{\theta}_{B}$ for all level-$0$ blocks $B$ (which are all good by definition). Now assume the above statements for some $k\ge 0$. Let $B_1^*$ be a level-$(k+1)$ good block and $B_1 \sim B_2$ two good and adjacent level-$k$ sub-blocks of $B_1^*$. We have 
\begin{align*}
Y^{(k)}_{B_1,B_2} &= \sign\Big(\sum_{(B,B')\in \Xi(B_1,B_2)} \widetilde{Y}_{B,B'} \widetilde{W}^{(k-1)}_{B}\widetilde{W}^{(k-1)}_{B'}\Big)\\
&= S_{B_1}^{(k)}S_{B_2}^{(k)} \cdot \sign\Big(\sum_{(B,B')\in \Xi(B_1,B_2)} \widetilde{Y}_{B,B'} \big(\widetilde{W}^{(k-1)}_{B}S_{B_1}^{(k)}\big)\big(\widetilde{W}^{(k-1)}_{B'}S_{B_2}^{(k)}\big)\Big).
\end{align*}
We want to show that the sum in the above display is positive. Item 4 in the induction hypothesis implies
\[\sum_{(B,B')\in \Xi(B_1,B_2)} \widetilde{Y}_{B,B'} \big(\widetilde{W}^{(k-1)}_{B}S_{B_1}^{(k)}\big)\big(\widetilde{W}^{(k-1)}_{B'}S_{B_2}^{(k)}\big) \ge \sum_{(B,B')\in \Xi(B_1,B_2)} \widetilde{Y}_{B,B'} \tilde{\theta}_{B}\tilde{\theta}_{B'} - 2\epsilon_k |\Xi(B_1,B_2)|.\]  
Now since $B_1^*$ is good the edge $(B_1,B_2)$ is honest and the above is further lower-bounded by
\[\big(\frac{9}{10}\scdelta - 2\epsilon_k\big)|\Xi(B_1,B_2)| \ge \frac{3}{5}|\Xi(B_1,B_2)|>0,\] 
for $\scdelta > \frac{9}{10}$ (i.e., $\scp < \frac{1}{20}$.) Therefore
\[Y^{(k)}_{B_1,B_2} = S_{B_1}^{(k)}S_{B_2}^{(k)},\]
for all $B_1 \sim B_2$ good sub-blocks. This means that quartets of good sub-blocks in a good block are necessarily coherent, proving item 1. 
The set $\mathcal{I}_{B_1^*}$ being connected, it follows that there exists $S_{B_1^*}^{(k+1)}\in \{\pm 1\}$ such that for all level-$k$ good sub-blocks $B_1 \in \mathcal{I}_{B_1^*}$,
\begin{equation}\label{eq:w_b}
W_{B_1}^{(k)} = S_{B_1}^{(k)}S_{B_1^*}^{(k+1)},
\end{equation}
hence the event $\mathcal{H}_{B_1^*}^{(k+1)}$ holds. Moreover, for such a good $B_1 \in \mathcal{I}_{B_1^*}$ we have
\[S_{B_1^*}^{(k+1)} \widetilde{W}^{(k)}_{B} = S_{B_1^*}^{(k+1)} W^{(k)}_{B_1} \widetilde{W}^{(k-1)}_{B} = S_{B_1}^{(k)}\widetilde{W}^{(k-1)}_{B} = \tilde{\theta}_B,\]  
for all level-$0$ blocks $B$ whose ancestors are good and agreeable. Hence item 3 holds at level $k+1$.
It remains to verify item 4 at level $k+1$. 
Let $B_2^*$ be a good level-$(k+1)$ block adjacent to $B_1^*$. We have
\begin{align*}
\sum_{\underset{\mathsf{col}(B)=1}{B \in B_1^* \cap \partial B_2^*}} \tilde{\theta}_{B} S_{B_1^*}^{(k+1)} \widetilde{W}^{(k)}_{B} = 
\sum_{B_1^* \ni \bar{B}_1 \sim \bar{B}_2 \in B_2^*}  \Big(\sum_{\underset{\mathsf{col}(B)=1}{ \bar{B}_1 \ni B \sim B' \in  \bar{B}_2}}  \tilde{\theta}_{B}  \widetilde{W}^{(k-1)}_{B} \Big) S_{B_1^*}^{(k+1)} W_{\bar{B}_1}^{(k)},
\end{align*}
where the outer sum is on level-$k$ sub-blocks $\bar{B}_1$ and $\bar{B}_2$. 
Let $\mathcal{A}$ be the subset of those $\bar{B}_1$, level-$k$ sub-blocks of $B_1^*$ which are good, are in $\mathcal{I}_{B_1^*}$, and are adjacent to a good sub-block $\bar{B}_2 \in B_2^*$. Since there is only one bad sub-block in $B_1^*$, there are at most $3^{d-1}$ positions on the boundary of $B_1^*$ which could possibly belong to an incoherent quartet because this quartet must necessarily contain that single bad sub-block. Moreover, $B_2^*$ being good, it has only one bad sub-block. Therefore the above sum is lower-bounded by
\begin{align*}
&\sum_{\bar{B}_1 \in \mathcal{A}}  \Big(\sum_{\underset{\mathsf{col}(B)=1}{ \bar{B}_1 \ni B \sim B' \in  \bar{B}_2}}  \tilde{\theta}_{B}  \widetilde{W}^{(k-1)}_{B} \Big) S_{B_1^*}^{(k+1)} W_{\bar{B}_1}^{(k)} - (1+3^{d-1}) |\Xi(\bar{B}_1,\bar{B}_2)| \\
&= \sum_{\bar{B}_1 \in \mathcal{A}}  \Big(\sum_{\underset{\mathsf{col}(B)=1}{ \bar{B}_1 \ni B \sim B' \in  \bar{B}_2}}  \tilde{\theta}_{B}  \widetilde{W}^{(k-1)}_{B} S^{(k)}_{\bar{B}_1}\Big)  - (1+3^{d-1}) |\Xi(\bar{B}_1,\bar{B}_2)|.
\end{align*}
We use item 4 of the induction hypothesis: $\underset{\underset{\mathsf{col}(B)=1}{ \bar{B}_1 \ni B \sim B' \in  \bar{B}_2}}{\sum}  \tilde{\theta}_{B}  \widetilde{W}^{(k-1)}_{B} S^{(k)}_{\bar{B}_1} \ge (1-\epsilon_k) |\Xi(\bar{B}_1,\bar{B}_2)|$, whence the above is lower bounded by
\begin{align*}
&\sum_{\bar{B}_1 \in \mathcal{A}}  (1-\epsilon_k)  |\Xi(\bar{B}_1,\bar{B}_2)| - (1+3^{d-1}) |\Xi(\bar{B}_1,\bar{B}_2)| \\ 
&\ge \sum_{\underset{\mathsf{col}(B)=1}{B_1^* \ni \bar{B}_1 \sim \bar{B}_2 \in B_2^*}}  (1-\epsilon_k)  |\Xi(\bar{B}_1,\bar{B}_2)| -  (2-\epsilon_k)(1+3^{d-1}) |\Xi(\bar{B}_1,\bar{B}_2)|\\
&= (1-\epsilon_k) |\Xi(B_1^*,B_2^*)| - (2-\epsilon_k)(1+3^{d-1}) |\Xi(\bar{B}_1,\bar{B}_2)|.
\end{align*}
Since $|\Xi(B_1^*,B_2^*)|/|\Xi(\bar{B}_1,\bar{B}_2)| = \ell_k^{d-1}$ we obtain
\begin{align*}
\sum_{\underset{\mathsf{col}(B)=1}{B \in B_1^* \cap \partial B_2^*}} \tilde{\theta}_{B} S_{B_1^*}^{(k+1)} \widetilde{W}^{(k)}_{B} &\ge \Big(1-\epsilon_k - \frac{2(1+3^{d-1})}{\ell_k^{d-1}}\Big) |\Xi(B_1^*,B_2^*)|\\
&= (1-\epsilon_{k+1}) |\Xi(B_1^*,B_2^*)|.
\end{align*}
This concludes the inductive argument. 
\end{proof}

We now finish the proof of Theorem~\ref{thm:multiscale}. Let $B_1$ and $B_2$ be two level-$0$ blocks and let $B^*$ be their most recent common ancestor in the hierarchical partitioning; we let $k$ be its level.  Item 2 in Lemma~\ref{lem:induction_agreeable} implies that a sufficient condition for accurately synchronizing $B_1$ and $B_2$, i.e., $\widetilde{T}_{B_1,B_2} = \tilde{\theta}_{B_1} \tilde{\theta}_{B_2}$, is that $B_1$ and $B_2$ along with their ancestors $B^{(j)}_{1}$ and $B^{(j)}_{2}$ for $j=0,\cdots,k$ up to $B^*$ are all good and agreeable. (Here we have $B^{(0)}_{1} = B_1$ and $B^{(k)}_{1} = B^*$ and similarly for $B_2$.) Let us call this event $J_{B_1,B_2}^{(k)}$. By a union bound we obtain
\[1-\P\big(J_{B_1,B_2}^{(k)}\big) \le \sum_{j=0}^{k-1} \sum_{a=1}^2\P\big(B^{(j)}_{a} ~\mbox{is bad or non-agreeable}\big)+\P\big(B^* ~\mbox{is bad}\big).\]
A non-agreeable block must belong to at least one incoherent quartet. Since a quartet of good sub-blocks belonging to a good block is necessarily coherent, a non-agreeable block inside a good block is the center of a cube of side-length 3 in which at least one block is bad. Hence for $0 \le j \le k-1$ and $a=1,2$,
\begin{align*}
\P\big(B^{(j)}_{a} ~\mbox{is bad or non-agreeable}\big) &\le \P\big(B^{(j)}_{a} ~\mbox{is bad or non-agreeable}~ \big| B^{(j+1)}_{a} ~\mbox{is good} \big)\\
&~~~+ \P\big(B^{(j+1)}_{a} ~\mbox{is bad} \big) \\
&\le  3^d \P\big(B^{(j)}_{a} ~\mbox{is bad}\big) +  \P\big(B^{(j+1)}_{a} ~\mbox{is bad} \big).
\end{align*}
Therefore,
\begin{align*}
1-\P\big(J_{B_1,B_2}^{(k)}\big) &\le 2 \cdot 3^d  \sum_{j=0}^{k-1} \P\big(B^{(j)}_{1} ~\mbox{is bad}\big) + 2 \sum_{j=1}^{k} \P\big(B^{(j)}_{1} ~\mbox{is bad}\big) + \P\big(B^* ~\mbox{is bad}\big)\\
&\le 2 (1+3^d) \sum_{j=1}^{k} \P\big(B^{(j)}_{1} ~\mbox{is bad}\big)\\
&\le 2 (1+3^d) \sum_{j=0}^{\infty} j^{2d} (2\kappa)^{-(d-1)(j+6)}
\le \frac{1}{21}.
\end{align*}
The third inequality follows from Lemma~\ref{lem:bad_block}, which requires $\scp \le p_0$, and this can be satisfied if $\scL$ is large enough and $\eta$ small enough, as already proved in Theorem~\ref{thm:renormalization_flow}. The last inequality is condition $\mathbf{A3}$ in Eq.~\eqref{eq:scales_condition}. 

Finally, concerning the running time of the algorithm, $B_1$ and $B_2$ are separated by distance at most $\frac{2n}{\scL}$. Further, the side length of a block of level $j$, $\ell_j$, grows super-exponentially, hence $B^*$ is of level $k=\bigo(\log n)$. 
At every level $j$, the random variables $(W^{(j)}_{B^{(j)}})$ and $(Y_{B_1,B_2}^{(j)})$ are constructed in time $c_0\ell_j^d$ for some absolute constant $c_0$, and there are $n_j = \big(2n \big/ \scL \cdot \prod_{j'\le j-1} \ell_{j'}\big)^d$ many blocks of level $j$. 
Therefore the variables $(\tilde{W}^{(k)}_{B})_{B \in \bbB_n}$ are constructed in time 
\begin{align*}
\sum_{j=0}^{k} (\ell_j n_j)^d &= \Big(\frac{2n}{\scL}\Big)^d\Big(1+\sum_{j=1}^{k} \prod_{j'=0}^{j-1} \ell_{j'}^{-d} \Big) \\
&\le c_0(d) \Big(\frac{2n}{\scL}\Big)^d.
\end{align*}
The final step of the algorithm is to take pairwise products $\widetilde{T}_{B_1,B_2} = \tilde{W}^{(k)}_{B_1}\tilde{W}^{(k)}_{B_2}$ which takes $\bigo(\big(\frac{2n}{\scL}\big)^{2d})$ steps. Since $\scL$ doesn't depend on $n$, the runtime of the whole algorithm is $\bigo(n^{2d})$.

\section{Posterior measures}
\label{sec:posteriors}
The analysis of the renormalization procedure relies on understanding the concentration properties of the posterior measure of $\theta_A = (\theta_x)_{x \in A}$ given information associated to $A \subset \Z^d$, for various subsets $A$. Concretely we will investigate the asymptotic properties of the `free energies' which are the normalized log-partition functions of the posterior measures and then extract convergence and concentration properties of certain overlaps. These free energies will converge to variational formulas given in terms of a simpler system where the side information is decoupled across vertices. We begin by describing this setting.

\subsection{Scalar side information}
Let $(B_n)_{n \ge 1}$ be a sequence of finite connected subsets of $\Z^d$ which is increasing: $B_{n} \subseteq B_{n+1}$, invades the entire lattice: $\bigcup_{n\ge 1} B_n = \Z^d$, and has vanishing isoperimetry: $\underset{n\to \infty}{\lim} |\partial B_n|/|B_n| = 0$. We call such a sequence a \emph{van Hove} sequence~\cite{friedli2017statistical}.
Consider the scalar additive Gaussian noise channel with SNR $\lambda \ge 0$ on $B_n$:
\begin{equation}\label{eq:gaussian_channel}
y_u = \sqrt{\lambda}\theta_u +z_u, ~~~ \forall u \in B_n,
\end{equation}
where $(z_u)_{u \in B_n}$ are i.i.d.\ Gaussian random variables independent of everything else.
Consider the posterior measure of $\theta_{B_n}$ given the lattice information $Y^{\delta}_{B_n} = \{Y^{\delta}_{uv}: (u,v) \in B_n^2 \cap \mathbb{E}^d\}$ and $Y^{\lambda}_{B_n}= \{y_u: u \in B_n\}$:
\begin{equation}\label{eq:gibbs_Bn_scalar}
\P\Big(\big(\theta_x\big)_{x\in B_n} \big| Y^{\delta}_{B_n}, Y^{\lambda}_{B_n} \Big) = \frac{1}{Z_{B_n}} e^{H_{B_n}^{\delta}(\theta) + H_{B_n}^{\lambda}(\theta)},
\end{equation}
where
\begin{equation}\label{eq:hamiltonian_lattice}
H_{B_n}^{\delta}(\theta) := \beta\sum_{\underset{|u-v|= 1}{u,v\in B_n}} Y^{\delta}_{uv}\theta_u\theta_v,
\end{equation}
with $\beta = \frac{1}{2}\log(\frac{1-p}{p})$ is the Hamiltonian encoding the lattice interaction and
\begin{equation}\label{eq:hamiltonian_scalar}
H_{B_n}^{\lambda}(\theta) = \sum_{u \in B_n} \big(\sqrt{\lambda}y_u\theta_u - \frac{\lambda}{2}\big),
\end{equation}
encodes the side information $Y^{\lambda}_{B_n}$.
We define the free energy  associated to this posterior by
\begin{equation}\label{eq:free_energy_scalar}
\fsc_{B_n}(\delta,\lambda) := \frac{1}{|B_n|} \E\log \Big\{2^{-|B_n|}\sum_{\theta\in \{\pm 1\}^{B_n}} e^{H_{B_n}^{\delta}(\theta) + H_{B_n}^{\lambda}(\theta)}\Big\},
\end{equation}
where the expectation is with respect to $Y^{\delta}_{B_n}$ and $Y^{\lambda}_{B_n}$. 
We also define a more general function with two parameters $s \in \R$ and $r \ge 0$:
\begin{equation}\label{eq:free_energy_scalar_general}
\widetilde{\fsc_{B_n}}(\delta,s,r) := \frac{1}{|B_n|} \E\log \Big\{2^{-|B_n|}\sum_{\theta\in \{\pm 1\}^{B_n}} e^{H_{B_n}^{\delta}(\theta) + {\widetilde{H}}^{r,s}_{B_n}(\theta)}\Big\},
\end{equation}
where 
\[{\widetilde{H}_{B_n}}^{r,s}(\theta) :=  \sum_{u \in B_n} \big(\sqrt{r}z_u\theta_u + s\theta_u\theta_{0u} - \frac{r}{2}\big),\]
where $(z_u)_{u \in B_n}$ are i.i.d.\ $N(0,1)$ r.v.'s and $(\theta_{0u})_{u \in B_n}$ are independent symmetric signs. Note that $\widetilde{\fsc_{B_n}}(\delta,r,r) = \fsc_{B_n}(\delta,r)$ for all $r \ge 0$\footnote{When $ s\neq r$, the Gibbs measure proportional to $e^{H_{B_n}^{\delta}(\theta) + {\widetilde{H}}^{r,s}_{B_n}(\theta)}$ doe not have a Bayesian interpretation as the posterior distribution of a set of random variables given some sort of observations.}. 

\begin{proposition} \label{prop:limit_Bn_scalar}
For all $\delta \in [0,1]$, $s\in\R$ and $r \ge 0$, $\widetilde{\fsc_{B_n}}(\delta,s,r)$ has a limit $\widetilde{\fsc}(\delta,s,r)$ as $n \to \infty$ which does not depend on the van Hove sequence $(B_n)$. Moreover the convergence is uniform on any compact set $[0, \delta_0] \times [-s_0,s_0]\times [0,r_0]$ with $s_0,r_0>0$ and $\delta_0<1$.
\end{proposition}
We let $\fsc(\delta,\lambda):=\widetilde{\fsc}(\delta,\lambda,\lambda)$ be the limit of $\fsc_{B_n}(\delta,\lambda)$. 
We do not know of an `explicit' expression for either $\widetilde{\fsc}$ or $\fsc$, but some of their properties are revealed through the sequence $\widetilde{\fsc_{B_n}}$. For instance, convexity and monotonicity are inherited by passage to the limit:
\begin{lemma}\label{lem:convexity}
For all $\delta \in [0,1]$ and $r \ge 0$, the map $s \mapsto \widetilde{\fsc}(\delta,s,r)$ is even and convex, and the map $\lambda \mapsto \fsc(\delta,\lambda)$ is nondecreasing and convex.
\end{lemma} 
Hence the map $\lambda \mapsto \fsc(\delta,\lambda)$ is differentiable everywhere expect on a countable set of points $\mathcal{D}$ possibly depending on $\delta$.  
Lemma~\ref{lem:convexity} implies that the derivatives $\frac{\rmd}{\rmd \lambda}\fsc_{B_n}$ converges to $\frac{\rmd}{\rmd \lambda}\fsc$ on $\R_+ \setminus \mathcal{D}$.  
An easy computation of the derivative of $f_{B_n}$ reveals that for all $\delta \in [0,1]$ and $\lambda \in \R_+ \setminus \mathcal{D}$, 
\begin{equation}\label{eq:overlap_scalar}
\frac{1}{2}\varphi^{\sv}_{B_n}:= \frac{1}{2|B_n|}\sum_{x\in B_n} \E\Big[\E\big[\theta_x | Y^{\delta}_{B_n}, Y^{\lambda}_{B_n}\big]^2\Big] \xrightarrow[n \to \infty]{} \frac{\rmd}{\rmd \lambda}\fsc.
\end{equation}

 Lemma~\ref{lem:convexity} also implies that $s \mapsto \widetilde{\fsc}(\delta,s,r)$ is differentiable on $\R \setminus \mathcal{D}'$ where $\mathcal{D}'$ is countable and depends on $\delta$ and $r$.
For technical reasons we will need the following regularity condition on $\fsc$:   
\begin{assumption}\label{assump:lipschitz} 
We assume the following.
\begin{itemize}
\item[A1.] For all $\delta>\delta_c$, the map $\lambda \mapsto \frac{\rmd}{\rmd \lambda}\fsc(\delta, \lambda)$ is Lipschitz on $\R_+\setminus \mathcal{D}$ with Lipschitz constant $L = L(\delta)<\infty$.
\item[A2.] For all $\delta>\delta_c$ and all $r \ge 0$, the map $s \mapsto \frac{\rmd}{\rmd s}\widetilde{\fsc}(\delta, s,r)$  is Lipschitz on $\R\setminus \mathcal{D}'$ uniformly in $r$ with Lipschitz constant $L' = L'(\delta)<\infty$. 
\end{itemize}
\end{assumption}
 These assumptions are highly plausible and are supported by physical intuition: we do not expect these model to have a phase transition in the strength $s$ of the magnetic field for $\delta > \delta_c$. For instance when $r=s = \lambda$, the scalar side information has the effect of destroying long range correlation, see e.g.,~\cite{MontanariSparse}, and every spin $\theta_x$ gains a small and independent bias towards the ground truth assignment when $\lambda>0$. This bias must increase in a smooth way as $\lambda$ increases, and the free energy should be analytic in $\lambda$.
 This assumption will be explicitly mentioned whenever needed, and not required if not invoked.

\subsection{GOE side information}
\label{sec:GOEside}
The side information  we are ultimately interested in is the spiked GOE $Y^{\eta}_{B_n} =\{Y^{\eta}_{uv}: u,v \in B_n\}$ where $Y^{\eta}_{uv} = \sqrt{\frac{\eta}{|B_n|}}\theta_u \theta_v  + Z_{uv}$, $Z_{uv} \sim N(0,1)$ independent of everything else. The posterior measure is
\begin{equation}\label{eq:gibbs_Bn}
\P\Big(\big(\theta_x\big)_{x\in B_n} \big| Y^{\delta}_{B_n}, Y^{\eta}_{B_n} \Big) = \frac{1}{Z_{B_n}}   e^{H_{B_n}^{\delta}(\theta) + H_{B_n}^{\eta}(\theta)},
\end{equation}
where
\begin{equation}
\label{eq:hamiltonian_goe}
H_{B_n}^{\eta}(\theta) := \sum_{u,v\in B_n} \Big(\sqrt{\frac{\eta}{|B_n|}} Y_{uv}^{\eta}\theta_{u}\theta_{v} - \frac{\eta}{2|B_n|} \Big).
\end{equation}
We similarly define the free energy 
\begin{equation}\label{eq:free_energy_Bn}
f_{B_n}(\delta,\eta) := \frac{1}{|B_n|} \E\log \Big\{2^{-|B_n|}\sum_{\theta\in \{\pm 1\}^{B_n}} e^{H_{B_n}^{\delta}(\theta) + H_{B_n}^{\eta}(\theta)}\Big\}.
\end{equation}
\begin{proposition} \label{prop:limit_Bn_GOE}
For all $\delta \in [0,1]$, $\eta \ge 0$, $f_{B_n}(\delta,\eta)$ has a limit $\phi(\delta,\eta)$ as $n \to \infty$ which does not depends on the van Hove sequence $(B_n)$, and we have the relation 
\begin{equation}\label{eq:var_formula_Bn}
\phi(\delta,\eta) = \sup_{q \ge 0} \Big\{ \fsc(\delta,\eta q) - \frac{\eta q^2}{4} \Big\}.
\end{equation}
\end{proposition}
Similarly to the scalar side information case, $\frac{\rmd}{\rmd \eta}f_{B_n}$ converges to $\frac{\rmd}{\rmd \eta}\phi$ everywhere except on countably many points. On the one hand, the derivative of $f_{B_n}$ is
\begin{equation}\label{eq:derivative_f_Bn_eta}
\frac{\rmd}{\rmd \eta}f_{B_n}=\frac{1}{4|B_n|^2}\sum_{x,y\in B_n} \E\Big[\E\big[\theta_x \theta_y| Y^{\delta}_{B_n}, Y^{\eta}_{B_n}\big]^2\Big] = \frac{1}{4}\varphi^{\se}_{B_n}.
\end{equation}
On the other hand, the derivative of $\phi$ can be linked to the maximizer in the variational formula~\eqref{eq:var_formula_Bn}. This proves Proposition~\ref{prop:conv_pair_corr} and provides a another characterization of $q_{\star}$: 
\begin{lemma}\label{lem:derivative_phi}
For all $\delta\in[0,1]$ and all $\eta$ where $\phi(\delta,\cdot)$ is differentiable, the above maximization problem has a unique maximizer $q_{\star} \in [0,1]$, and $\frac{\rmd}{\rmd \eta}\phi(\delta,\lambda) = q_{\star}^2/4$. Therefore, for all $\delta$ and all except countably many $\eta$, we obtain the relation 
$\varphi^{\se}_{B_n}  \to q_{\star}^2$ as $n \to \infty$.
\end{lemma}
\begin{proof}
This argument appears in~\cite{lelarge2017fundamental} but we reproduce it here for the sake of completeness. First, it is easy to see that the supremum in~\eqref{eq:var_formula_Bn} is achieved. Indeed $\fsc$ being convex, it is also continuous, and one has the elementary bound $\fsc(\delta,\lambda) \le 2\beta d +\lambda$. Therefore, large values of $q$ cannot achieve the supremum. Next, we reparametrize the variational formula as $\phi(\delta,\eta) = \sup_{\bar{q} \ge 0}\{\fsc(\delta, \bar{q}) - \frac{\bar{q}^2}{4\eta}\}$ with $\bar{q} = \eta q$. 
The envelope theorem implies that on points of differentiability of $\phi$, we have $\frac{\rmd }{\rmd \eta}\phi = \bar{q}^2/4\eta^2$ where $\bar{q}$ is any point achieving the maximum. Therefore the maximum is achieved at a unique point $q_{\star}$ and we have $\frac{\rmd}{\rmd \eta}\phi(\delta,\lambda) = q_{\star}^2/4$. 
Moreover from~\eqref{eq:derivative_f_Bn_eta}, the fact that  $\frac{\rmd}{\rmd \eta}f_{B_n}$ converges to $\frac{\rmd}{\rmd \eta}\phi$ whenever the derivative is defined implies the last claim $\varphi^{\se}_{B_n}  \to q_{\star}^2$. 
\end{proof}

\subsection{One-block and two-block posteriors}
\label{sec:one_two_block}
In this subsection we analyze the \emph{one-block} and \emph{two-block} posteriors: these are the posterior measures of $(\theta_x)_{x\in B}$ and $(\theta_x)_{x\in B \cup B'}$ for $B \sim B'$ respectively, where the blocks and the associated side information were constructed in Section~\ref{sec:algorithm}. Since the strength of the GOE side information is inhomogeneous within a given block, the posteriors and their free energies are not special cases of the setting discussed above. Nevertheless the results and arguments will have a similar flavor.       

The one-block posterior is given by
\begin{equation}\label{eq:gibbs_one_block}
\P\Big(\big(\theta_x\big)_{x\in B} \big| Y^{\delta}_B , Y^{\eta}_B\Big) = \frac{1}{Z_B}  \omega^{\sSK}_B(\theta) \cdot e^{H_{B}^{\delta}(\theta)},
\end{equation}
where as before, $H^{\delta}_B$ is the Hamiltonian encoding the lattice interaction on $B$, and  
\begin{equation}\label{eq:boltzmann_one_block}
\omega^{\sSK}_B(\theta) :=  \prod_{B' \sim B} e^{H^{\eta}_{B'}(\theta)}.
\end{equation}
is the Boltzmann weight encoding the heterogeneous GOE side information~\eqref{eq:processed_side_info}:
For $B' = B+\scL a$, $a \in \{\pm e_1,\cdots,\pm e_d\}$, and for $\theta \in \{\pm 1\}^B$,
\begin{align}\label{eq:hamiltonian_goe_inhomogeneous}
H^{\eta}_{B'}(\theta) &:= \sum_{u,v\in B} \Big(\sqrt{\frac{t \eta}{|B|}} Y_{uv}^{\eta,(a)}\theta_{u}\theta_{v} - \frac{t \eta}{2|B|}\Big) \\
&~+\sum_{u,v\in B \cap B'} \Big(\sqrt{\frac{(1-t)\eta}{|B\cap B'|}} Y_{uv}^{\eta,(a,\cap)}\theta_{u}\theta_{v} - \frac{(1-t)\eta}{2|B\cap B'|}\Big) \nonumber\\
&~+\sum_{u,v\in B \setminus B'} \Big(\sqrt{\frac{(1-t)\eta}{|B\setminus B'|}} Y_{uv}^{\eta,(a,\setminus)}\theta_{u}\theta_{v} - \frac{(1-t)\eta}{2|B\setminus B'|}\Big).\nonumber
\end{align}
Similarly, for two adjacent blocks $B \sim B'$, the two-block posterior is
\begin{equation}\label{eq:gibbs_two_block}
\P\Big(\big(\theta_x\big)_{x\in B \cup B'} \big| \bigcup_{A \in \{B,B'\}} \{Y^{\delta}_A , Y^{\eta}_A \}\Big) = \frac{1}{Z_{B\cup B'}} \omega_{B \cup B'}^{\sSK}(\theta) \cdot e^{H_{B\cup B'}^{\delta}(\theta)},
\end{equation} 
where 
for $\theta \in \{\pm 1\}^{B \cup B'}$, denoting by $\theta_B$ and $\theta_{B'}$ the restrictions of $\theta$ to $B$ and $B'$ respectively, 
\begin{equation}\label{eq:boltzmann_two_block}
\omega_{B\cup B'}^{\sSK}(\theta) := \prod_{B'' \sim B} e^{H^{\eta}_{B''}(\theta_{B})} 
\cdot \prod_{B'' \sim B'} e^{H^{\eta}_{B''}(\theta_{B'})}.
\end{equation}
 We define the free energies of the above posteriors as 
 \begin{equation}\label{eq:free_energies}
  f_{B}(\delta,\eta,t) := \frac{1}{|B|} \E \log Z_{B},\quad \mbox{and} \quad f_{B\cup B'}(\delta,\eta,t) := \frac{1}{|B|} \E \log Z_{B \cup B'},
 \end{equation}
where $Z_B$ and $Z_{B \cup B'}$ are the normalizing constant of the one-block and two-block posteriors, respectively. 
Let $\alpha := \underset{\scL \to \infty}{\lim} \frac{|B \cap B'|}{|B|}=\frac{1}{3^d+d}$.

\begin{proposition}\label{prop:limits_phi1_phi2}
Under Assumption~\ref{assump:lipschitz}, there exists $\eta_0 = \eta_0(d,\delta)>0$ such that if $\eta \in \R_+ \setminus \mathcal{D}$ and $\eta\le \eta_0$ then the following holds. For all $B, B' \in \bbB$, with $B \sim B'$, and all $\delta> \delta_c$, and $t\in [0,1]$, the free energies $f_{B}$ and $f_{B \cup B'}$ converge to limits $\phi_1$ and $\phi_2$, respectively, as $\scL\to \infty$. Moreover, $\phi_1$ and $\phi_2$ admit the following variational representations  
\begin{align}
\phi_1 &= \sup_{q \ge 0} \Big\{\fsc(\delta;2 d \eta q) - \frac{d}{2} \eta q^2\Big\} = \phi(\delta,2d\eta),\label{eq:var_1}\\
\phi_2 &= \sup_{q \ge 0} \Big\{2(1-\alpha)\fsc(\delta;2 d \eta q) +\alpha \fsc(\delta;4 d \eta q) - d \eta q^2\Big\}.\label{eq:var_2}
\end{align}
In particular $\phi_1$ and $\phi_2$ do not depend on the time parameter $t$ when $\eta\le \eta_0$.
\end{proposition}

Independence with respect to $t$ implies a fundamental \emph{locking} property of the overlaps: under either the one-block or the two-block posterior measures, overlaps of two independent samples on different subsets of vertices of the whole system are almost equal. 
To state this property, we adopt the following the notation.  For $A \in \{B , B \cup B'\}$ define
\[ V_{A}(B;B') := \alpha  \E \left\langle \big(R_{1,2}(B \cap B') - R_{1,2}(B)\big)^2\right\rangle_{A}
+(1-\alpha) \E \left\langle \big(R_{1,2}( B \setminus B') - R_{1,2}(B)\big)^2\right\rangle_{A},\]
 where the Gibbs average $\langle \cdot \rangle_A$ is a short-hand for the average w.r.t\ the one-block posterior in the case $A = B$ or the two-block posterior in the case $A = B \cup B'$, and $R_{1,2}(A) = \frac{1}{|A|}\sum_{x \in A} \theta^1_x \theta^2_x$ where $\theta^1,\theta^2$ is a pair of independently drawn samples from the one/two-block posterior measure. 
\begin{proposition}\label{prop:locking} 
Under Assumption~\ref{assump:lipschitz}, there exists $\eta_0 = \eta_0(d,\delta)>0$ such that for $\eta\le \eta_0$ almost all $t \in [0,1]$, we have $V_{B}(B,B') \to 0$ for all $B' \sim B$, $V_{B\cup B'}(B,B'')\to 0$ for all $B'' \sim B$, and $V_{B\cup B'}(B',B'')\to 0$ for all $B'' \sim B'$ as $\scL \to \infty$.
\end{proposition}
\begin{proof}
Convexity of $f_{B}$ and $f_{B \cup B'}$ w.r.t.\ $t$ implies that their derivatives converge to the derivatives of the limits $\phi_1$ and $\phi_2$ everywhere except on countably many points.  
A short computation of the finite-volume derivatives reveals 
 \begin{align*}
 \frac{\rmd }{\rmd t}f_{B} &= -\frac{\eta}{2} \sum_{B' \sim B} V_{B}(B;B'), \qquad \mbox{and}\\
 \frac{\rmd }{\rmd t}f_{B\cup B'} &= -\frac{\eta}{2} \sum_{B'' \sim B} V_{B \cup B'}(B,B'')  -\frac{\eta}{2}\sum_{B'' \sim B'} V_{B \cup B'}(B';B''). 
\end{align*}  
Since by Proposition~\ref{prop:limits_phi1_phi2}, the limits do not depend on $t$ for $\eta \le \eta_0$, the above right-hand sides converge to zero.
\end{proof}
Now that overlaps of sub-blocks of a given block are locked together, they are all asymptotically equal to the overlap on the entire block ($B$ in the case of the one-block posterior and $B \cup B'$ in the case of the two-block posterior). The limits of these global overlaps are again available by differentiating w.r.t.\ $\eta$. Similarly to $V_A(B;B')$, we define
 \[W_A(B;B') :=  t  \E \left\langle R_{1,2}(B)^2\right\rangle_{A} + 
 (1-t) \Big(\alpha \E \left\langle R_{1,2}(B \cap B')^2\right\rangle_{A} + (1-\alpha) \E \left\langle R_{1,2}(B \setminus B')^2\right\rangle_{A}\Big).\] 
 Thus we have
  \begin{align*}
 \frac{\rmd }{\rmd \eta}f_{B} &= \frac{1}{4} \sum_{B' \sim B}W_B(B;B'), \qquad \mbox{and}\\
 \frac{\rmd }{\rmd \eta}f_{B\cup B'} &=  \frac{1}{4} \sum_{B'' \sim B} W_{B\cup B'}(B;B'') + \frac{1}{4}  \sum_{B'' \sim B'} W_{B\cup B'}(B';B'').
\end{align*}
We deduce from Proposition~\ref{prop:locking} that 
\[\frac{\rmd }{\rmd \eta}f_{B} = \frac{d}{2}  \E \left\langle R_{1,2}(B)^2\right\rangle_{B} + o_{\scL}(1), ~~ \mbox{and} ~~ \frac{\rmd }{\rmd \eta}f_{B\cup B'} = d \E \left\langle R_{1,2}(B \cup B')^2\right\rangle_{B\cup B'} + o_{\scL}(1),\] 
where $o_{\scL}(1) \to 0$ for almost all $t$ and all $\eta\le\eta_0$. Hence the following
\begin{corollary}\label{cor:overlap_conv}  
Under Assumption~\ref{assump:lipschitz}, let $q^{\bullet}_1$ and $q^{\bullet}_2$ be the unique maximizers in the variational problems~\eqref{eq:var_1} and~\eqref{eq:var_2} respectively. For almost all $t \in [0,1]$ and all except countably many $\eta \in [0,\eta_0]$, we have $\E \left\langle R_{1,2}(B)^2\right\rangle_{B} \to {q^{\bullet}_1}^2$ and $ \E \left\langle R_{1,2}(B \cup B')^2\right\rangle_{B\cup B'} \to {q^{\bullet}_2}^2$ as $\scL \to \infty$.
\end{corollary}
\begin{proof}
From~\eqref{eq:var_1} and Lemma~\ref{lem:derivative_phi} we have $ \frac{\rmd }{\rmd \eta} \phi_1 = d {q^{\bullet}_1}^2/2$.  
Second, note that the same argument used to obtain Lemma~\ref{lem:derivative_phi} applies to the formula~\eqref{eq:var_2}, hence the existence and uniqueness of the maximizer $q^{\bullet}_2$, and $\frac{\rmd}{\rmd \lambda}\phi_2= d{q^{\bullet}_2}^2$. Now, given formulas above for the derivatives of $f_{B}$ and $f_{B\cup B'}$,  and since these derivatives converge to $\frac{\rmd }{\rmd \eta} \phi_1$ and $\frac{\rmd }{\rmd \eta} \phi_2$ respectively as $\scL \to \infty$ for all except countably many $\eta>0$, we obtain the desired result.
\end{proof}
Next, we show that the maximizers $q_{\star}$, $q^{\bullet}_1$ and $q^{\bullet}_2$ of the variational formulas~\eqref{eq:var_formula_Bn}, \eqref{eq:var_1} and \eqref{eq:var_2} respectively are close together when $\eta$ is small.  
\begin{lemma} \label{lem:maximizers_are_close}
Under Assumption~\ref{assump:lipschitz}, and for $\delta>\delta_c$, there exist a constant $c_0 = c_0(d,L,\eta_0)>0$ such that for all except countably many $\eta \in [0,\eta_0]$, $|q^{\bullet}_2 - q^{\bullet}_1| \vee |q^{\bullet}_1-q_{\star}| \le c_0\eta$. 
\end{lemma}
\begin{proof}
Since $\delta>\delta_c$ all the maximizers are strictly positive and must satisfy first order optimality conditions:
\[{\fsc}'(\delta;\eta q_{\star}) = \frac{q_{\star}}{2}, ~~ {\fsc}'(\delta,2d\eta q^{\bullet}_1) = \frac{q^{\bullet}_1}{2}, ~~ (1-\alpha){\fsc}'(\delta,2d\eta q^{\bullet}_2) + \alpha {\fsc}'(\delta,4d\eta q^{\bullet}_2) = \frac{q^{\bullet}_2}{2}.\]   
We take the difference of the first two equations and use the assumed Lipschitz property of ${\fsc}'$ to obtain $|q^{\bullet}_1-q_{\star}| \le 2L \eta |2d q^{\bullet}_1-q_{\star}|$, or $|q^{\bullet}_1-q_{\star}| \le \frac{2L\eta}{1-2L\eta}(2d-1)q_{\star}$. And we use the trivial bound $q_{\star}\le 1$ (this is because ${\fsc}'(\delta,\lambda) \le 1$). Similarly, taking the difference of the last two equations we obtain $|q^{\bullet}_2-q^{\bullet}_1| \le \frac{4L\alpha\eta}{1-4L\eta}q^{\bullet}_2$.
\end{proof}
 Corollary~\ref{cor:overlap_conv} and Lemma~\ref{lem:maximizers_are_close}  show that the squared global overlaps $\E \left\langle R_{1,2}(B)^2\right\rangle_{B}$ and $\E \left\langle R_{1,2}(B\cup B')^2\right\rangle_{B\cup B'}$ converge to the same value $q_{\star \star}^2(\delta)$ as $\eta\to 0$.

\subsection{Decoupling bounds}
\label{sec:decoup}
It will be useful for technical reasons to consider a model where one receives side information from the scalar additive Gaussian noise channel~\eqref{eq:gaussian_channel} in addition to the GOE side information. Let us consider the one-block and two-block posteriors, similarly to~\eqref{eq:gibbs_one_block} and~\eqref{eq:gibbs_two_block}, with this additional scalar side information. These probabilities are respectively proportional to
\begin{align*}
\omega^{\sSK}_B(\theta) \cdot e^{H_{B}^{\delta}(\theta)+ H_{B}^{\lambda}(\theta)}, \mbox{ for } \theta \in B ~\mbox{ and }~ \omega^{\sSK}_{B\cup B'}(\theta) \cdot e^{H_{B \cup B'}^{\delta}(\theta)+H_{B \cup B'}^{\lambda}(\theta)}, \mbox{ for } \theta \in B \cup B',
\end{align*}   
where $\omega^{\sSK}_{B}$ is defined in~\eqref{eq:boltzmann_one_block} and $\omega^{\sSK}_{B\cup B'}$ in~\eqref{eq:boltzmann_two_block}, and $H_{B}^{\lambda}(\theta)$ and $H_{B \cup B'}^{\lambda}(\theta)$ are defined in~\eqref{eq:hamiltonian_scalar}. We state and prove a sequence of technical lemmas which will be extensively used in the rest of the paper. 
The starting point is the following extension of Proposition~\ref{prop:limits_phi1_phi2}:
\begin{proposition} \label{prop:additional_scalar}
Let $f_{B}(\delta,\eta,t,\lambda)$ and $f_{B\cup B'}(\delta,\eta,t,\lambda)$ be the free energies of the above posterior measures (defined similarly to~\eqref{eq:free_energies}). Then for all $\delta \in [0,1]$, $\eta \ge 0$, $t \in [0,1]$ and $\lambda \ge 0$, $f_{B}$ and $f_{B \cup B'}$ converge to $\phi_1$ and $\phi_2$, respectively, as $\scL\to \infty$.
Moreover, under Assumption~\ref{assump:lipschitz}, there exists $\eta_0 = \eta_0(d,\delta)>0$ such that if $\eta\le \eta_0$ then $\phi_1$ and $\phi_2$ admit the following variational representations  
\begin{align}
\phi_1 &= \sup_{q \ge 0} \Big\{\fsc(\delta;2 d \eta q+\lambda) - \frac{d}{2} \eta q^2\Big\},\label{eq:var_1_scalar}\\
\phi_2 &= \sup_{q \ge 0} \Big\{2(1-\alpha)\fsc(\delta;2 d \eta q+\lambda) +\alpha \fsc(\delta;4 d \eta q+\lambda) - d \eta q^2\Big\}.\label{eq:var_2_scalar}
\end{align}
\end{proposition}
Similarly to previous computations, the vertex- and pair-correlations defined respectively as   
\begin{equation}
\varphi_{A}^{\sv} := \frac{1}{|A|} \sum_{x\in A}  \E \Big[\E\big[\theta_x| Y^{\delta,\eta,\lambda}_{A}\big]^2\Big] 
~~\mbox{and}~~
\varphi_{A}^{\se} := \frac{1}{|A|^2} \sum_{x,y\in A}  \E \Big[\E\big[\theta_x\theta_{y}| Y^{\delta,\eta,\lambda}_{A}\big]^2\Big],\label{eq:one_two_point}
\end{equation}
where $Y^{\delta,\eta,\lambda}_{A} = \{Y^{\delta}_{A},Y^{\eta}_{A},Y^{\lambda}_{A}\}$ and $A$ is either $B$ or $B \cup B'$, have well-defined limits as $\scL \to \infty$ for all except countably many values of $\eta$ and $\lambda$. Note that the introduction of the scalar side information breaks the sign symmetry of the random variables $\theta_x$ under the posterior, and the value of the one-point correlation function $\varphi_{A}^{\sv}$ is no longer trivially zero. Furthermore, these two quantities are related in a simple way in the limit. 
\begin{lemma}\label{lem:edge_vertex_overlap}
Let $A$ be either $B$ or $B \cup B'$. Under Assumption~\ref{assump:lipschitz}, for all $\delta > \delta_c$, all except countably many $\eta \ge 0$ and $\lambda \ge 0$, $\varphi_{A}^{\sv}$ and $\varphi_{A}^{\se}$ have limits and $\underset{\scL \to \infty}{\lim} (\varphi_{A}^{\sv})^2 = \underset{\scL \to \infty}{\lim}  \varphi_{A}^{\se}$. 
\end{lemma}
\begin{proof} 
We only consider the case $A=B$. The case $A = B \cup B'$ can be treated similarly. 
Let $\bar{q}^{\bullet}$ be the unique maximizer in~\eqref{eq:var_1_scalar} (see Lemma~\ref{lem:derivative_phi}).
Convexity in $\lambda$ and $\eta$ imply that $\varphi_{A}^{\sv}$ and $\varphi_{A}^{\se}$ converge to $2 \frac{\rmd }{\rmd x_2}\fsc(\delta,2d \eta \bar{q}^{\bullet}+\lambda)$, and $\bar{q}^{\bullet 2}$ respectively (denoting by $x_2$ the second variable of $\fsc$). 
When $\delta>\delta_c$, $\bar{q}^{\bullet}\ge q_{\star\star}(\delta)>0$. Therefore, by optimality of $\bar{q}^{\bullet}$ we have
\[\frac{\rmd }{\rmd x_2}\fsc(\delta,2d \eta\bar{q}^{\bullet}+\lambda) = \frac{\bar{q}^{\bullet}}{2},\]
and this proves the claim.
\end{proof}   
Next, we exploit information contained in the second derivatives of the free energies. 
\begin{lemma}\label{lem:decoupling_scalar}
Let $A$ be $B$ or $B \cup B'$. For almost all $\lambda\ge 0$,
\[\lim_{\scL \to \infty} \frac{1}{|A|^2}\sum_{x,y \in A} \E \Big[\Big(\E\big[\theta_x\theta_y| Y^{\delta,\eta,\lambda}_{A}\big]-\E\big[\theta_x| Y^{\delta,\eta,\lambda}_{A}\big] \E\big[\theta_y| Y^{\delta,\eta,\lambda}_{A}\big]\Big)^2\Big] = 0.\]
\end{lemma}
\begin{proof}
Call the quantity in the above display $R$. Taking two derivatives of the free energy $f_A$ with respect to $\lambda$, we obtain 
\[ \frac{\rmd^2 }{\rmd \lambda^2} f_A = \frac{1}{4|A|}\sum_{x,y \in A} \E \Big[\Big(\E\big[\theta_x\theta_y| Y^{\delta,\eta,\lambda}_{A}\big]-\E\big[\theta_x| Y^{\delta,\eta,\lambda}_{A}\big] \E\big[\theta_y| Y^{\delta,\eta,\lambda}_{A}\big]\Big)^2\Big] = \frac{|A|}{4} R.\]
Moreover, the first derivative is uniformly bounded: $\frac{\rmd }{\rmd \lambda} f_A\le \frac{1}{2}$, so $\int_0^\lambda \big(\frac{\rmd^2 }{\rmd \lambda^2} f_A\big) \rmd \lambda' \le \frac{\lambda}{2}$. Therefore $\int_0^\lambda R \rmd \lambda' \le 2/|A| \to 0$ as $\scL \to \infty$.  Since $R$ is non-negative, it must converge to $0$ for almost all values of $\lambda$. 
\end{proof} 
 Lemmas~\ref{lem:edge_vertex_overlap} and~\ref{lem:decoupling_scalar} provide averaged forms of decoupling of the variables at different vertices. Next, we show a form of continuity in $\lambda$ at zero for the pair correlations:    
\begin{lemma}\label{lem:removing_scalar}
Let $A$ be $B$ or $B \cup B'$. For almost all $\eta \ge 0$,
\[\lim_{\lambda\to 0} \lim_{\scL \to \infty} \frac{1}{|A|^2}\sum_{x,y \in A} \E \Big[\Big(\E\big[\theta_x\theta_y| Y^{\delta,\eta,\lambda}_{A}\big]-\E\big[\theta_x\theta_y| Y^{\delta,\eta}_{A}\big]\Big)^2\Big] = 0.\]
\end{lemma}
\begin{proof}
Call $R$ the quantity whose double limit is taken in the above display. Since $Y_{A}^{\delta,\eta,\lambda}$ contains more information than $Y_{A}^{\delta,\eta}$, we have 
\begin{align*}
R &= \frac{1}{|A|^2}\sum_{x,y \in A} \Big\{\E \Big[\E\big[\theta_x\theta_y| Y^{\delta,\eta,\lambda}_{A}\big]^2\Big]- \E \Big[\E\big[\theta_x\theta_y| Y^{\delta,\eta}_{A}\big]^2\Big]\Big\}\\
&= \varphi_{A}^{\se}(\delta,\eta,\lambda) - \varphi_{A}^{\se}(\delta,\eta,0).
\end{align*}
Since $\varphi_{A}^{\se} = 4\frac{\rmd}{\rmd \eta} f_A$, we have 
\[\int_0^{\eta} R \rmd \eta' \le 4\big|f_A(\delta,\eta,t,\lambda) -f_A(\delta,\eta,t,0)\big| + 4\big|f_A(\delta,0,t,\lambda) -f_A(\delta,0,t,0)\big|.\]
Since $\lambda \mapsto f_A(\delta,\eta,t,\lambda)$ is $L$-Lipschitz with $L = \frac{1}{2}$ uniformly in the other variables (recall that the derivative is uniformly bounded by $\frac{1}{2}$), we arrive at $\int_0^{\eta} R \rmd \eta' \le 4\lambda$. Since $R$ is non-negative and has a limit as $\scL \to \infty$, Fatou's lemma implies
\[\int_0^{\eta} \lim_{\scL \to \infty} R~ \rmd \eta' \le 4\lambda.\] 
Sending $\lambda$ to zero finishes the proof.
\end{proof}
Finally, we prove another form of decoupling which allows to split correlations of quadruples of vertices into the product of pair correlations:  
\begin{lemma}\label{lem:decoupling_variance}
Let $A$ be $B$ or $B \cup B'$. For all $\lambda \ge 0$ and almost all $\eta\ge 0$,
\[\lim_{\scL \to \infty} \frac{1}{|A|^4} \sum_{x,y,w,z \in A} \E \Big[\Big(\E\big[\theta_x\theta_y\theta_w\theta_z| Y^{\delta,\eta,\lambda}_{A}\big]-\E\big[\theta_x\theta_y| Y^{\delta,\eta,\lambda}_{A}\big]\E\big[\theta_w\theta_z| Y^{\delta,\eta,\lambda}_{A}\big]\Big)^2\Big]=0.\]
\end{lemma}
\begin{proof}
This time we compute the second derivative of $f_A$ with respect to $\eta$:
\[ \frac{\rmd^2 }{\rmd \eta^2} f_A = \frac{1}{8|A|}\sum_{x,y,w,z \in A} \E \Big[\Big(\E\big[\theta_x\theta_y\theta_w\theta_z| Y^{\delta,\eta,\lambda}_{A}\big]-\E\big[\theta_x\theta_y| Y^{\delta,\eta,\lambda}_{A}\big] \E\big[\theta_w\theta_z| Y^{\delta,\eta,\lambda}_{A}\big]\Big)^2\Big].\]
We conclude in the same way as Lemma~\ref{lem:decoupling_scalar}, by noting that the first derivative $\frac{\rmd }{\rmd \eta} f_A$ is also uniformly bounded (by $\frac{1}{4}$).
\end{proof}

\section{Analysis of correlations: proof of Theorems~\ref{thm:everything} and~\ref{thm:lower_bound}}
\label{sec:proof_everything}
In this section we prove Theorem~\ref{thm:everything}. We first prove items 1.\ and 2., then we turn to item 3. 

\subsection{Proof of item 1: analysis of $M_B$ and $M_{B^{\bullet}}$}
We first show convergence of the expectation of $M_{B}^2$.  Recall that 
\[M_B =  \frac{1}{|B|} \sum_{x\in B}\theta^{B}_x \theta_x,\] 
where $\big(\theta^{B}_x\big)_{x\in B} \sim \P(\cdot | Y_B )$ is sampled from the one-block posterior on $B$, with $Y_B = \{Y^{\delta}_B , Y^{\eta}_B\}$. ($M_{B^{\bullet}}$ is defined similarly.) Then
\begin{align*}
\E[M_{B}^2] &= \frac{1}{|B|^2}\E \Big[\E\Big[\Big(\sum_{x\in B} \theta^{B}_x\theta_x\Big)^2 \Big| Y_B\Big]\Big] \\
&=  \frac{1}{|B|^2}\sum_{x,y \in B }  \E\big[\E[\theta_x\theta_y |Y_{B}]^2\big]\\
&=\E \left\langle R_{1,2}(B)^2\right\rangle_{B}.
\end{align*}
Form Corollary~\ref{cor:overlap_conv} and Lemma~\ref{lem:maximizers_are_close} the above converges to $q_{\star\star}^2(\delta)$ for almost all $t \in [0,1]$ as $\scL \to \infty$ followed by $\eta \to 0$.

Next we show that $\E\big[(M_B - M_{B^{\bullet}})^2\big]\to 0$.  Recall that $\alpha = \frac{|J|}{|B|}$ where $J$ is a joint, and let $M_{B\setminus B'} := \frac{1}{|B\setminus B'|}\sum_{x \in B\setminus B'}\theta^{B}_x \theta_x$.
 We have the two relations
 \begin{align*}
 M_{B} &= (1-2\alpha d) M_{B^{\bullet}} + \alpha\sum_{B' \sim B} M_{B \cap B'},\\
 M_{B\setminus B'} &= \frac{1-2\alpha d}{1-\alpha} M_{B^{\bullet}} + \frac{\alpha}{1-\alpha}\sum_{\underset{B'' \neq B'}{B'' \sim B}} M_{B \cap B''}.
 \end{align*}
 So
 \begin{align*}  
 M_B - M_{B^{\bullet}} &= \alpha \sum_{B' \sim B} (M_{B\cap B'} - M_{B^{\bullet}}),\\
 M_B - \frac{1}{2d}\sum_{B' \sim B}M_{B\setminus B'} &= \frac{\alpha(1-2\alpha d)}{2d(1-\alpha)} \sum_{B' \sim B} (M_{B}- M_{B\cap B'}).
  \end{align*}
  Therefore
 \begin{align*} 
 \E\big[(M_B - M_{B^{\bullet}})^2\big] &= \alpha^2 \E  \Big[\big(\sum_{B' \sim B} (M_{B\cap B'} - M_{B^{\bullet}}\big)^2\Big]\\
 &= c(d,\alpha) \E  \Big[\big( M_{B} - \frac{1}{2d}\sum_{B' \sim B} M_{B \setminus B'}\big)^2\Big]\\
 &\le \frac{c(d,\alpha)}{2d}   \sum_{B' \sim B} \E  \big[( M_{B} - M_{B \setminus B'})^2\big].
  \end{align*}   
Proposition~\ref{prop:locking} implies that the above converges to zero for almost all $t \in [0,1]$ and all $\eta < \eta_0$.  
  
Next, we show that $M_{B}^2$ has vanishing variance.We have 
 \begin{align*}
 \var(M_{B}^2) &= \E[M_{B}^4] - \E[M_{B}^2]^2\\
 &= \frac{1}{|B|^4}\sum_{x,y,z,w \in B} \E\big[\E [\theta_x\theta_y \theta_z\theta_w|Y_{B}]^2\big]-\E[M_{B}^2]^2.
 \end{align*}
 The second term in the above display converges to $q_{\star\star}^4(\delta)$. It remains to study the fourth moment $\E[M_{B}^4]$.
 Using the decoupling Lemma~\ref{lem:decoupling_variance}, we have
 \[\frac{1}{|B|^4}\sum_{x,y,z,w \in B} \E\Big[\E \big[\theta_x\theta_y \theta_z\theta_w|Y^{\delta,\eta}_{B}]^2\big] = \frac{1}{|B|^4}\sum_{x,y,z,w \in B} \E\Big[\E\big[\theta_x\theta_y|Y^{\delta,\eta}_{B}\big]^2\E\big[\theta_z\theta_w|Y^{\delta,\eta}_{B}\big]^2\Big] +o_{\scL}(1),\]     
 where $o_{\scL}(1) \to 0$ for almost every $\eta>0$ as $\scL \to \infty$. 
 Now we introduce side information from the scalar additive Gaussian noise channel~\eqref{eq:gaussian_channel} using Lemma~\ref{lem:removing_scalar} (with $A=B$):
 \begin{align*}
 \frac{1}{|B|^4}&\sum_{x,y,z,w \in B} \E\Big[\E \big[\theta_x\theta_y|Y^{\delta,\eta}_{B}\big]^2\E\big[\theta_z\theta_w|Y^{\delta,\eta}_{B}\big]^2\Big]\\
& =\frac{1}{|B|^4}\sum_{x,y,z,w \in B} \E\Big[\E \big[\theta_x\theta_y|Y^{\delta,\eta,\lambda}_{B}\big]^2\E\big[\theta_z\theta_w|Y^{\delta,\eta,\lambda}_{B}\big]^2\Big]+\mbox{error},
\end{align*}
where $\mbox{error} \to 0$ for almost every $\eta>0$ as $\scL \to \infty$ then $\lambda \to 0$. 
Using Lemma~\ref{lem:decoupling_scalar}, we have 
 \begin{align*}
 &\frac{1}{|B|^4}\sum_{x,y,z,w \in B} \E\Big[\E\big[\theta_x\theta_y|Y^{\delta,\eta,\lambda}_{B}]^2\E[\theta_z\theta_w|Y^{\delta,\eta,\lambda}_{B}\big]^2\Big]\\
  &= \frac{1}{|B|^4}\sum_{x,y,z,w \in B} \E\Big[\E\big[\theta_x|Y^{\delta,\eta,\lambda}_{B}\big]^2\E \big[\theta_y|Y^{\delta,\eta,\lambda}_{B}\big]^2\E \big[\theta_z|Y^{\delta,\eta,\lambda}_{B}\big]^2\E\big[\theta_w|Y^{\delta,\eta,\lambda}_{B}\big]^2\Big]+\mbox{error}\\
  &=\E\left[\Big(\frac{1}{|B|}\sum_{x \in B}\E\big[\theta_x|Y^{\delta,\eta,\lambda}_{B}\big]^2\Big)^4\right]+\mbox{error},
  \end{align*}
where $\mbox{error} \to 0$ for almost every $\lambda>0$ as $\scL \to \infty$. Let $X := \frac{1}{|B|}\sum_{x \in B}\E [\theta_x|Y^{\delta,\eta,\lambda}_{B}]^2$. The main term in the above display is $\E[X^4]$. We will compare it to $\E[X^2]^2$: since $X \in [0,1]$ almost surely, we have
\[\var(X^2) = \frac{1}{2} \E[(X^2-X'^2)^2] \le 2\E[(X-X')^2] = 4\var(X).\]
(Here, $X'$ is an independent copy of $X$.) Moreover,
\begin{align*}
\var(X) &= \frac{1}{|B|^2}\sum_{x,y \in B} \E\Big[\E\big [\theta_x|Y^{\delta,\eta,\lambda}_{B}\big]^2 \E\big[\theta_y|Y^{\delta,\eta,\lambda}_{B}\big]^2\Big] - \E\Big[\frac{1}{|B|}\sum_{x \in B}\E \Big[\theta_x|Y^{\delta,\eta,\lambda}_{B}\big]^2\Big]^2\\
&= \frac{1}{|B|^2}\sum_{x,y \in B} \E\Big[\E \big[\theta_x\theta_y|Y^{\delta,\eta,\lambda}_{B}\big]^2\Big] - \E\Big[\frac{1}{|B|}\sum_{x \in B}\E \big[\theta_x|Y^{\delta,\eta,\lambda}_{B}\big]^2\Big]^2 + \mbox{error}\\
&= \varphi_{B}^{\se}(\delta,\eta,\lambda) -  \varphi_{B}^{\sv}(\delta,\eta,\lambda)^2 + \mbox{error}.
\end{align*}
 Lemma~\ref{lem:removing_scalar} implies $\mbox{error} \to 0$ for almost every $\eta>0$ as $\scL \to \infty$ then $\lambda \to 0$. Moreover, from Lemma~\ref{lem:edge_vertex_overlap}, we have $\varphi_{B}^{\se}(\delta,\eta,\lambda) -  \varphi_{B}^{\sv}(\delta,\eta,\lambda)^2 \to 0$ as $\scL \to \infty$. Thus, we have shown that for almost every $\eta>0$,
 \[\lim_{\lambda\to 0} \lim_{\scL \to \infty}\var(X^2) = 0.\] 
Now, since Lemma~\ref{lem:decoupling_scalar} implies 
\[\E[X^2] = \E\Big[\Big(\frac{1}{|B|}\sum_{x \in B}\E \big[\theta_x|Y^{\delta,\eta,\lambda}_{B}\big]^2\Big)^2\Big] =  \frac{1}{|B|^2}\sum_{x,y \in B} \E\Big[\E\big[\theta_x\theta_y|Y^{\delta,\eta,\lambda}_{B}\big]^2\Big] +\mbox{error},\]
the above tends to $q_{\star\star}^2(\delta)$ as $\scL\to \infty$ then $\lambda \to 0$, and then $\eta \to 0$. We conclude that $\lim_{\eta \to 0} \lim_{\scL \to \infty}\var(M_B^2) = 0$. 
 
\subsection{Proof of item 2: analysis of $W_{B,B'}$}
We proceed analogously.  Recall that 
\[W_{B,B'} = \frac{1}{|B \cap B'|}  \sum_{x\in B \cap B'} \theta^{B}_x\theta^{B'}_x,\]
where $\big(\theta^{B}_x\big)_{x\in B} \sim \P(\cdot | Y_B )$ and $\big(\theta^{B'}_x\big)_{x\in B'} \sim \P(\cdot | Y_{B'} )$ independently. Therefore
\begin{align*}
\E[W_{B,B'}^2] &= \frac{1}{|B\cap B'|^2}\E\Big[\E\Big[\Big(\sum_{x\in B \cap B'} \theta^{B}_x\theta^{B'}_x\Big)^2 \Big| Y_{B}, Y_{B'}\Big]\Big]\\ 
&= \frac{1}{|B\cap B'|^2}\sum_{x,y \in B \cap B'}  \E\Big[\E\big[\theta_x\theta_y |Y_{B}\big]\E\big[\theta_x\theta_y |Y_{B'}\big]\Big].
\end{align*}
We will use overlap locking (Proposition~\ref{prop:locking}) to replace the conditional expectations $\E[\theta_x\theta_y |Y_{B}]$ and $\E[\theta_x\theta_y |Y_{B'}]$  in the above expression by $\E[\theta_x\theta_y |Y_{B \cup B'}]$ where $Y_{B\cup B'} = \{Y_{B} , Y_{B'}\}$ is the union of $Y_{B}$ and $Y_{B'}$. 
Now, let $(\theta^{B\cup B'}_x)_{x \in B \cup B'} \sim \P(\cdot | Y_{B \cup B'})$ be drawn from the two-block posterior conditionally independently from everything else, and let
 \[\tilde{M}_{B \cap B'} := \frac{1}{|B\cap B'|}\sum_{x \in B \cap B'} \theta_x^{B\cup B'}\theta_x.\] 
First we have 
\begin{equation}\label{eq:overlap_intra_inter}
\E[\tilde{M}_{B \cap B'}^2]= \frac{1}{|B\cap B'|^2}\sum_{x,y \in B \cap B'} \E\Big[\E\big[\theta_x\theta_y |Y_{B \cup B'}\big]^2\Big] = \E \left\langle R_{1,2}(B\cap B')^2\right\rangle_{B \cup B'}.
\end{equation}

Second, we since all variables are bounded, a simple triangle inequality implies
\begin{align*}
\Big|\E[W_{B,B'}^2]  - \E[\tilde{M}_{B \cap B'}^2]\Big| &\le \frac{1}{|B\cap B'|^2}\sum_{x,y \in B \cap B'} \Big\{\E\Big[\Big|\E\big[\theta_x\theta_y |Y_{B}\big] - \E\big[\theta_x\theta_y |Y_{B \cup B'}\big]\Big|\Big]\\
&\hspace{3.5cm}+\E\Big[\Big|\E\big[\theta_x\theta_y |Y_{B'}\big] - \E\big[\theta_x\theta_y |Y_{B \cup B'}\big]\Big|\Big]\Big\}.
\end{align*}
We use Jensen's inequality and the fact that $B$ and $B'$ play symmetric roles to write 
\begin{align*}
&\le \frac{2}{|B\cap B'|^2}\sum_{x,y \in B \cap B'} \E\Big[\Big(\E\big[\theta_x\theta_y |Y_{B}\big] - \E[\theta_x\theta_y |Y_{B \cup B'}\big]\Big)^2\Big]^{1/2}\\
&\le \Big(\frac{4}{|B\cap B'|^2}\sum_{x,y \in B \cap B'} \E\Big[\Big(\E\big[\theta_x\theta_y |Y_{B}\big] - \E\big[\theta_x\theta_y |Y_{B \cup B'}\big]\Big)^2\Big]\Big)^{1/2}.
\end{align*}
Since $ Y_{B}  \subset Y_{B \cup B'}$, using iterated expectations, we see that 
\[\E\big[\big(\E[\theta_x\theta_y |Y_{B}] - \E[\theta_x\theta_y |Y_{B \cup B'}]\big)^2\big] = \E\big[\E[\theta_x\theta_y |Y_{B\cup B'}]^2\big] - \E\big[\E[\theta_x\theta_y |Y_{B}]^2\big].\] 
It follows that 
\begin{equation}\label{eq:overlap_difference}
\Big|\E[W_{B,B'}^2]  - \E[\tilde{M}_{B \cap B'}^2]\Big|  \le 2 \Big( \E\big\langle R_{1,2}(B \cap B')^2\big\rangle_{B\cup B'}-\E\big\langle R_{1,2}(B \cap B')^2\big\rangle_{B}\Big)^{1/2}.
\end{equation}
By virtue of Proposition~\ref{prop:locking}, we have on the one hand, $R_{1,2}(B \cap B')$ and $R_{1,2}(B)$ lock together under the one-block posterior on $B$,
 and on the other hand, $R_{1,2}(B \cap B')$ and $R_{1,2}(B \cup B')$ lock together under the two-block posterior. More precisely, for almost all $t\in [0,1]$ and all $\eta \in [0,\eta_0]$, it holds 
\begin{align*}
&\Big|\E\big\langle R_{1,2}(B \cap B')^2\big\rangle_{B} - \E\big\langle R_{1,2}(B)^2\big\rangle_{B}\Big| \xrightarrow[\scL \to \infty]{} 0,~~~~ \mbox{and}~~\\
&\Big|\E\big\langle R_{1,2}(B \cap B')^2\big\rangle_{B\cup B'} - \E\big\langle R_{1,2}(B \cup B')^2\big\rangle_{B\cup B'}\Big| \xrightarrow[\scL \to \infty]{} 0.
\end{align*}
Furthermore, we know that $\E\big\langle R_{1,2}(B)^2\big\rangle_{B} \to {q^{\bullet}_1}^2(\delta,\eta)$ and $\E\big\langle R_{1,2}(B \cup B')^2\big\rangle_{B\cup B'} \to {q^{\bullet}_2}^2(\delta,\eta)$ from Corollary~\ref{cor:overlap_conv} for all except countably many $\eta \in [0,\eta_0]$. Now from Lemma~\ref{lem:maximizers_are_close}, $q^{\bullet}_1-q^{\bullet}_2 \to 0$ as $\eta \to 0$. We deduce from this and the bound~\eqref{eq:overlap_difference} that for almost every $t\in[0,1]$,
\[\lim_{\eta \to 0}\lim_{\scL\to \infty} \Big|\E[W_{B,B'}^2]  - \E[\tilde{M}_{B \cap B'}^2]\Big| = 0,\]
and Eq.\eqref{eq:overlap_intra_inter} implies 
\[\lim_{\eta \to 0}\lim_{\scL\to \infty} \E[\tilde{M}_{B \cap B'}^2] = \lim_{\eta \to 0} {q^{\bullet}_2}^2(\delta,\eta) = q_{\star\star}^2(\delta).\]
This establishes the first claim of item 2.

As for the variance, it suffices to consider the fourth moment of $W_{B,B'}$. We proceed similarly to the case of $M_B$: 
 \[\E[W_{B,B'}^4] = \frac{1}{|B\cap B'|^4}\sum_{x,y,z,w \in B\cap B'} \E\Big[\big[\theta_x\theta_y \theta_z\theta_w|Y^{\delta,\eta}_{B}\big] \cdot \E \big[\theta_x\theta_y \theta_z\theta_w|Y^{\delta,\eta}_{B'}\big]\Big].\]
We use Lemma~\ref{lem:decoupling_variance} (applied with $A=B$ then $A=B'$) to decouple the pairs $x,y$ and $w,z$: 
 \begin{align*}
& \E[W_{B,B'}^4] = \\ &\frac{1}{|B\cap B'|^4}\sum_{x,y,z,w \in B\cap B'} \E\Big[\E \big[\theta_x\theta_y |Y^{\delta,\eta}_{B}\big] \E \big[\theta_x\theta_y |Y^{\delta,\eta}_{B'}\big] \E\big[ \theta_z\theta_w|Y^{\delta,\eta}_{B}\big] \E\big[ \theta_z\theta_w|Y^{\delta,\eta}_{B'}\big]\Big] +o_{\scL}(1),
 \end{align*}
 where $o_{\scL}(1) \to 0$ as $\scL\to \infty$ for almost every $\eta$.
Now we repeat the argument used in the analysis of $\E[W_{B,B'}^2]$ to argue that the above is
\[\frac{1}{|B\cup B'|^4}\sum_{x,y,z,w \in B\cup B'} \E\Big[\E \big[\theta_x\theta_y |Y_{B\cup B'}\big]^2 \E \big[ \theta_z\theta_w|Y_{B\cup B'}\big]^2\Big]+\mbox{error},\]
where $\mbox{error} \to 0$ as $\scL \to \infty$ then $\eta\to 0$, and $Y_{B\cup B'} = \{Y_{B}, Y_{B'}\}$ is as defined above. 
We use Proposition~\ref{lem:decoupling_variance} to merge the pairs $x,y$ and $w,z$. The above becomes
\[\frac{1}{|B\cup B'|^4}\sum_{x,y,z,w \in B\cup B'} \E\Big[\E \big[\theta_x\theta_y\theta_z\theta_w|Y_{B\cup B'}\big]^2\Big]+\mbox{error}.\] 
The main term in the above display is equal to $\E[M_{B\cup B'}^4]$. Here the same analysis used for $\E[M_{B}^4]$ in the previous subsection applies as well. Therefore $\E[M_{B\cup B'}^4]$ converges to $q_{\star\star}^4(\delta)$. We conclude that $\var(W_{B,B'}^2)\to 0$.

\subsection{Proof of item 3: analysis of $W_{B,B'}M_BM_{B'}$}
We first perform a preliminary computation and then sketch the argument, which we then execute in three steps. 
Let $Y_{B \cup B'} = \{Y_{B},Y_{B'}\}$ where as in the analysis of $\E[M_{B}^2]$, we use the short-hand $Y_B = \{Y^{\delta}_B , Y^{\eta}_B\}$. Then
\begin{align}\label{eq:triangle_overlap}
{|B|\cdot |B'| \cdot |B \cap B'|} & \cdot \E[W_{B,B'} M_B M_{B'}] \nonumber\\
&=\E\Big[\E\Big[\Big(\sum_{x\in B \cap B'} \theta^{B}_x\theta^{B'}_x\Big)\cdot \Big(\sum_{y\in B} \theta^{B}_y \theta_y\Big)\cdot \Big(\sum_{z\in B'}\theta^{B'}_z \theta_z\Big) ~\Big|~ Y_{B \cup B'}\Big]\Big] \nonumber\\
&=\sum_{x \in B \cap B'} \sum_{y\in B} \sum_{z \in B'} \E\Big[\E\big[\theta^{B}_x\theta^{B'}_x\theta^{B}_y \theta_y\theta^{B'}_z\theta_z |Y_{B \cup B'}\big]\Big]\nonumber\\
&=\sum_{x \in B \cap B'} \sum_{y\in B} \sum_{z \in B'} \E\Big[\E\big[\theta_x \theta_y | Y_B\big] \E\big[\theta_x\theta_z|Y_{B'}\big] \E\big[\theta_y \theta_z|Y_{B \cup B'}\big]\Big].
\end{align}
In this case the strategy is to replace the conditional expectations $\E\big[\theta_x \theta_y | Y_{B}\big]$ and $\E\big[\theta_x\theta_z|Y_{B'}\big]$ in the right-hand side with  $\E\big[\theta_x \theta_y | Y_{B \cup B'}\big]$ and $\E\big[\theta_x\theta_z|Y_{B \cup B'}\big]$ respectively.
Once we succeed at this operation, $\E[W_{B,B'} M_B M_{B'}]$ is approximated by the quantity 
\begin{align*}
\frac{1}{|B|\cdot |B'| \cdot |B \cap B'|}\sum_{x\in B\cap B',y\in B,z\in B' } \E\Big[\E\big[\theta_x \theta_y | Y_{B \cup B'}\big] \E\big[\theta_x\theta_z|Y_{B \cup B'}\big] \E[\theta_y \theta_z | Y_{B \cup B'}\big] \Big].
\end{align*}
Notice that the above sum is over different subsets of $B \cup B'$. We then exploit the fact that local overlaps lock under the two-block posterior  to replace the above with a homogenous sum over all vertices of $B \cup B'$: 
\[\frac{1}{|B \cup B'|^3}\sum_{x,y,z \in B\cup B'} \E\Big[\E\big[\theta_x \theta_z | Y_{B \cup B'}\big] \E\big[\theta_y\theta_z|Y_{B \cup B'}\big] \E[\theta_x \theta_y | Y_{B \cup B'}\big] \Big].\]
We remark that this is the trace of the third power of a certain matrix $\chi$, which we call  \emph{susceptibility matrix} of the block $B \cup B'$, defined by 
\begin{equation}\label{eq:susceptibility}
\chi_{x,y} := \frac{1}{|B \cup B'|}\E\big[\theta_x \theta_y | Y_{B \cup B'}\big],~~~ x,y \in B \cup B'.
\end{equation}
We will then argue that $\chi$ is approximately a \emph{rank-one} matrix, in the sense that its Frobenius and operator norms are equal in the limit $\scL \to \infty$ then $\eta \to 0$, with common asymptotic value $q_{\star\star}(\delta)$. Since $\trace (\chi^3)^{1/3}$ is sandwiched between the operator and Frobenius norms, a concentration argument finally yields that $\E[W_{B,B'} M_B M_{B'}]$ converges to $q_{\star\star}^{3}(\delta)$.   

\paragraph{Step 1: replacing the conditional expectations.}
The first step is to replace $Y_{B}$ and $Y_{B'}$ by $Y_{B \cup B'}$ in the conditional expectations in~\eqref{eq:triangle_overlap}. Since $Y_B \subset Y_{B \cup B'}$, we have
\begin{align*}
\frac{1}{|B| \cdot |B \cap B'|}&\sum_{x\in B \cap B',y\in B } \E\Big[\big(\E\big[\theta_x \theta_y | Y_{B}\big] - \E\big[\theta_x \theta_y | Y_{B\cup B'}\big]\big)^2\Big]\\
&= \frac{1}{|B| \cdot |B \cap B'|}\sum_{x\in B \cap B',y\in B} \E\Big[\E\big[\theta_x \theta_y | Y_{B\cup B'}\big]^2\Big] - \E\Big[\E\big[\theta_x \theta_y | Y_{B}\big]^2\Big]\\
&=\E\big\langle R_{1,2}(B)  R_{1,2}(B \cap B')\big\rangle_{B \cup B'} - \E\big\langle R_{1,2}(B)  R_{1,2}(B \cap B')\big\rangle_{B}.
\end{align*}
Overlap locking under both the one-block and two-block posteriors, as given in Proposition~\ref{prop:locking}, implies that the above is equal to
\[\E\big\langle R_{1,2}(B \cup B')^2\big\rangle_{B \cup B'} - \E\big\langle R_{1,2}(B)^2\big\rangle_{B} + o_{\scL}(1),\]
where $o_{\scL}(1) \to 0$ as $\scL \to \infty$ for almost all $t\in [0,1]$ and all $\eta \in [0,\eta_0]$.  
Now we use Corollary~\ref{cor:overlap_conv} and Lemma~\ref{lem:maximizers_are_close} to deduce that for almost all $t$,
\[\lim_{\eta\to0} \lim_{\scL\to 0}\E\big\langle R_{1,2}(B \cup B')^2\big\rangle_{B \cup B'} - \E\big\langle R_{1,2}(B)^2\big\rangle_{B} =0.\]
The same result obviously holds if $B$ is replaced by $B'$ in the above argument. Therefore we can substitute $B$ and $B'$ with $B \cup B'$ in the conditional expectations in the expression of $\E[W_{B,B'} M_B M_{B'}]$, Eq.\eqref{eq:triangle_overlap}: as $\scL \to 0$ then $\eta \to 0$ we have
\begin{align*}
\Big|&\E[W_{B,B'} M_B M_{B'}] \\ &- \frac{1}{|B|\cdot |B'| \cdot |B \cap B'|} \sum_{x\in  B \cap B', y\in B, z\in B'}  \hspace{-.5cm}\E\Big[\E\big[\theta_x \theta_y | Y_{B\cup B'}\big] \E\big[\theta_x\theta_z|Y_{B\cup B'}\big] \E[\theta_y \theta_z |Y_{B\cup B'}\big]  \Big] \Big| \longrightarrow 0.
\end{align*}

\paragraph{Step 2: replacing the domain of summation.}
The next step is to replace the domains of summation in the above sum to $B \cup B'$ for the three indices. The above sum can be written as
\[\E\big\langle R_{1,2}(B) R_{2,3}(B \cap B') R_{3,1}(B')\big\rangle_{B \cup B'}.\] 
Since each of the three overlaps above is locked to $R_{a,b}(B \cup B')$ for the appropriate replica pair $(a,b)$ (this is again a straightforward application of Proposition~\ref{prop:locking}, the above expression is equal to
\[\E\big\langle R_{1,2}(B \cup B') R_{2,3}(B \cup B') R_{3,1}(B \cup B')\big\rangle_{B \cup B'} + o_{\scL}(1),\]
where again $o_{\scL}(1) \to 0$ as $\scL \to \infty$ for almost all $t \in [0,1]$ and all $\eta \in [0,\eta_0]$.  
As explained in the sketch of the proof, the above triangular product can be expressed as follows:
 \begin{align*}
 \E\big\langle R_{1,2}(B \cup B') & R_{2,3}(B \cup B') R_{3,1}(B \cup B')\big\rangle_{B \cup B'} \\
 &= \frac{1}{|B \cup B'|^3}\sum_{x,y,z \in B\cup B'} \E\Big[\E\big[\theta_x \theta_z | Y_{B \cup B'}\big] \E\big[\theta_y\theta_z|Y_{B \cup B'}\big] \E[\theta_x \theta_y | Y_{B \cup B'}\big] \Big]\\
 & = \E \trace(\chi^3),
 \end{align*}
 where $\chi$ is the susceptibility matrix~\eqref{eq:susceptibility}.
 
\paragraph{Step 3: structure of the susceptibility matrix.} Next, we prove that $\E \trace(\chi^3) \to q_{\star\star}^3(\delta)$ for almost all $t \in [0,1]$.  
We proceed by showing upper and lower bounds separately. For the upper bound, we observe that as $\scL \to \infty$ then $\eta\to 0$, we have
\[\E\trace(\chi^2) = \E \left\langle R_{1,2}(B \cup B')^2\right\rangle_{B \cup B'} \longrightarrow q_{\star\star}^2(\delta).\] 
This is obtained through previously employed arguments: Corollary~\ref{cor:overlap_conv} and Lemma~\ref{lem:maximizers_are_close}.
Moreover, by monotonicity of the $\ell_p$ norm and the fact that $\chi$ is a symmetric and positive semi-definite (PSD) matrix,
\[\trace(\chi^3)^{1/3} \le \trace(\chi^2)^{1/2}.\]
Thus, $\E[\trace(\chi^3)] \le \E[\trace(\chi^2)^{3/2}]$.
Now, since
\[\E\big[\trace(\chi^2)^{2}\big] - \E\big[\trace(\chi^2)\big]^{2} =  \E\big[M_B^4\big] - \E\big[M_B^2\big]^2 = \var(M_B^2),\]
which, according to Theorem~\ref{thm:everything}, item 1, converges to zero as $\scL \to 0$ then $\eta \to 0$, we have 
\[\limsup_{\eta \to 0} \limsup_{\scL\to \infty} \Big(\E[\trace(\chi^2)^{3/2}] - \E[\trace(\chi^2)]^{3/2}\Big) =0.\] 
Further, since $\E[\trace(\chi^2)]^{3/2} \to q_{\star\star}^3(\delta)$ as $\scL\to \infty$ then $\eta\to 0$, this establishes the upper bound 
\begin{equation*}
\limsup_{\eta \to 0} \limsup_{\scL\to \infty} \E\big[\trace(\chi^3)\big] \le q_{\star\star}^3(\delta).
\end{equation*}
As for the lower bound, we have $\trace(\chi^3) \ge \|\chi\|_{\op}^3$ since $\chi$ is PSD. We will prove that 
\begin{equation}\label{eq:lower_bound_op}
\liminf_{\eta \to 0} \liminf_{\scL\to \infty} \E\|\chi\|_{\op} \ge q_{\star\star}(\delta).
\end{equation}
This implies a matching lower bound through Jensen's inequality.
We have the following characterization of the operator norm for a PSD matrix:
\[\|\chi\|_{\textup{op}} = \sup_{|u|=1} u^\intercal \chi u,\] 
where $|\cdot|$ is the $\ell_2$ norm of a vector.
Now it suffices to exhibit a candidate unit vector $u$ that (approximately) achieves the required bound. 
Our candidate is $u = \hat{u}/|\hat{u}|$, where
\begin{equation}\label{eq:candidate_rankone}
\hat{u}_x = \frac{1}{\sqrt{|B \cup B'|}}\E\big[\theta_x|Y_{B \cup B'}^{\delta,\eta,\lambda}\big],~~~ x \in B \cup B',
\end{equation}
where we have added additional side information from the scalar Gaussian noise channel~\eqref{eq:gaussian_channel} with SNR $\lambda>0$ on $B \cup B'$. We show next that as $\lambda \to 0$, the (sequence of) vector(s) $u$ achieves the supremum. 
We use the decoupling Lemma~\ref{lem:decoupling_scalar} (with $A = B \cup B'$) to obtain 
\begin{align*}
 \E\big[\hat{u}^\intercal \chi \hat{u}\big] &= \frac{1}{|B \cup B'|^2}\sum_{x,y\in B\cup B'} \E\Big[\E\big[\theta_x \theta_y | Y^{\delta,\eta}_{B \cup B'}\big]\E\big[\theta_x|Y_{B \cup B'}^{\delta,\eta,\lambda}\big]\E\big[\theta_y|Y_{B \cup B'}^{\delta,\eta,\lambda}\big]\Big]\\
 &= \frac{1}{|B \cup B'|^2}\sum_{x,y\in B\cup B'} \E\Big[\E\big[\theta_x \theta_y | Y^{\delta,\eta}_{B \cup B'}\big]\E\big[\theta_x\theta_y|Y_{B \cup B'}^{\delta,\eta,\lambda}\big]\Big] + o_{\scL}(1),
\end{align*}
for almost every $\lambda>0$.
Moreover, according to Lemma~\ref{lem:removing_scalar} we can remove the scalar side information and incur a small error: 
\begin{align*}
\E\big[\hat{u}^\intercal \chi \hat{u}\big] &= \frac{1}{|B \cup B'|^2}\sum_{x,y\in B\cup B'} \E\Big[\E\big[\theta_x \theta_y | Y^{\delta,\eta}_{B \cup B'}\big]^2\Big]
+ \mbox{error}\\
&= \E\trace(\chi^2)+ \mbox{error},
\end{align*}  
where $\mbox{error} \to 0$ as $\scL\to \infty$ and then $\lambda \to 0$ for almost every $\eta>0$. We have already established that $\E\trace(\chi^2) \to q_{\star\star}^2(\delta)$ as $\scL\to \infty$ and $\eta\to 0$. Now it remains to analyze $|\hat{u}|$. 
We have
\[\E\big[|\hat{u}|^2\big] = \frac{1}{|B \cup B'|}\sum_{x\in B\cup B'} \E\Big[\E\big[\theta_x | Y^{\delta,\eta,\lambda}_{B \cup B'}\big]^2\Big] = \varphi^{\sv}_{B \cup B'}(\delta,\eta,\lambda).\]
Now we examine the fourth moment of $|\hat{u}|$: 
\begin{align*}
\E\big[|\hat{u}|^4\big] &= \frac{1}{|B \cup B'|^2}\sum_{x,y\in B\cup B'} \E\Big[\E\big[\theta_x | Y^{\delta,\eta,\lambda}_{B \cup B'}\big]^2\E\big[\theta_y | Y^{\delta,\eta,\lambda}_{B \cup B'}\big]^2\Big]\\
&= \frac{1}{|B \cup B'|^2}\sum_{x,y\in B\cup B'} \E\Big[\E\big[\theta_x \theta_y| Y^{\delta,\eta,\lambda}_{B \cup B'}\big]^2\Big]+\mbox{error}_1\\
&= \varphi^{\se}_{B \cup B'}(\delta,\eta,\lambda)+\mbox{error}_1, 
\end{align*}
where, similarly to previous arguments, $\mbox{error}_1 \to 0$ for almost every $\lambda>0$ and all $\eta>0$. 
Now observe that by virtue of Lemma~\ref{lem:edge_vertex_overlap}, 
\[\varphi^{\se}_{B \cup B'}(\delta,\eta,\lambda) - \varphi^{\sv}_{B \cup B'}(\delta,\eta,\lambda)^2 \to 0,~~ \mbox{as}~~ \scL \to \infty.\] 
Thus $\var(|\hat{u}|^2) \to 0$.
As a consequence, we have 
\begin{align*}
 \E\big[u^\intercal \chi u\big]  &= \E\Big[\frac{\hat{u}^\intercal \chi \hat{u}}{|\hat{u}|^2}\Big]
 = \E\Big[\frac{\hat{u}^\intercal \chi \hat{u}}{\E[|\hat{u}|^2]}\Big] + \mbox{error},
\end{align*}
where 
\begin{align*}
|\mbox{error}| &\le \E\left[\frac{|\hat{u}^\intercal \chi \hat{u}| \cdot |\Delta|}{\E[|\hat{u}|^2] \cdot |\hat{u}|^2}\right],~~~\Delta= |\hat{u}|^2 - \E[|\hat{u}|^2]\\
&\le \frac{\E[|\Delta| \indi\{|\hat{u}|<\epsilon\}]}{\E[|\hat{u}|^2]} +  \E\left[\frac{|\hat{u}^\intercal \chi \hat{u}| \cdot |\Delta|}{\E[|\hat{u}|^2] \cdot |\hat{u}|^2}\indi\{|\hat{u}|\ge\epsilon\}\right], ~~~ (\mbox{for } \epsilon>0)\\
&\le 2 \P(|\hat{u}|<\epsilon) +\frac{\E[|\Delta|]}{ \epsilon^2 \E[|\hat{u}|^2] }.
\end{align*}
By Chebychev's inequality, $\P(|\hat{u}|<\epsilon) \le \var(|\hat{u}|^2)/(\E[|\hat{u}|^2]-\epsilon^2)$ for $\epsilon < \E[|\hat{u}|^2]^{1/2}$. On the other hand $\E[|\Delta|] \le \var(|\hat{u}|^2)^{1/2}$. 
When $\delta>\delta_c$,  $q_{\star\star}(\delta)>0$ and  $\E[|\hat{u}|^2]>0$ for $\scL$ large enough, and the above tends to zero.
Therefore, we have $\E[u^\intercal \chi u]  \to q_{\star\star}(\delta)$ when $\delta>\delta_c$ as $\scL\to \infty$ then $\eta \to 0$. This implies the desired lower bound~\eqref{eq:lower_bound_op}.

Putting the upper and lower bounds  together we deduce that 
\[\lim_{\eta\to 0} \lim_{\scL\to \infty} \E\trace (\chi^3) = q_{\star\star}^3(\delta).\]

\subsection{Proof of Theorem~\ref{thm:lower_bound}}
\label{sec:proof_13}
For any $T$ with output values in $\{\pm 1\}$ and $L_n \ge 1$,  Cauchy-Schwarz inequality implies 
\begin{align*}
|\risk_{\Lambda_n}(T)| &= \frac{1}{|\Lambda_n|^2} \left|\sum_{u,v\in \Lambda_n} \E\Big[T_{uv}(Y_{\Lambda_n}^{\delta,\eta,L_n})\E\big[\theta_u\theta_v|Y_{\Lambda_n}^{\delta}, Y_{\Lambda_n}^{\eta,L_n}\big]\Big]\right|\\
&\le \Big(\frac{1}{|\Lambda_n|^2} \sum_{u,v\in \Lambda_n} \E\Big[\E\big[\theta_u\theta_v|Y_{\Lambda_n}^{\delta}, Y_{\Lambda_n}^{\eta,L_n}\big]^2\Big]\Big)^{1/2}.
\end{align*}   
We add auxiliary GOE side information with SNR parameter $\eta_2$ and interaction range $L = 2n+1$, (i.e., there is a measurement $Y_{uv}^{\eta_2}$ available for any pair $(u,v)$ in $\Lambda_n$) and let 
\[\tilde{\varphi}_{\Lambda_n}^{\se} := \frac{1}{|\Lambda_n|^2} \sum_{u,v\in \Lambda_n} \E\Big[\E\big[\theta_u\theta_v|Y_{\Lambda_n}^{\delta}, Y_{\Lambda_n}^{\eta,L_n},Y_{\Lambda_n}^{\eta_2}\big]^2\Big].\]
 We clearly have $|\risk_{\Lambda_n}(T)| \le \big(\tilde{\varphi}_{\Lambda_n}^{\se}\big)^{1/2}$. Now recall the definition of $\varphi_{\Lambda_n}^{\se}$:
 \[\varphi_{\Lambda_n}^{\se}(\delta,\eta_2) := \frac{1}{|\Lambda_n|^2} \sum_{u,v\in \Lambda_n} \E\Big[\E\big[\theta_u\theta_v|Y_{\Lambda_n}^{\delta},Y_{\Lambda_n}^{\eta_2}\big]^2\Big].\]
 The crucial point is that $\varphi_{\Lambda_n}^{\se}(\delta,\eta_2)$ and  $\tilde{\varphi}_{\Lambda_n}^{\se}(\delta,\eta,\eta_2)$ are close when $\eta$ is small:  
 \begin{lemma}
For almost all $\eta_2>0$,
\[\lim_{\eta \to 0}\liminf_{n\to \infty} (\tilde{\varphi}_{\Lambda_n}^{\se}- \varphi_{\Lambda_n}^{\se}) = 0.\]
\end{lemma}
\begin{proof}
The proof is similar to that of Lemma~\ref{lem:removing_scalar}. Let $\tilde{f}_n(\delta,\eta,L_n,\eta_2) =\frac{1}{|\Lambda_n|}\E\log Z_{\Lambda_n}$ where $Z_{\Lambda_n}$ is the partition function of $\P(\cdot\, | Y_{\Lambda_n}^{\delta}, Y_{\Lambda_n}^{\eta,L_n},Y_{\Lambda_n}^{\eta_2})$. Then $\eta \mapsto \tilde{f}_n(\delta,\eta,L_n,\eta_2)$ is $\frac{1}{2}$-Lipschitz as seen by taking the $\eta$-derivative of $\tilde{f}_n$:
\[\frac{\rmd}{\rmd \eta} \tilde{f}_n = \frac{1}{2|\Lambda_n| L_n^d}\sum_{u \in \Lambda_n} \sum_{\underset{|u-v|\le L_n}{v\in \Lambda_n}} \E\Big[\E\big[\theta_u\theta_v|Y_{\Lambda_n}^{\delta}, Y_{\Lambda_n}^{\eta,L_n},Y_{\Lambda_n}^{\eta_2}\big]^2\Big] \le \frac{1}{2}. \]  
On the other hand, the $\eta_2$-derivative of $\tilde{f}_n$ is $\frac{1}{2}\tilde{\varphi}_{\Lambda_n}^{\se}(\delta,\eta,\eta_2)$. Therefore, for all $\eta_2$,
\begin{align*}
\int_0^{\eta_2} (\tilde{\varphi}_{\Lambda_n}^{\se}(\delta,\eta,\eta_2')- \varphi_{\Lambda_n}^{\se}(\delta,\eta_2')) \rmd \eta_2' &= 2(\tilde{f}_n(\delta,\eta,L_n,\eta_2)-\tilde{f}_n(\delta,\eta,L_n,0)) \\
&~~~- 2(\tilde{f}_n(\delta,0,L_n,\eta_2)-\tilde{f}_n(\delta,0,L_n,0))\\
&=2(\tilde{f}_n(\delta,\eta,L_n,\eta_2)-\tilde{f}_n(\delta,0,L_n,\eta_2))\\
&~~~- 2(\tilde{f}_n(\delta,\eta,L_n,0)-\tilde{f}_n(\delta,0,L_n,0))\\
&\le 2\eta.
\end{align*}
We used the Lipschitz property of $\tilde{f}_n$ in $\eta$ to obtain the last line. Since the integrand is nonnegative, Fatou's lemma implies
\[\int_0^{\eta_2} \liminf_{n\to \infty}(\tilde{\varphi}_{\Lambda_n}^{\se}(\delta,\eta,\eta_2')- \varphi_{\Lambda_n}^{\se}(\delta,\eta_2')) \rmd \eta_2' \le 2\eta.\] 
Now, letting $\eta \to 0$ implies the result.
\end{proof}
Given the above result, we have for almost all $\eta_2>0$,
\[\limsup_{\eta \to 0} \liminf_{n \to \infty} |\risk_{\Lambda_n}(T)| \le \lim_{n \to \infty} \big(\varphi_{\Lambda_n}^{\se}(\delta,\eta_2)\big)^{1/2}.\]
Since $\underset{\eta_2\to 0}{\lim}\,\underset{n\to \infty}{\lim} \varphi_{\Lambda_n}^{\se}(\delta,\eta_2) = 0$ when $\delta<\delta_c$ we have
\[\lim_{\eta \to 0} \liminf_{n \to \infty} |\risk_{\Lambda_n}(T)| = 0.\]

\section{Limits of the free energies}
\label{sec:free_energy}
In this section we prove convergence of the various free energies defined in Section~\ref{sec:posteriors}, i.e., we prove Propositions~\ref{prop:limit_Bn_scalar}, \ref{prop:limit_Bn_GOE}, \ref{prop:limits_phi1_phi2} and~\ref{prop:additional_scalar}.
\subsection{Proof of Proposition~\ref{prop:limit_Bn_scalar}}
We start with the free energy $\fsc_{B_n}(\delta,\lambda)$ in Eq.~\eqref{eq:free_energy_scalar}. The argument is standard and we follow Friedli and Velenik's exposition~\cite[Chapter 3]{friedli2017statistical}. Minor changes need to be made in order to accommodate our setting with disorder and the additional Hamiltonian $H^{\lambda}_{B_n}$ encoding the scalar side information.
We start with the special case where $B_n = D_n = \{0,1,\cdots,2^n\}^d$, we then extend the convergence to a general van Hove sequence. 

Observe that $D_{n+1}$ can be partitioned disjointly into $2^d$ translates of $D_n$; call them $D_{n}^{(1)},\cdots,D_{n}^{(2^d)}$. Moreover, the Hamiltonian~\eqref{eq:hamiltonian_lattice} corresponding to the information $Y^{\delta}_{D_n}$ observed in the edges of the lattice can be decomposed as follows:          
 \begin{align*}
 H_{D_{n+1}}^{\delta}(\theta) &= \beta\sum_{\underset{|u-v|= 1}{u,v\in D_{n+1}}} Y^{\delta}_{uv}\theta_u\theta_v,\\
 &= \sum_{i=1}^{2^d} H_{D_{n}^{(i)}}^{\delta}(\theta) + R,
 \end{align*}
 where $H_{D_{n}^{(i)}}^{\delta}(\theta)$ are defined exactly as in~\eqref{eq:hamiltonian_lattice} (observe that they are identical in law), and $R$ collects the terms corresponding to the edges whose endpoints belong to two different sub-blocks $D_{n}^{(i)}$. Since there are at most $|\partial D_{n+1}| = \beta d 2^{(n+1) (d-1)}$ such edges, and all individual terms in the above Hamiltonian are bounded by one in absolute value, we have $|R| \le \beta d 2^{(n+1) (d-1)}$.  Therefore, denoting by $Z_{D_{n+1}}$ and $Z_{D^{(i)}_{n}}$ the partition functions we have
 \[\prod_{i=1}^{2^d}Z_{D_{n}^{(i)}} \cdot e^{-\beta d 2^{(n+1) (d-1)}} \le Z_{D_{n+1}} \le \prod_{i=1}^{2^d}Z_{D_{n}^{(i)}} \cdot e^{\beta d 2^{(n+1) (d-1)}}.\]   
Since  $\{Z_{D_{n}^{(i)}}: i=1,\cdots,2^d\}$ are all equal in distribution, taking logarithms and then expectations, we get
\[\frac{2^d |D_{n}|}{|D_{n+1}|} \fsc_{D_n} - \beta d \frac{2^{(n+1) (d-1)}}{|D_{n+1}|} \le \fsc_{D_{n+1}} \le \frac{2^d |D_{n}|}{|D_{n+1}|} \fsc_{D_n} + \beta d \frac{2^{(n+1) (d-1)}}{|D_{n+1}|}.\] 
 Since $|D_{n}| = 2^{nd}$ for all $n$, this yields
 \[ |\fsc_{D_{n+1}} -  \fsc_{D_n}| \le \frac{\beta d}{2^{n+1}}.\]
 We deduce from this that the sequence $(\fsc_{D_n})_{n\ge 0}$ is Cauchy: $|\fsc_{D_{m}} -  \fsc_{D_n}| \le \beta d 2^{-n}$ for all $m \ge n$. Hence it converges to a limit $\fsc$ as $n \to \infty$, and the convergence is uniform on any interval $[0,\beta_0]$ with $\beta_0<\infty$. 
 
 Consider now an arbitrary van Hove sequence $(B_{n})_{n \ge 0}$. Fix an integer $k$ and partition $\Z^d$ into disjoint translates of $D_k= \{0,1,\cdots,2^k\}^d$. For each $n$ let $[B]_n$ be the minimal cover of $B_n$ by translates of $D_k$. We have
 \[|\fsc_{B_n} - \fsc| \le |\fsc_{B_n} - \fsc_{[B]_n}|+ |\fsc_{[B]_n} - \fsc_{D_k}| + |\fsc_{D_k} - \fsc|.\]  
Now we bound each one of these three terms. For the first term, we let $\Delta_n = [B]_n \setminus B_n$. We then have
 \[Z_{B_n} \cdot \widetilde{Z}_{\Delta_n} \cdot e^{-2d\beta |\Delta_n|} \le Z_{[B]_n} \le Z_{B_n} \cdot \widetilde{Z}_{\Delta_n} \cdot e^{2d\beta |\Delta_n|},\]
 where 
 \[\widetilde{Z}_{\Delta_n} = 2^{-|\Delta_n|}\sum_{\theta\in \{\pm 1\}^{\Delta_n}} e^{H^{\lambda}_{\Delta_n}(\theta)}.\]
 Therefore, letting $\vartheta =\frac{1}{|\Delta_n|} \E\log \widetilde{Z}_{\Delta_n} = \E\log \cosh(\sqrt{\lambda}z+\lambda)$ with $z\sim N(0,1)$, we have
 \[\Big| \fsc_{[B]_n} - \frac{|B_n|}{|[B]_n|}\fsc_{B_n} - \frac{|\Delta_n|}{|[B]_n|}\vartheta\Big| \le \frac{2d\beta |\Delta_n|}{|[B]_n|}.\]
On the one hand, $[B]_n$ is a minimal cover of $B_n$ by translates of $D_k$, so we have $|\Delta_n| \le |D_k| \cdot |\partial B_n|$. On the other hand $|[B]_n| \ge |B_n|$, and $\fsc_{B_n}$ is uniformly bounded by $2\beta d +\vartheta$. Thus
\[|\fsc_{B_n} - \fsc_{[B]_n}| \le C \frac{|\partial B_n|}{|B_n|} \xrightarrow[n\to \infty]{} 0\]    
for some constant $C = C(\beta,d,k,\lambda)>0$. Convergence is again uniform on $[0,\beta_0] \times[0,\lambda_0]$ for any $\beta_0,\lambda_0>0$.
Next, we address the second term. The same estimates previously conducted show that 
 \begin{align*}
 |\fsc_{[B]_n} - \fsc_{D_k}| &\le \frac{\beta |\partial [B]_n|}{|[B]_n|}
 \le  \frac{\beta |\partial D_k|}{|D_k|} \le \beta d 2^{-k}.
 \end{align*}
 Lastly, we already know that $|\fsc_{D_k} -\fsc| \to 0$ as $k \to \infty$. Now we take $n \to\infty$ then $k \to \infty$ to conclude the proof.

 \subsection{Proof of Proposition~\ref{prop:limit_Bn_GOE}}
 \label{sec:proof_limit_Bn_GOE}
Here we prove convergence of the free energy $f_{B_n}(\delta,\eta)$ with GOE side information with SNR $\eta>0$. The argument relies on Guerra's interpolation method~\cite{guerra2002thermodynamic} and concentration of measure arguments. It is almost identical to the proof in the spiked GOE case with no lattice information ($\delta=0$). We follow the exposition in~\cite{alaoui2018estimation}.    
The lower and upper bounds are proved separately. 
We start with the lower bound. We use the interpolating Hamiltonian
\begin{align*}
H_{s}(\theta) &= H^{\delta}_{B_n}(\theta) + \sum_{u,v\in B_n} \Big(\sqrt{\frac{s\eta}{|B_n|}} Y_{uv}^{s\eta}\theta_{u}\theta_{v} - \frac{s\eta}{2|B_n|} \Big)\\
&~~~+ \sum_{u\in B_n} \Big( \sqrt{(1-s)\eta q} \, y^{\lambda}_{u}\theta_{u}- \frac{(1-s)\eta q}{2} \Big),
\end{align*}
where $q\ge 0$ is a fixed constant to be chosen later, $Y_{uv}^{s\eta} = \sqrt{\frac{s\eta}{|B_n|}} \theta_{0u}\theta_{0v} + Z_{uv}$  and $y^{\lambda} = \sqrt{\lambda} \theta_{0u} +z_u$ with $\lambda = (1-s)\eta q$ and $Z_{uv}$ and $z_{u}$ are independent $N(0,1)$ r.v.'s. From now on we denote the hidden assignment by $(\theta_{0u})_{u \in B_n}$ in order to distinguish it from a generic vector $(\theta_u)_{u \in B_n}$. The interpolating free energy is 
\[\varphi(s) = \frac{1}{|B_n|} \E \log\Big\{2^{-|B_n|}\sum_{\theta \in \{\pm 1\}^{B_n}} e^{H_s(\theta)}\Big\}. \] 
Note that $\varphi(1) = f_{B_n}(\delta,\eta)$. On the other end of the interpolation,
\begin{align*}
\varphi(0) &=  \frac{1}{|B_n|} \E \log\Big\{2^{-|B_n|}\sum_{\theta \in \{\pm 1\}^{B_n}} e^{ H^{\delta}_{B_n}(\theta) +  \underset{u\in B_n}{\sum} ( \sqrt{\eta q} y^{\lambda}_{u}\theta_{u}- \frac{\eta q}{2})}\Big\} \\
&= \fsc_{B_n}(\delta,\eta q).
\end{align*}

We compute the derivative of $\varphi$ using gaussian integration by parts to obtain 
\[\varphi'(s) = -\frac{\eta}{4} \E \big\langle(R_{1,2}(B_n) - q)^2 \big\rangle_s + \frac{\eta}{2} \E \big\langle(R_{1,0}(B_n) - q)^2 \big\rangle_s - \frac{\eta q^2}{4},\]
where $\langle \cdot \rangle_s$ is the Gibbs measure associated to the Hamiltonian $H_s$.
(Here and from now on, $R_{1,0}(B_n):= \frac{1}{|B_n|} \sum_{u \in B_n} \theta_{u}\theta_{0u}$.)
By Bayes' rule $R_{1,2}(B_n)$ and $R_{1,0}(B_n)$ are identical in law, and we obtain
\[\varphi'(s) = \frac{\eta}{4} \E \big\langle(R_{1,2}(B_n) - q)^2 \big\rangle_s - \frac{\eta q^2}{4}.\]
Since the first in the right-hand side in the above display is non-negative, we obtain the lower bound $\varphi'(s) \ge - \frac{\eta q^2}{4}$ for all $s$. Integrating w.r.t.\ $s$ yields the lower bound   
\[f_{B_n}(\delta,\eta)  \ge  \fsc_{B_n}(\delta, \eta q) - \frac{\eta q^2}{4},\]
valid for all $n$ and all $q$.
From Proposition~\ref{prop:limit_Bn_scalar}, $\fsc_{B_n}$ converges to $\fsc$ as $n \to \infty$, hence the lower bound
\begin{equation}\label{eq:lower_bd_fr}
\liminf_{n \to \infty} f_{B_n}(\delta,\eta) \ge  \sup_{q\ge 0} \Big\{ \fsc(\delta, \eta q) - \frac{\eta q^2}{4}\Big\}.
\end{equation}

As for the upper, we use the same argument while fixing the value of the overlap $R_{1,0}(B_n)$ to a prescribed value $m$ along the interpolation. Define the set $O_n :=\{k/|B_n| : k \in \Z, |k|\le |B_n| \}$ and let $m \in O_n$. Consider the constrained free energy 
 \[\varphi(s;m) = \frac{1}{|B_n|} \E \log\Big\{2^{-|B_n|}\sum_{\theta \in \{\pm 1\}^{B_n}}  e^{\tilde{H}_s(\theta)} \bfone_{R_{1,0}(B_n)=m}\Big\},\]
 where
 \begin{align*}
\tilde{H}_{s}(\theta) &= H^{\delta}_{B_n}(\theta) + \sum_{u,v\in B_n} \Big(\sqrt{\frac{s\eta}{|B_n|}} Z_{uv}\theta_{u}\theta_{v} + \frac{s\eta}{|B_n|}\theta_{u}\theta_{v} \theta_{0u}\theta_{0v} - \frac{s\eta}{2|B_n|} \Big)\\
&~~~+ \sum_{u\in B_n} \Big( \sqrt{(1-s)\eta q} \, z_{u}\theta_{u} + (1-s)\eta m \theta_{u} \theta_{0u} - \frac{(1-s)\eta q}{2} \Big).
\end{align*}
(Note that this Hamiltonian no longer corresponds to an inference problem as before.)
The derivative reads
\[\varphi'(s) = -\frac{\eta}{4} \E \big\langle(R_{1,2}(B_n) - q)^2 \big\rangle_s + \frac{\eta}{2} \E \big\langle(R_{1,0}(B_n) - m)^2 \big\rangle_s - \frac{\eta m^2}{2}+ \frac{\eta q^2}{4}.\]
The middle term vanishes due the constraint on $R_{1,0}(B_n)$ and we obtain the upper bound $\varphi'(s)\le - \frac{\eta m^2}{2}+ \frac{\eta q^2}{4}$. Therefore
\begin{equation}\label{eq:interpolation_upper_bound}
\varphi(1;m) \le \inf_{q \ge 0} \Big\{\varphi(0;m) - \frac{\eta m^2}{2}+ \frac{\eta q^2}{4}\Big\}.
\end{equation}
On the one hand we have $\varphi(0;m) \le \widetilde{\fsc_{B_n}}(\delta,\eta m,\eta q)$, where
\begin{equation}\label{eq:fsc_tilde}
\widetilde{\fsc_{B_n}}(\delta,\eta m,\eta q) := \frac{1}{|B_n|} \E \log\Big\{2^{-|B_n|}\sum_{\theta \in \{\pm 1\}^{B_n}}  e^{H^{\delta}_{B_n}(\theta)+ \underset{u\in B_n}{\sum} (\sqrt{\eta q} z_{u}\theta_{u} + \eta m \theta_{u} \theta_{0u} - \frac{\eta q}{2}) }\Big\}.
\end{equation}
 We can follow the exact same argument leading to Proposition~\ref{prop:limit_Bn_scalar} to prove that $\widetilde{\fsc_{B_n}}(\delta,\eta m,\eta q)$ converges to a limit $\fsc(\delta, \eta m,\eta q)$ as $n \to \infty$; the only difference being $m$ and $q$ are not necessarily equal, but this does not alter the argument in any significant way. Moreover, the convergence is uniform on any compact set $K \subset [0,1] \times \R_+ \times \R_+$.  

On the other hand, let $Z_n(m) = 2^{-|B_n|}\sum_{\theta \in \{\pm 1\}^{B_n}}  e^{H_{B_n}^{\delta}(\theta)+H_{B_n}^{\eta}(\theta)} \bfone_{R_{1,0}(B_n)=m}$. We have
\begin{align}\label{eq:upp_bound_fr}
f_{B_n} &=\frac{1}{|B_n|} \E \log \sum_{m \in O_n} Z_n(m) \nonumber\\
&\le \frac{1}{|B_n|} \E \log \max_{m \in O_n} Z_n(m) + \frac{\log(2|B_n|+1)}{|B_n|} \nonumber\\
&= \frac{1}{|B_n|} \E \max_{m \in O_n}\log  Z_n(m) + \frac{\log(2|B_n|+1)}{|B_n|}.
\end{align}
At this point we want to show concentration of the random variables $F_n(m) := \log Z_n(m)$ so that we can compare $\E \max_{m} F_n(m)$ with $\max_m \E F_n(m)$. First, the random variables $\theta_{0u}$ are irrelevant as they can be absorbed in $\theta_u$ through the change of variables $\theta_u\theta_{0u} \mapsto \theta_u$ without changing the distribution of $\{F_n(m): m \in O_n\}$. So without loss of generality we set $\theta_{0u}=1$ for all $u \in B_n$. 
Now we deal with the randomness of $Z_{uv}$ and $Y^{\delta}_{uv}$. We use the so-called entropy method. We see $F_{n}(m)$ as a measurable Lipschitz  function of the random variables $(Z_{uv})_{u,v\in B_n}$ and $(Y^{\delta}_{uv})_{u,v\in B_n, |u-v|=1}$. Since we are only interested in an upper tail bound, it is enough to prove a logarithmic Sobolev inequality of the form 
\begin{equation}\label{eq:log_sobolev}
\Ent\big(e^{\gamma F_n}\big) \le \frac{\gamma^2\sigma_n^2}{2}\E\big[e^{\gamma F_n}\big] ~~~\mbox{for all}~~\gamma\ge 0,
\end{equation}
for some $\sigma_n>0$, where we define the `entropy' functional of a positive random variable $Y$  by
\[\Ent(Y) := \E Y\log Y - \E Y \log \E Y.\]
In our particular case, we will show that~\eqref{eq:log_sobolev} holds with $\sigma_n^2 = (\frac{\eta}{2}+2d c(\beta))|B_n|$ for some finite constant $c(\beta)>0$. We momentarily assume this to be true in order to finish the general argument, and then return to its proof later on. Herbst's argument; see~\cite{boucheron2013concentration}, then implies a sub-Gaussian bound on the moment generating function: 
\[\E e^{\gamma (F_n - \E F_n)} \le e^{\gamma^2\sigma_n^2/2}~~~\mbox{for all}~~\gamma\ge 0.\]
This in turn implies a bound on the expected maximum deviation $\E \max (F_n - \E F_n)$:
\begin{align*}
\E \Big[\max_{m \in O_n} \{F_n(m) - \E F_n(m)\}\Big] &\le \frac{1}{\gamma} \log\Big\{\sum_{m\in O_n} \E e^{\gamma (F_n(m) - \E F_n(m))}\Big\}\\
&\le \frac{\log (2|B_n|+1)}{\gamma} +\frac{\gamma \sigma_n^2}{2}.
\end{align*}
We let $\gamma = 1/\sqrt{|B_n|}$ and obtain that the above is bounded by $\bigo\big(\sqrt{|B_n|}\log |B_n|\big)$. 
Now, coming back to~\eqref{eq:upp_bound_fr}, we obtain
\[f_{B_n} \le \frac{1}{|B_n|} \max_{m \in O_n} \E \log Z_n(m) + \bigo\Big(\frac{\log |B_n|}{\sqrt{|B_n|}}\Big).\]
We notice that $\E \log Z_{n}(m) =\varphi(1;m)$, thus the interpolation upper bound~\eqref{eq:interpolation_upper_bound} and~\eqref{eq:fsc_tilde} imply
\begin{equation}
f_{B_n} \le \sup_{m \in [-1,1]} \inf_{q \ge 0} \Big\{\widetilde{\fsc_{B_n}}(\delta,\eta m, \eta q) - \frac{\eta m^2}{2}+ \frac{\eta q^2}{4}\Big\} + \bigo\Big(\frac{\log |B_n|}{\sqrt{|B_n|}}\Big).
\end{equation}
Since $\widetilde{\fsc_{B_n}}$ converges uniformly to $\widetilde{\fsc}$ on any compact set, this implies the upper bound
\begin{equation}\label{eq:upper_bd_fr}
\limsup_{n \to \infty} f_{B_n} \le \sup_{m \in [-1,1]} \inf_{q \ge 0} \Big\{\widetilde{\fsc}(\delta,\eta m, \eta q) - \frac{\eta m^2}{2}+ \frac{\eta q^2}{4}\Big\}.
\end{equation}

Next, we show that the upper and lower bounds~\eqref{eq:upper_bd_fr} and~\eqref{eq:lower_bd_fr} match. For any $m \in [-1,1]$ we obtain an upper bound on~\eqref{eq:upper_bd_fr} by letting $q= |m|$. Now observe that the map $r \mapsto \widetilde{\fsc}(\delta,r, s)$ is even since $r \mapsto \widetilde{\fsc_{B_n}}(\delta, r, s)$ is even. This implies the bound
 \[\limsup_{n \to \infty} f_{B_n} \le \sup_{m \in [-1,1]}\Big\{\widetilde{\fsc}(\delta,\eta |m|, \eta |m|) - \frac{\eta |m|^2}{4}\Big\},\]
 which is identical to the lower bound since $\widetilde{\fsc}(\delta, r, r) = \fsc(\delta,r)$ for $r \ge 0$, and thus we conclude that $f_{B_n}$ has a limit given by the above variational formula.
 
Now it remains to prove the log-Sobolev inequality~\eqref{eq:log_sobolev}. 

\begin{proof}[Proof of the bound~\eqref{eq:log_sobolev}]
We lighten the notation and denote by $X = (X_1,\cdots,X_N)$ the random variables $\{Z_{uv}:u,v \in B_n\} \cup \{Y^{\delta}_{uv}: u,v \in B_n, |u-v|=1\}$: we have $X_i = Z_{uv}$ for some one-to-one mapping $(u,v)\mapsto i$ for $1 \le i \le N_1$, and $X_i = Y^{\delta}_{uv}$ for $N_1+1 \le i\le N$. Denote $F = F(X_1,\cdots,X_N)$ the random variable $F_{n}(m) = \log Z_n(m)$. Denote $X^{\setminus k} = (X_1,\cdots,X_{k-1},X_{k+1},\cdots,X_N)$. The following tensorization formula for the entropy $\Ent$ is well known; see e.g.~\cite{boucheron2013concentration}: for all $\gamma \ge 0$,
\begin{equation}\label{eq:tensorization}
\Ent\big(e^{\gamma F}\big) \le \E\Big[\sum_{k=1}^N \Ent\big(e^{\gamma F} \big| X^{\setminus k}\big)\Big].
\end{equation}
Therefore is suffices to bound the entropies $\Ent\big(e^{\gamma F} \big| X^{\setminus k}\big)$ in which $F$ is seen as a univariate function of $X_{k}$. 
For $1\le k \le N_1$ the variable $X_k$ is Gaussian $N(0,1)$. We use the Gaussian logarithmic Sobolev inequality (Theorem 5.4 in~\cite{boucheron2013concentration}) 
\[\Ent\big(e^{\gamma F} \big| X^{\setminus k}\big) \le \frac{\gamma^2}{2} \E\big[(\partial_k F)^2e^{\gamma F} \big| X^{\setminus k}\big],\]   
valid for all $\gamma \ge 0$. Since $\partial_{k} F = \sqrt{\frac{\eta}{|B_n|}}\langle \theta_u\theta_v\rangle$, we have $|\partial_{k} F| \le \sqrt{\frac{\eta}{|B_n|}}$, and therefore
\begin{equation}\label{eq:log_sob_gaussian}
\Ent\big(e^{\gamma F} \big| X^{\setminus k}\big) \le \frac{\gamma^2\eta}{2|B_n|}\E\big[e^{\gamma F} \big| X^{\setminus k}\big].
\end{equation}
For $N_1+1\le k\le N$  the variable $X_k$ is Rademacher with probability $1-p$ for taking the value $1$ and probability $p$ for $-1$. We use the following logarithmic Sobolev inequality for Rademacher random variables (Theorem 5.2 in~\cite{boucheron2013concentration}) :
\[\Ent\big(e^{\gamma F} \big| X^{\setminus k}\big) \le \frac{c(p)}{2} \E\Big[\big(e^{\gamma F/2} - e^{\gamma \tilde{F}^{(k)}/2}\big)^2\big| X^{\setminus k}\Big],\]
where $c(p) = \frac{1}{1-2p}\log(\frac{1-p}{p}) = 2\beta/\delta$ and $\tilde{F}^{(k)} = F(X_1,\cdots,X_k',\cdots,X_N)$ where $X_k'$ is an independent copy of $X_k$.
By convexity of the exponential, $e^{t} \le e^{s} + (t-s)e^t$ for all $t,s \in \R$, so
\begin{align*}
\E\Big[\big(e^{\gamma F/2} - e^{\gamma \tilde{F}^{(k)}/2}\big)^2\big| X^{\setminus k}\Big] &\le 
\frac{\gamma^2}{4} \E\Big[(F -\tilde{F}^{(k)})^2 e^{\gamma F} \bfone_{F \ge \tilde{F}^{(k)}} \big| X^{\setminus k}\Big] \\
&~~+  \frac{\gamma^2}{4}\E\Big[(\tilde{F}^{(k)}-F)^2 e^{\gamma  \tilde{F}^{(k)}} \bfone_{F \le \tilde{F}^{(k)}} \big| X^{\setminus k}\Big]\\
&=\frac{\gamma^2}{2} \E\Big[(F -\tilde{F}^{(k)})^2 e^{\gamma F} \bfone_{F \ge \tilde{F}^{(k)}} \big| X^{\setminus k}\Big].
\end{align*}
The last line follows since the two terms in the first inequality are identical. Now since $\partial_k F = \beta \langle \theta_u\theta_v\rangle$ which is $\le \beta$ in absolute value, we have $|F - \tilde{F}^{(k)}| \le \beta |X_k - X_k'| \le 2\beta$. Plugging this estimate into the above bounds we obtain
\begin{equation}\label{eq:log_sob_rademacher}
\Ent\big(e^{\gamma F} \big| X^{\setminus k}\big) \le c(p)\beta^2\gamma^2 \E\big[e^{\gamma F} \big| X^{\setminus k}\big].
\end{equation}
Now we combine the two bounds~\eqref{eq:log_sob_gaussian} and~\eqref{eq:log_sob_rademacher} into~\eqref{eq:tensorization} to obtain
\begin{align*}
\Ent\big(e^{\gamma F}\big) &\le \Big(\frac{\gamma^2\eta}{2|B_n|} {|B_n|\choose2} + c(p)\beta^2\gamma^2 d |B_n|\Big) \E\big[e^{\gamma F} \big]\\
&\le \frac{\gamma^2\sigma_n^2}{2}\E\big[e^{\gamma F} \big],
\end{align*}
with $\sigma_n^2 = (\frac{\eta}{2} +2c(p)\beta^2 d)|B_n|$.
\end{proof}

\subsection{Proof of Propositions~\ref{prop:limits_phi1_phi2} and~\ref{prop:additional_scalar}}
\label{sec:proofs_limits}
In this subsection we prove convergence of the free energies of the one-block and two-block posteriors respectively. The difference with the setting of Proposition~\ref{prop:limit_Bn_GOE} is the inhomogeneity of the SNR in the GOE side information, e.g., pairs of vertices in $B \setminus B'$ have a different SNR than the ones in $B \cap B'$. Although we apply the same approach, the inhomogeneity will lead to technical complications which prevent us from proving equality of the upper and lower bounds achieved via the interpolation method. Nevertheless, Assumption~\ref{assump:lipschitz} allows us to salvage this situation for small $\eta$.    
 
We only write a detailed argument for Proposition~\ref{prop:limits_phi1_phi2}; the proof extends verbatim to Proposition~\ref{prop:additional_scalar}. 
    
Recall that for a given block $B$, the one-block posterior is given by
\begin{equation*}\label{eq:gibbs_one_block2}
\P\Big(\big(\theta_x\big)_{x\in B} \big| Y^{\delta}_B , Y^{\eta}_B\Big) = \frac{1}{Z_B}  \omega^{\sSK}_B(\theta) \cdot e^{H_{B}^{\delta}(\theta)},
\end{equation*}
and for two adjacent blocks $B$ and $B'$, the two-block posterior is
\begin{equation*}\label{eq:gibbs_two_block2}
\P\Big(\big(\theta_x\big)_{x\in B \cup B'} \big| \bigcup_{A \in \{B,B'\}} \{Y^{\delta}_A , Y^{\eta}_A \}\Big) = \frac{1}{Z_{B\cup B'}} \omega_{B \cup B'}^{\sSK}(\theta) \cdot e^{H_{B\cup B'}^{\delta}(\theta)},
\end{equation*} 
where 
\[H_{B \cup B'}^{\delta}(\theta) := \beta\sum_{\underset{|u-v|=1}{u,v\in B \cup B'}} Y^{\delta}_{uv}\theta_u\theta_v,\] 
and for $\theta = (\theta_B,\theta_{B'}) \in \{\pm 1\}^{B \cup B'}$,
\begin{align*}
\omega_{B\cup B'}^{\sSK}(\theta) &:= \prod_{B'' \sim B} e^{H^{\eta}_{B''}(\theta_{B})} 
\cdot \prod_{B'' \sim B'} e^{H^{\eta}_{B''}(\theta_{B'})}.
\end{align*}
 The free energies of the above posteriors are 
\begin{align*} 
f_{B} &= \frac{1}{|B|} \E \log\Big\{2^{-|B|} \sum_{\theta \in \{\pm 1\}^{B}} \omega^{\sSK}_B(\theta) \cdot e^{H_{B}^{\delta}(\theta)} \Big\} ~~~\mbox{and}\\  
f_{B\cup B'} &= \frac{1}{|B|}\E \log\Big\{2^{-|B \cup B'|} \sum_{\theta \in \{\pm 1\}^{B\cup B'}} \omega^{\sSK}_{B \cup B'}(\theta) \cdot e^{H_{B \cup B'}^{\delta}(\theta)} \Big\}.
\end{align*}

\subsection{Limit of $f_{B}$}
The same arguments already used can be applied: rather than having one parameter in the interpolation argument, we will have several, each one assigned to a specific spatial region. As before, we prove the lower and upper bounds separately.    
The overlap replacement in the interpolation proceeds according to the following table with $2d+1$ free parameters: 
\begin{align*}  
R_{1,2}\big(B^{\bullet}\big) ~&\longleftrightarrow~ q^{\bullet}\\
R_{1,2}(B \cap B') ~&\longleftrightarrow~ q_{B\cap B'}, ~~~ \forall ~B' \sim B,
\end{align*} 
In the above, we denote $B^{\bullet} = B \setminus \underset{B' \sim B}{\cup} B'$.
Overlaps in other regions are then obtained as convex combinations of $\{q^{\bullet},q_{B \cap B'}: B' \sim B\}$. For instance, with $\alpha' = \frac{|B \cap B'|}{|B|}$, we have
\begin{align*}  
R_{1,2}(B ) ~&\longleftrightarrow~ q_{B} := (1-2d\alpha') q^{\bullet} + \sum_{B'\sim B} \alpha' q_{B \cap B'},\\
R_{1,2}(B \setminus B') ~&\longleftrightarrow~ q_{B \setminus B'} := \frac{1-2d\alpha'}{1-\alpha'} q^{\bullet} + \sum_{\underset{B'' \neq B'}{B''\sim B}} \frac{\alpha'}{1-\alpha'} q_{B \cap B''}.
\end{align*} 
For all the arguments to come, we will be neglecting a $\smallo_{\scL}(1)$ error stemming from replacing $\alpha'$ which a priori varies with $\scL$ with its limit $\alpha$. This can be safely done because all involved quantities are Lipschitz in $\alpha'$. So tacitly assume $\alpha'=\alpha$ with no further comment. 

We let $\hsc(\theta_{B}; \lambda) :=\sum_{x \in B} (\sqrt{\lambda} y_x \theta_x - \lambda/2)$, where $y_x = \sqrt{\lambda}\theta_x + z_x$.
With this notation, we interpolate between the two Hamiltonians
\begin{align*}
 H^{\eta}_{B'}(\theta_B) ~&\longleftrightarrow~ \hsc(\theta_B; t \eta q_{B}) + \hsc(\theta_{B\cap B'}; (1-t) \eta q_{B \cap B'})  + \hsc(\theta_{B\setminus B'}; (1-t) \eta q_{B \setminus B'}).
\end{align*}
Interpolating \`a la Guerra, the derivative along the interpolation path reads
\begin{align*} 
\varphi' &= \sum_{B' \sim B} \Big\{ \frac{t\eta}{4} \Big(\E \big\langle \big(R_{1,2}(B) - q_B\big)^2\big\rangle - q_B^2\Big)\\
&~~~~~~~~~~~+  \alpha \frac{(1-t)\eta}{4} \Big(\E \big\langle \big(R_{1,2}(B\cap B') - q_{B\cap B'}\big)^2\big\rangle - q_{B\cap B'}^2\Big)\\
&~~~~~~~~~~~+ (1-\alpha) \frac{(1-t)\eta}{4} \Big( \E \big\langle \big(R_{1,2}(B\setminus B') - q_{B\setminus B'}\big)^2\big\rangle - q_{B\setminus B'}^2\Big) \Big\}.
\end{align*}
Dropping the squared differences which are nonnegative we obtain a lower bound
 \begin{align*} 
\varphi' &\ge 
-\frac{\eta}{4} \sum_{B' \sim B}\Big( t q_B^2
+  \alpha (1-t) q_{B\cap B'}^2
+ (1-\alpha) (1-t) q_{B\setminus B'}^2 \Big).
\end{align*}
At the decoupled end of the interpolation, the free energy reads
 \begin{align*}
 \varphi(0) = \frac{1}{|B|}\E \log \Big\{ 2^{-|B|}\sum_{\theta \in \{\pm 1\}^B} e^{H_{B}^{\delta}(\theta)} \omega^{\sc}(\theta)\Big\},
\end{align*}
where
 \begin{align*}
  \omega^{\sc}(\theta) &=
  \prod_{B' \sim B} \exp\Big\{\hsc(\theta_B; t \eta q_{B}) + \hsc(\theta_{B\cap B'}; (1-t) \eta q_{B \cap B'})  + \hsc(\theta_{B\setminus B'}; (1-t) \eta q_{B \setminus B'})\Big\} \\
 &\stackrel{d}{=}  \exp\big\{\hsc(\theta_{B^{\bullet}}; r^{\bullet})\} 
 \cdot \prod_{B'\sim B}\exp\big\{ \hsc(\theta_{B\cap B'};r_{B \cap B'})\big\},
\end{align*}
where 
\begin{align}
r^{\bullet} &:=  \sum_{B'\sim B} (t \eta q_{B} + (1-t) \eta q_{B \setminus B'}),\label{eq:r_bullet}\\
r_{B \cap B'} &:=   2dt \eta q_{B}+(1-t) \eta q_{B \cap B'} + \sum_{\underset{B'' \neq B'}{B''\sim B}} (1-t) \eta q_{B \setminus B''}.\label{eq:r_cap}
\end{align}
From this expression we can obtain an asymptotic formula for $\varphi(0)$ as $\scL \to \infty$. We partition $B$ into $B^{\bullet} \cup_{B' \sim B} (B \cap B')$, and observe that omitting the interactions corresponding to the edges of $\Z^d$ crossing the boundaries of this partition in the Hamiltonian $H^{\delta}_B(\theta)$ contributes an error term $o_{\scL}(1)$:  
\[\lim_{\scL \to \infty} \varphi(0) = (1-2d\alpha)\fsc(\delta;r^{\bullet}) +  \alpha \sum_{B'\sim B} \fsc(\delta;r_{B'}).\] 
Thus we obtain a lower bound on the free energy $f_B$: 
\begin{equation}\label{eq:lower_bd_one_block}
\liminf_{\scL \to \infty} f_B \ge \phi_1,
\end{equation} 
where, letting $\underline{q} = \{q^{\bullet}, q_{B\cap B'}: B'\sim B\}$,
\begin{align}\label{eq:free_energy_lower_bd_one_block}
\phi_1 :=&  
\sup_{\underline{q} \in [0,1]^{2d+1}} \Big\{ 
(1-2d\alpha)\fsc(\delta;r^{\bullet}) +  \alpha \sum_{B'\sim B} \fsc(\delta;r_{B\cap B'}) \nonumber\\
&\hspace{2cm}- \frac{1}{4}\sum_{B' \sim B}\Big( t\eta q_B^2
+  \alpha (1-t)\eta q_{B\cap B'}^2
+ (1-\alpha) (1-t)\eta q_{B\setminus B'}^2 \Big)\Big\} .
\end{align}
Now we prove an upper bound by restricting to configurations of fixed overlaps $\underline{m} = \{m^{\bullet}, m_{B\cap B'}: B'\sim B\}$, exactly in the same way as done for $f_{B_n}(\delta,\eta)$. 
We consider the restricted free energy
  \[\varphi(s;\underline{m}) = \frac{1}{|B|} \E \log\Big\{2^{-|B|}\sum_{\theta \in \{\pm 1\}^{B}}  \omega_{B}^{\sSK}(\theta;s) e^{H^{\delta}_{B}(\theta)} 
  \bfone_{
  \underset{R_{1,0}(B^{\bullet}) = m^{\bullet}}{
  R_{1,0}(B\cap B') = m_{B \cap B'} \forall B' \sim B,}
 } \Big\} \]
 where
\[ \omega_{B}^{\sSK}(\theta;s) = \prod_{B' \sim B} e^{H^{\eta}_{B'}(\theta;s)},\]
and for each $B' \sim B$,
\[H^{\eta}_{B'}(\theta;s) = I + II +III,\]
\begin{align*}
I &= \sum_{u,v\in B} \Big(\sqrt{\frac{ts\eta}{|B|}} Z_{uv}^{(a,\bullet)}\theta_{u}\theta_{v} + \frac{ts\eta}{|B|}\theta_{u}\theta_{v} \theta_{0u}\theta_{0v} - \frac{ts\eta}{2|B|} \Big)\\
&~~~+ \sum_{u\in B} \Big( \sqrt{(1-s)t\eta q_B} \, z^{(a,\bullet)}_{u}\theta_{u} + (1-s)t\eta m_B \theta_{u} \theta_{0u} - \frac{(1-s)t\eta q_B}{2} \Big),\\
II &=  \sum_{u,v\in B \cap B'} \Big(\sqrt{\frac{(1-t)s\eta}{|B\cap B'|}} Z_{uv}^{(a,\cap)}\theta_{u}\theta_{v} + \frac{(1-t)s\eta}{|B \cap B'|}\theta_{u}\theta_{v} \theta_{0u}\theta_{0v}- \frac{(1-t)s\eta}{2|B\cap B'|}\Big)\\
&~~~+\sum_{u\in B \cap B'} \Big( \sqrt{(1-t)(1-s)\eta q_{B\cap B'}} \, z_{u}^{(a,\cap)}\theta_{u} + (1-s)(1-t)\eta m_{B\cap B'} \theta_{u} \theta_{0u} - \frac{(1-t)(1-s)\eta q_{B\cap B'}}{2} \Big)\\
III &= \sum_{u,v\in B \setminus B'} \Big(\sqrt{\frac{(1-t)s\eta}{|B\setminus B'|}} Z_{uv}^{(a,\setminus)}\theta_{u}\theta_{v} + \frac{(1-t)s\eta}{|B \setminus B'|}\theta_{u}\theta_{v} \theta_{0u}\theta_{0v} - \frac{(1-t)s\eta}{2|B\setminus B'|}\Big)\\
&~~~+\sum_{u\in B \setminus B'} \Big( \sqrt{(1-t)(1-s)\eta q_{B\setminus B'}} \, z_{u}^{(a,\setminus)}\theta_{u} + (1-s)(1-t)\eta m_{B\setminus B'} \theta_{u} \theta_{0u} - \frac{(1-t)(1-s)\eta q_{B\setminus B'}}{2} \Big),
\end{align*}
where
\begin{align*}
m_{B} &= (1-2d\alpha)m^{\bullet} + \sum_{B' \sim B} \alpha m_{B \cap B'} ~~~\mbox{and}\\
m_{B \setminus B'} &=  \frac{1-2d\alpha}{1-\alpha} m^{\bullet} + \sum_{\underset{B'' \neq B'}{B''\sim B}} \frac{\alpha}{1-\alpha} m_{B \cap B''}.
\end{align*}
Observe that at $s=1$ the Hamiltonian $H^{\eta}_{B'}(\theta;s)$ is equal to the $H^{\eta}_{B'}(\theta)$ given in Eq.~\eqref{eq:hamiltonian_goe_inhomogeneous}. Conducting exactly the same argument as in subsection~\ref{sec:proof_limit_Bn_GOE} we obtain an upper bound
\begin{equation}\label{eq:upper_bd_one_block}
\limsup_{\scL \to \infty} f_{B}  \le \tilde{\phi}_1,
\end{equation}
where
\begin{align}\label{eq:free_energy_upper_bd_one_block}
\tilde{\phi}_1 &= \sup_{\underline{m} \in [-1,1]^{2d+1}} \inf_{\underline{q} \in [0,1]^{2d+1}} 
\Big\{
(1-2d\alpha) \widetilde{\fsc}(\delta,s^{\bullet}, r^{\bullet}) + \alpha \sum_{B'\sim B} \widetilde{\fsc}(\delta;s_{B\cap B'},r_{B\cap B'}) \nonumber\\
&\hspace{2cm}- \frac{1}{2}\sum_{B' \sim B}\Big( t\eta m_B^2
+  \alpha (1-t)\eta m_{B\cap B'}^2
+ (1-\alpha) (1-t)\eta m_{B\setminus B'}^2 \Big)\nonumber\\
&\hspace{2cm}+ \frac{1}{4}\sum_{B' \sim B}\Big( t\eta q_B^2
+  \alpha (1-t)\eta q_{B\cap B'}^2
+ (1-\alpha) (1-t)\eta q_{B\setminus B'}^2 \Big)\Big\},
\end{align} 
where $r^{\bullet}$ and $r_{B \cap B'}$ are defined in~\eqref{eq:r_bullet} and~\eqref{eq:r_cap} respectively, and $s^{\bullet}$ and $s_{B \cap B'}$ are defined similarly, with $m_{B}$, $m_{B \cap B'}$ and $m_{B \setminus B'}$ replacing $q_{B}$, $q_{B \cap B'}$ and $q_{B \setminus B'}$ respectively in those definitions.

Next, if we could argue that the supremum in $\underline{m}$ in formula~\eqref{eq:free_energy_upper_bd_one_block} is achieved when all coordinates of $\underline{m}$ are nonnegative, then we can see that this formula matches the lower bound~\eqref{eq:free_energy_lower_bd_one_block} simply by setting $q^{\bullet} = m^{\bullet}$ and $q_{B \cap B'}= m_{B \cap B'}$ for all $B'\sim B$, and we would be able prove that  $\underset{\scL \to \infty}{\lim} f_B = \phi_1$. Unfortunately we are unable to follow this line of reasoning: The fact that the map $s \mapsto  \widetilde{\fsc}(\delta;s,r)$ is even provides little leverage to argue about every single coordinate of $\underline{m}$. 
Nevertheless we will see that under Assumption~\ref{assump:lipschitz} both formulas can be simplified: the optimal values of the parameters must all be equal, and this implies equality of $\phi_1$ and $\tilde{\phi}_1$.

\subsection{Reducing the number of parameters}
We show that under Assumption~\ref{assump:lipschitz}, the variational formulas defining $\phi_1$ Eq.~\eqref{eq:free_energy_lower_bd_one_block} and $\tilde{\phi}_1$ Eq.~\eqref{eq:free_energy_upper_bd_one_block} are achieved when all the overlap parameters become equal when $\eta$ is small. This allows to argue that $\phi_1 = \tilde{\phi}_1$ in this regime, and thus leads to the simplified formula $\phi_1 = \phi(\delta,2d \eta)$, Eq.\eqref{eq:var_1}. 

We start by writing the first-order stationarity conditions for the optimality of $\underline{q}$. Let $F = F(\underline{q})$ be the potential being maximized in~\eqref{eq:free_energy_lower_bd_one_block}. Via straightforward computations we can write the derivatives of the potential $F$ w.r.t.\  $\underline{q}$ in a condensed form. 
We fix an ordering of the variables $\{q^{\bullet}, q_{B \cap B'}: B'\sim B\}$ and consider $\underline{q}$ as a vector of size $2d+1$.   
For all $i\in [2d+1]$, 
\begin{align*}
\frac{\rmd F}{\rmd \underline{q}_i} &= (1-2\alpha d) \frac{\rmd r^{\bullet}}{\rmd \underline{q}_i} \cdot \big({\fsc}'(\delta; r^{\bullet}) - \frac{q^{\bullet}}{2}\big) 
+  \alpha \sum_{B'\sim B}\frac{\rmd r_{B \cap B'}}{\rmd \underline{q}_i}  \cdot\big({\fsc}'(\delta;r_{B \cap B'}) -  \frac{q_{B \cap B'}}{2}\big).
\end{align*}
 Observe from the above formula that optimizers $\underline{q}$ of~\eqref{eq:free_energy_lower_bd_one_block} are in the interior of $\R_+^{2d+1}$ if $\delta>\delta_c$. Indeed, since $\frac{\rmd}{\rmd \lambda} \fsc(\delta;\lambda)\ge \frac{\rmd}{\rmd \lambda} \fsc(\delta;0^+)>0$ when $\delta>\delta_c$, so the value $F( \underline{q})$ for a point $\underline{q}$ on the boundary can always be improved. 

Define the $2d+1 \times 2d +1$ Jacobian matrix $J$ whose columns are the vectors $(1-2\alpha d)(\frac{\rmd r^{\bullet}}{\rmd \underline{q}_i})_{i\in [2d+1]}$,  $\alpha(\frac{\rmd r_{B'}}{\rmd \underline{q}_i})_{i\in [2d+1]}$ for $B' \sim B$.  
With this notation the gradient of the potential $F$ is given by
\[\Big(\frac{\rmd F}{\rmd \underline{q}_i}\Big)_{i \in [2d+1]} = J\underline{x},\] 
where $\underline{x} \in \R^{2d+1}$ has coordinates ${\fsc}'(\delta; r^{\bullet}) - \frac{q^{\bullet}}{2}$ and ${\fsc}'(\delta;r_{B \cap B'}) -  \frac{q_{B \cap B'}}{2}$, $B' \sim B$.
This matrix only depends on the parameters $t$ and $\eta$. Moreover, 
\begin{lemma}\label{lem:invertible}
If $t <1$ and $2\alpha d < 1$ then the matrix $J$ is invertible.
\end{lemma}
\begin{proof}
After some manipulations, the matrix $J$ can be written as 
\begin{align*}
J = 
\begin{pmatrix}
a & \gamma & \cdots& \cdots &\gamma\\
\gamma & b & \beta & \cdots &\beta \\ 
\vdots &  \beta & \ddots & \ddots & \vdots\\
\vdots &  \vdots & \ddots & \ddots & \beta\\
\gamma & \beta & \cdots & \beta&   b 
\end{pmatrix} \, ,
\end{align*}
where 
\begin{align*}
a &= 2d (1-2\alpha d)^2 \frac{1-\alpha t}{1-\alpha}, ~~~ b = \alpha (1-(1-2\alpha d)t),\\
 \beta &= \frac{\alpha^2}{1-\alpha} (2d-1+ (1-2\alpha d)t), ~~~ \gamma = \frac{\alpha}{1-\alpha} (1-2\alpha d)(2d-1+ (1-2\alpha d)t).
\end{align*}
Let $A$ be the lower principal minor of size $2d \times 2d$ (i.e., $A$ is obtained by excluding the first row and first column of $J$). We can see that $A = (b-\beta) I + \beta \bfone \bfone^\top$.
Schur's complement formula implies that $J$ is invertible if and only if $a \neq 0$ and $A - \frac{\gamma^2}{a} \bfone \bfone^\top$ is invertible. We have
\[A - \frac{\gamma^2}{a} \bfone \bfone^\top = (b-\beta)I + (\beta - \frac{\gamma^2}{a}) \bfone \bfone^\top.\] 
Direct computations show that
\begin{align*}
b - \beta &= \frac{\alpha}{1-\alpha}(1-t)(1-2\alpha d),\\
\beta - \frac{\gamma^2}{a} &= \frac{\alpha^2 D}{(1-\alpha)(1-\alpha t)} \big((1-(1-\alpha)^2)(2d-1+(1-2\alpha d)t)+1-t\big) ,
\end{align*}
with $D = \frac{2d -1+ (1-2\alpha d)t}{2d}$. The above quantities are strictly positive when $t <1$ and $2\alpha d < 1$ so $J$ is invertible under these conditions.
\end{proof}
Therefore, the first-order stationarity conditions for $F$ imply $\underline{x}=0$, i.e., 
\begin{equation}\label{eq:first_order}
{\fsc}'(\delta; r^{\bullet}) = \frac{q^{\bullet}}{2}\quad \mbox{and} \quad
{\fsc}'(\delta;r_{B \cap B'}) =  \frac{q_{B \cap B'}}{2}~~~ \forall~ B' \sim B.
\end{equation}

 \begin{lemma}\label{lem:symmetry}
 Under Assumption~\ref{assump:lipschitz}, there exists $\eta_0>0$ depending on $d$, $\alpha$ and $L$ such that if $\eta< \eta_0$, then all solutions $\underline{q}$ to~\eqref{eq:first_order} must satisfy
 \[q^{\bullet} = q_{B \cap B'} ~~~\mbox{for all}~ B'\sim B.\]  
 \end{lemma}
 \begin{proof}
 We compute pairwise differences of the equations in~\eqref{eq:first_order} and use the assumed Lipchitz property of $\frac{\rmd}{\rmd \lambda}{\fsc}$, as per item A1 of Assumption~\ref{assump:lipschitz}:
 for $B'$ and $B''$ two neighbors of $B$ we have
 \begin{align*}
 |q^{\bullet} - q_{B \cap B'}| &\le 2 L |r^{\bullet} - r_{B \cap B'}|,\\ 
 |q_{B \cap B'} - q_{B \cap B''}|  &\le 2L |r_{B \cap B'} - r_{B \cap B''}|.
 \end{align*}
 On the other hand we have
 \begin{align*}
 r^{\bullet} - r_{B \cap B'} &= (1-t)\eta (q_{B \setminus B'} - q_{B \cap B'})\\
 &= \frac{(1-t)\eta}{1-\alpha} (q_{B} - q_{B \cap B'}),\\
 r_{B \cap B'} - r_{B \cap B''} &= (1-t)\eta \big((q_{B \cap B'}-q_{B \cap B''}) - (q_{B \setminus B'} -q_{B \setminus B''})\big)\\
&~~~= \frac{(1-t)\eta}{1-\alpha} (q_{B \cap B'} - q_{B \cap B''}).
 \end{align*}
so
 \begin{align*}
 |q^{\bullet} - q_{B \cap B'}| &\le \frac{2L\eta}{1-\alpha}|q_{B} - q_{B \cap B'}|,\\
 |q_{B \cap B'} - q_{B \cap B''}|   &\le \frac{2L\eta}{1-\alpha}|q_{B\cap B'} - q_{B \cap B''}|.
 \end{align*}
If $\frac{2L\eta}{1-\alpha} <1$ then the second inequality implies $q_{B \cap B'} = q_{B \cap B''}$ and so $q_B= (1-2d\alpha)q^{\bullet}+2d\alpha q_{B \cap B'}$ for any $B'\sim B$. The first inequality then implies 
\[|q^{\bullet} - q_{B \cap B'}| \le \frac{2L\eta}{1-\alpha} (1-2d\alpha)|q^{\bullet} - q_{B \cap B'}|.\]
If $\frac{2L\eta}{1-\alpha} (1-2d\alpha) <1$ then $q^{\bullet} = q_{B\cap B'}$ for all $B'\sim B$. 
 \end{proof}
 
 Lemma~\ref{lem:symmetry} implies that the variational formula~\eqref{eq:free_energy_lower_bd_one_block} can be simplified when $\eta$ is small: 
 \begin{corollary}
Under Assumption~\ref{assump:lipschitz}, there exists $\eta_0>0$ such that for $\eta \in \R_+ \setminus \mathcal{D}$ and $\eta\le \eta_0$,
\[\phi_1 = \sup_{q \in [0,1]} \Big\{\fsc(\delta;2 d \eta q) - \frac{d}{2} \eta {q}^2\Big\} .\]
In particular, $\phi_1$ doesn't depend on $t$.
 \end{corollary}
 
 Now we consider $\tilde{\phi}_1$ and apply the same argument. The saddle point conditions of optimality of $\underline{m}$ and $\underline{q}$ in the max-min problem~\eqref{eq:free_energy_upper_bd_one_block}  are 
 \begin{align}
\frac{\partial F}{\partial m_i} &= 0~~\mbox{for all}~ i \in [2d+1], \label{eq:saddle_1}\\
\frac{\partial F}{\partial q_i} &=0 ~~~\mbox{or}~~~ q_i=0 ~~\mbox{for all}~ i \in [2d+1], \label{eq:saddle_2}.
\end{align}
 Condition~\eqref{eq:saddle_1} can be rewritten as $ J\underline{x} =0$ where $J$ is the matrix of Lemma~\ref{lem:invertible} and $\underline{x}$ has coordinates $\partial_s\tilde{\fsc}(\delta; s^{\bullet},r^{\bullet}) - m^{\bullet}$ and $\partial_s \tilde{\fsc}(\delta;s_{B \cap B'}, r_{B \cap B'}) -  m_{B \cap B'}$, $B' \sim B$.
 \begin{lemma}  
Under Assumption~\ref{assump:lipschitz}, there exists $\eta_0>0$ depending on $d$, $\alpha$ and $L$ such that if $\eta< \eta_0$, then all solutions $(\underline{m},\underline{q})$ to~\eqref{eq:saddle_1}-\eqref{eq:saddle_2} must satisfy
\[m^{\bullet} = m_{B \cap B'}~~~\mbox{for all}~ B'\sim B.\]
 \end{lemma}
 \begin{proof}
 Item A2 of Assumption~\ref{assump:lipschitz} says that the map $s \mapsto \partial_s \widetilde{\fsc}(\delta;s, r)$ is Lipschitz uniformly in $r$. Therefore the same argument of Lemma~\ref{lem:symmetry} can be applied to reason that~\eqref{eq:saddle_1} implies $m^{\bullet} = m_{B \cap B'}$ for all $B'\sim B$.
 \end{proof}
 Given the result of the above lemma, we can simplify the formula~\eqref{eq:free_energy_upper_bd_one_block} for $\tilde{\phi}_1$ as follows:
\begin{align*}
\tilde{\phi}_1 &= \sup_{m \in [-1,1]} \inf_{\underline{q} \in [0,1]^{2d+1}} 
\Big\{
(1-2d\alpha) \widetilde{\fsc}(\delta,2d\eta m, r^{\bullet}) + \alpha \sum_{B'\sim B} \widetilde{\fsc}(\delta;2d\eta m,r_{B\cap B'}) \nonumber\\
&\hspace{4cm}- d \eta m^2\\
&\hspace{2cm}+ \frac{1}{4}\sum_{B' \sim B}\Big( t\eta q_B^2
+  \alpha (1-t)\eta q_{B\cap B'}^2
+ (1-\alpha) (1-t)\eta q_{B\setminus B'}^2 \Big)\Big\}.\nonumber
\end{align*} 
Since $s \mapsto  \widetilde{\fsc}(\delta;s,r)$ is even there is no loss in generality in assuming $m \ge 0$. Now we obtain $\tilde{\phi}_1 \le \phi_1$ by setting all coordinates of $\underline{q}$ to the common value $m$. Finally, this implies equality:
 \begin{corollary}
Under Assumption~\ref{assump:lipschitz}, there exists $\eta_0>0$ such that for $\eta \in \R_+ \setminus \mathcal{D}$ and $\eta\le \eta_0$ and for all $t \in [0,1]$ we have
\[\tilde{\phi}_1 = \phi_1 = \sup_{q \in [0,1]} \Big\{\fsc(\delta;2 d \eta q) - \frac{d}{2} \eta {q}^2\Big\} .\]
Hence the one-block free energy has a limit 
\[\underset{\scL \to \infty}{\lim} f_B = \phi_1.\]
 \end{corollary}

\subsection{Limit of $f_{B \cup B'}$}
We proceed similarly to the case of the one-block free energy. In this case we have  $4d+1$ free parameters. Recall the definition of the core $B^{\bullet}$ of a block $B$: $B^{\bullet} = B \setminus \underset{B'' \sim B}{\cup} B''$, and ${B'}^{\bullet} = B' \setminus  \underset{B'' \sim B'}{\cup} B''$. 
We summarize the correspondence between overlaps and free parameters in the following table:
\begin{align*}  
R_{1,2}\big(B^{\bullet}\big) ~&\longleftrightarrow~ q^{\bullet}_B\\
R_{1,2}\big({B'}^{\bullet}\big) ~&\longleftrightarrow~ q^{\bullet}_{B'}\\
R_{1,2}(B \cap B'') ~&\longleftrightarrow~ q_{B\cap B''}, ~~~ \forall ~B'' \sim B,\\
R_{1,2}(B' \cap B'') ~&\longleftrightarrow~ q_{B'\cap B''}, ~~~ \forall ~B'' \sim B'.
\end{align*} 
We obtain the remaining relevant overlaps as 
\begin{align*}  
R_{1,2}(B ) ~&\longleftrightarrow~ q_{B} := (1-2d\alpha) q^{\bullet}_B + \sum_{B''\sim B} \alpha q_{B \cap B''},\\
R_{1,2}(B') ~&\longleftrightarrow~ q_{B'} := (1-2d\alpha) q^{\bullet}_{B'} +\sum_{B''\sim B'} \alpha q_{B' \cap B''},\\
R_{1,2}(B \setminus B'') ~&\longleftrightarrow~ q_{B \setminus B''} := \frac{1-2d\alpha}{1-\alpha} q^{\bullet}_{B} +  \frac{\alpha}{1-\alpha} \sum_{B''' \sim B} q_{B \cap B'''},~~ \forall B'' \sim B,\\
R_{1,2}(B' \setminus B'') ~&\longleftrightarrow~ q_{B' \setminus B''} := \frac{1-2d\alpha}{1-\alpha} q^{\bullet}_{B'} + \sum_{B'''\sim B'} \frac{\alpha}{1-\alpha} q_{B' \cap B'''},~~ \forall B'' \sim B'.
\end{align*} 
We execute the same interpolation as in the previous section:
\begin{align*}
 H^{\eta}_{B''}(\theta_B) ~&\longleftrightarrow~ \hsc(\theta_B; t \eta q_{B}) + \hsc(\theta_{B\cap B''}; (1-t) \eta q_{B \cap B''}) \qquad \forall~B''\sim B\\
 &\hspace{1cm}+  \hsc(\theta_{B\setminus B''}; (1-t) \eta q_{B \setminus B''}),\\
 H^{\eta}_{B''}(\theta_{B'}) ~&\longleftrightarrow~ \hsc(\theta_{B'}; t \eta q_{B'}) + \hsc(\theta_{B'\cap B''}; (1-t) \eta q_{B' \cap B''}) \qquad \forall~B''\sim B'\\
 &\hspace{1cm}+  \hsc(\theta_{B'\setminus B''}; (1-t) \eta q_{B' \setminus B''}).\\
\end{align*}

We proceed exactly as for the computation of $f_B$, and obtain upper and lower bounds on limiting free energy: 
\[\liminf_{\scL \to \infty} f_{B \cup B'} \ge \phi_2, ~~~\mbox{and}~~~
\limsup_{\scL \to \infty} f_{B \cup B'} \le \tilde{\phi}_2,\]
with
\begin{align}\label{eq:free_energy_lower_bd_two_block}
\phi_2 :=
\sup_{\underline{q} \in [0,1]^{4d+1}} \Big\{& 
(1-2d\alpha)\fsc(\delta;r^{\bullet}_B) + (1-2d\alpha)\fsc(\delta;r^{\bullet}_{B'}) 
+  \alpha \fsc(\delta;r_{B\cap B'}) \\
&+  \sum_{\underset{B'' \neq B'}{B''\sim B}} \alpha \fsc(\delta;r_{B \cap B''})
+  \sum_{\underset{B'' \neq B}{B''\sim B'}} \alpha \fsc(\delta;r_{B'\cap B''})\nonumber\\
&- \frac{1}{4}\sum_{B'' \sim B}\Big( t\eta q_B^2
+  \alpha (1-t)\eta q_{B\cap B''}^2
+ (1-\alpha) (1-t)\eta q_{B\setminus B''}^2 \Big)\nonumber\\
&- \frac{1}{4}\sum_{B'' \sim B'}\Big( t\eta q_{B'}^2
+  \alpha (1-t)\eta q_{B'\cap B''}^2
+ (1-\alpha) (1-t)\eta q_{B'\setminus B''}^2 \Big)\Big\}\nonumber.
\end{align}  
and
\begin{align}\label{eq:free_energy_upper_bd_two_block}
\tilde{\phi}_2 :=
\inf_{\underline{m} \in [-1,1]^{4d+1}} \sup_{\underline{q} \in [0,1]^{4d+1}} \Big\{& 
(1-2d\alpha)\widetilde{\fsc}(\delta;s^{\bullet}_B,r^{\bullet}_B) + (1-2d\alpha)\widetilde{\fsc}(\delta;s^{\bullet}_{B'},r^{\bullet}_{B'}) 
+  \alpha \widetilde{\fsc}(\delta;s_{B \cap B'},r_{B\cap B'}) \nonumber\\
&+  \sum_{\underset{B'' \neq B'}{B''\sim B}} \alpha \widetilde{\fsc}(\delta;s_{B \cap B''},r_{B \cap B''})
+  \sum_{\underset{B'' \neq B}{B''\sim B'}} \alpha \widetilde{\fsc}(\delta;s_{B' \cap B''},r_{B'\cap B''})\nonumber\\
&- \frac{1}{2}\sum_{B'' \sim B}\Big( t\eta m_B^2
+  \alpha (1-t)\eta m_{B\cap B''}^2
+ (1-\alpha) (1-t)\eta m_{B\setminus B''}^2 \Big)\nonumber\\
&- \frac{1}{2}\sum_{B'' \sim B'}\Big( t\eta m_{B'}^2
+  \alpha (1-t)\eta m_{B'\cap B''}^2
+ (1-\alpha) (1-t)\eta m_{B'\setminus B''}^2 \Big)\nonumber\\
&+ \frac{1}{4}\sum_{B'' \sim B}\Big( t\eta q_B^2
+  \alpha (1-t)\eta q_{B\cap B''}^2
+ (1-\alpha) (1-t)\eta q_{B\setminus B''}^2 \Big)\nonumber\\
&+ \frac{1}{4}\sum_{B'' \sim B'}\Big( t\eta q_{B'}^2
+  \alpha (1-t)\eta q_{B'\cap B''}^2
+ (1-\alpha) (1-t)\eta q_{B'\setminus B''}^2 \Big)\Big\}.
\end{align} 
With `$r$' parameters
\begin{align*}
r^{\bullet}_B &= \sum_{B''\sim B} (t \eta q_{B} + (1-t) \eta q_{B \setminus B''}),\\
r^{\bullet}_{B'} &= \sum_{B''\sim B'} (t \eta q_{B'} + (1-t) \eta q_{B' \setminus B''}),\\
r_{B\cap B'} &= t \eta q_{B} +2(1-t) \eta q_{B \cap B'} + t \eta q_{B'}\\
&~~~+\sum_{\underset{B'' \neq B'}{B'' \sim B}} (t\eta q_{B} + (1-t)\eta q_{B\setminus B''})
+\sum_{\underset{B'' \neq B}{B'' \sim B'}} (t\eta q_{B'} + (1-t)\eta q_{B'\setminus B''}),\\
r_{B \cap B''} &=  t \eta q_{B} +(1-t) \eta q_{B \cap B''}
+\sum_{\underset{B''' \neq B''}{B''' \sim B}} (t\eta q_{B} + (1-t)\eta q_{B\setminus B'''}), ~~~ \forall B'' \neq B',\\
r_{B' \cap B''} &=   t \eta q_{B'}+(1-t) \eta q_{B' \cap B''}
+\sum_{\underset{B''' \neq B''}{B''' \sim B'}} (t\eta q_{B'} + (1-t)\eta q_{B'\setminus B'''}), ~~~ \forall B'' \neq B,
\end{align*}
and `$s$' parameters defined similarly where the `$q$' variables are replaced by `$m$' variables. 
Now we can reduce the number of parameters in the above variational problems in the same way as for $f_{B}$ under Assumption~\ref{assump:lipschitz}: 
 \begin{corollary}
Under Assumption~\ref{assump:lipschitz}, there exists $\eta_0>0$ such that for $\eta \in \R_+ \setminus \mathcal{D}$ and $\eta\le \eta_0$ and for all $t \in [0,1]$ we have
\[\tilde{\phi}_2 = \phi_2 = \sup_{q \in [0,1]} \Big\{2(1-\alpha)\fsc(\delta;2 d \eta q) +\alpha \fsc(\delta;4 d \eta q) - d \eta q^2\Big\}.\]
Hence the two-block free energy has a limit 
\[\underset{\scL \to \infty}{\lim} f_{B\cup B'} = \phi_2.\]
 \end{corollary}
This concludes the proof of Proposition~\ref{prop:limits_phi1_phi2}.

\paragraph{Acknowledgements.}
The author is highly indebted to Andrea Montanari for numerous and invaluable discussions of the main ideas underlying this work.  


\begin{small}

\bibliographystyle{amsalpha}

\bibliography{all-bibliography.bib}
\addcontentsline{toc}{section}{References}
\end{small}

\end{document}